\documentclass[]{preprint}
\usepackage{times}% 
\usepackage{amsmath}
\usepackage{amssymb}% 
\usepackage{ScriptFonts}% 
\usepackage{mhequ}
\usepackage{mhenvs}
\usepackage{mhsymb}
\usepackage{mhfig}
\usepackage{microtype}

\usepackage{graphicx}
\usepackage{verbatim}

\newtheorem{innercustomthm}{Algorithm}
\newenvironment{algorithm}[1]
  {\renewcommand\theinnercustomthm{#1}\innercustomthm}
  {\endinnercustomthm}

\definecolor{darkred}{rgb}{0.9,0.1,0.1}

\def\K{\mathbf{K}}
\def\P{\mathbf{P}}
\def\Q{\mathbf{Q}}
\def\E{\mathbf{E}}
\def\TV{\mathrm{TV}}
\def\Lip{\mathrm{Lip}}

\def\l{\mathfrak{l}}
\def\s{\mathfrak{s}}
\def\e{\mathfrak{e}}
\def\v{\chi}

\def\cM{\mathcscr{M}}
\def\FF{\mathcscr{F}}
\def\GG{\mathcscr{G}}
\def\MM{\mathcscr{M}}
\def\QQ{\mathcscr{Q}}
\def\J{\mathcscr{J}}
\def\Dpe{\Delta^\eps}
\def\tJ{\,\,\,\tilde{\!\!\!\J}}
\def\QQt{\,\,\,\tilde{\!\!\!\QQ}^\eps}
\def\QQtd{\,\,\,\raisebox{0em}[0em]{$\tilde{\!\!\!\QQ}$}^{\eps}_\gamma}
\def\tJe{\,\,\tilde{\!\!\J^\eps}}

\def\eqlaw{\stackrel{\mbox{\tiny law}}{=}}

\begin{document} 

%  \date{December 21, 2007}
  \title{The Brownian fan}
  \author{Martin~Hairer\inst{1} and Jonathan~Weare\inst{2}}
  \institute{ Mathematics Department, the University of Warwick  \and Statistics Department and the James Frank Institute, the University of Chicago
   \\ \email{M.Hairer@Warwick.ac.uk}, \email{weare@uchicago.edu}}
  \titleindent=0.65cm

  \maketitle
  \thispagestyle{empty}

\begin{abstract}
We provide a mathematical study of the modified Diffusion Monte Carlo (DMC)
algorithm introduced in the companion article \cite{DMC}.  DMC is a simulation technique that uses branching particle systems to represent expectations associated with Feynman-Kac formulae.  
We provide a detailed heuristic explanation of why, in cases in which a stochastic integral appears in the Feynman-Kac formula (e.g. in rare event simulation, continuous time filtering, and other settings), 
the new algorithm is expected to converge in a suitable sense to a limiting process as 
the time interval between branching steps goes to $0$.  
The situation studied here stands in stark contrast to the ``na\"\i ve'' generalisation of the DMC
algorithm which would lead to an exponential explosion of the number of particles, thus precluding the
existence of any finite limiting object.
Convergence is shown rigorously in the simplest possible situation of a random walk, biased by a linear potential.
The resulting limiting object, which we call the ``Brownian fan'', is a very natural new mathematical object
of independent interest.

\end{abstract}


\setcounter{tocdepth}{2}
\tableofcontents

\section{Introduction}

Consider a Markov chain $y_{t_k}$ with transition probabilities $\P(x,dy)$ on some state
space $\CX$ and, in anticipation of developments below, indexed by a sequence of real numbers $t_0<t_1<t_2<\cdots.$ The nature of the state space does not matter (Polish is enough), but one should think
of $\CX = \R^d$ for definiteness. Given functions $\v \colon \CX^2 \to \R$ and $f\colon \CX \to \R$ 
with sufficient integrability properties (say bounded), 
the Diffusion Monte Carlo (DMC) algorithm (see \cite{DMC} for description of DMC compatible with the discussion below) computes an estimate of expectations of the form
\begin{equation}\label{discavet2}
\langle f \rangle_t =  \mathbf{E}\biggl( f(y_t) \exp\Bigl(- \sum_{t_k\leq t}\v(y_{t_k},y_{t_{k+1}}) \Bigr)\biggr)\;,
\end{equation}
where $y$ a realisation of the Markov chain described by $\P.$  Strictly for convenience we have assume that $t$ is among the times $t_0,t_1,t_2,\dots.$
More precisely, at time $t,$ DMC produces a collection of $N_{t}$ copies of the underlying system, $\{x_{t}^{(i)}\}_{i=1}^{N_{t}}$ so that 
\begin{equ}[e:unbiased]
\mathbf{E} \sum_{i=1}^{N_{t}} f(x_{t}^{(i)}) = \langle f \rangle_t.
\end{equ}
Of course, the expectation in \eqref{discavet2} can be computed by generating many independent sample trajectories of $y.$  However, in most cases the weights $\exp\bigl(- \sum_{t_k\leq t}v(y_{t_k},y_{t_{k+1}}) \bigr)$ will quickly degenerate so that a huge number of samples is required to generate a single statistically significant sample.  At each time step DMC removes samples with very low weights and replaces them with multiple copies of samples with larger weights, focusing computational effort on statistically significant samples.

Variants of DMC are used regularly in a wide range of fields including electronic structure, rare event simulation, and data assimilation (see \cite{HammersleySIS:1954, defreitas05,Allen:FFS:2006, Johansen:SMCrare:2006, DMC} for a small sample of these applications).  The algorithm has also been the subject of significant mathematical inquiry \cite{DelMoral:FK:2011}.  Continuous time limits have been considered \cite{Rousset:PhD:2006} for cases in which $\v$ scales like the time interval between branching steps. Perhaps because standard DMC does not have a limit in those cases, more general Feynman-Kac formulae do not seem to have been considered despite their appearance in applications.

For the purposes of this article, the main example one should keep in mind is that where $y$ is a time discretisation
of a diffusion process obtained for example by applying an Euler scheme with fixed stepsize $\eps$:
\begin{equ}[e:rareEvent]
y_{t_{k+1}} =y_{t_k} + \eps F(y_{t_k}) + \sqrt \eps \Sigma(y_{t_k}) \,\xi_{k+1}\;,
\end{equ}
where $\eps = t_{k+1}-t_k$ and the $\xi_k$ are a sequence of i.i.d.\ random variables (not necessarily Gaussian).
Regarding the function $\v$, we will mostly consider the case $\v(y,\tilde y) = V(\tilde y) - V(y)$.  As described in \cite{DMC}, this choice arrises in application of DMC to rare event simulation problems.  Notice that, for this choice of $\v$ and for $y$ in \eqref{e:rareEvent}, when $\eps$ is small, $\v(y_{t_k},y_{t_{k+1}})\sim \sqrt {\eps}.$  This causes a dramatic failure in standard DMC \cite{DMC}.  Yet, for small $\eps$ (with fixed $t$) the expectation $\langle f\rangle_t$ defined in \eqref{discavet2} has the perfectly nice limit
\[
 \mathbf{E}\biggl( f(y_t) \exp\Bigl(- V(y_t) + V(y_0) \Bigr)\biggr)
 \]
 where now $y_t$ is the continuous time limit of \eqref{e:rareEvent} (the diffusion with drift $F$ and diffusion $\Sigma$).  When a stochastic integral appears in the exponential the situation is completely analogous.  It is natural then to search for a modification of DMC that can handle these settings.

%One might argue that in this case no special algorithm is required, since this is essentially equivalent to
%the case $\v = 0$ by performing the substitution $f \mapsto f e^{-V}$.
%
%[.............]

It was shown in \cite{DMC} that the following algorithm provides an unbiased estimator for $\langle f \rangle_t$ (it satisfies \eqref{e:unbiased}) and is superior to DMC (it has lower variance and  equivalent expected cost).
\begin{algorithm}{TDMC}\label{tdmc}
Ticketed DMC
{\tt
\begin{enumerate}
\item Begin with $M$ copies $x^{(j)}_0 = x_0$.  For each $j=1,\dots,M$ choose an independent random variable
$\theta^{(j)}_0\sim \mathcal{U}(0,1)$.
\item At step $k$  there are $N_{t_k}$ samples $(x^{(j)}_{t_k},\theta^{(j)}_{t_k})$.  Evolve each of the $x^{(j)}_{t_k}$ one step to generate
$N_{t_k}$ values  \[\tilde x^{(j)}_{t_k}\sim \mathbf{P}\bigl( y_{t_{k+1}}\in dx\,|\, y_{t_k}=x^{(j)}_{t_k}\bigr).\]
\item For each $j=1,\dots, N_{t_k}$, let
\begin{equ}[e:defP]
P^{(j)} = e^{-\v(x^{(j)}_{t_k},\tilde x^{(j)}_{t_{k+1}})}.
\end{equ}
If $P^{(j)} < \theta^{(j)}_{t_k}$ then
set 
\[
N^{(j)} = 0.
\]
If $P^{(j)} \geq \theta^{(j)}_{t_k}$ then set
\begin{equ}[e:noff]
%N_j = \lfloor P_j - \xi_j \rfloor \vee 0\;,
N^{(j)} = \max\{\lfloor P^{(j)} + u^{(j)} \rfloor , 1\}\;,
\end{equ}
where $u^{(j)}$ are independent $\mathcal{U}(0,1)$ random variables.
\item For  $j=1,\dots, N_{t_k}$, if $N^{(j)}>0$ set 
\[
x^{(j,1)}_{t_{k+1}} = \tilde x^{(j)}_{t_{k+1}}\quad\text{and}\quad \theta^{(j,1)}_{t_{k+1}} = \frac{\theta^{(j)}_{t_k}}{P^{(j)}}
\]
and for $i=2,\dots,N^{(j)}$ 
\[
x^{(j,i)}_{t_{k+1}} = \tilde x^{(j)}_{t_{k+1}}\quad\text{and}\quad  \theta^{(j,i)}_{t_{k+1}}\sim  \mathcal{U}((P^{(j)})^{-1},1).  
\]
\item  Finally set $N_{t_{k+1}} = \sum_{j=1}^{N_{t_k}} N^{(j)}$ and list the $N_{t_{k+1}}$ vectors 
$\bigl\{x^{(j,i)}_{t_{k+1}}\bigr\}$  as $\bigl\{x^{(j)}_{t_{k+1}}\bigr\}_{j=1}^{N_{t_{k+1}}}$.
\item
At time $t$ produce the estimate
\[
 \widehat  f _t  = \frac{1}{M}\sum_{j=1}^{N_t} f(x^{(j)}_t)\;.
 \]
\end{enumerate}
}\end{algorithm}
%\begin{algorithm}{TDMC} \label{tdmc}
%Ticketed DMC
%\begin{enumerate}
%\item Begin with $M$ copies $x^{(j)}_0 = x_0$.  For each $j=1,\dots,M$ choose an independent random variable
%$\theta^{(j)}_0\sim \mathcal{U}(0,1)$.
%\item At step $k$  there are $N_{k}$ particles $(x^{(j)}_{k},\theta^{(j)}_{k})$.  Evolve each of the $x^{(j)}_{k}$ 
%by one step to generate
%$N_{k}$ new values  $\tilde x^{(j)}_{{k+1}}\sim \mathbf{P}\bigl(x^{(j)}_{k}, \,\cdot\,\bigr)$.
%\item For each $j=1,\dots, N_{k}$, let
%\begin{equ}[e:defP]
%P^{(j)} = e^{-\v(x^{(j)}_{k},\tilde x^{(j)}_{k+1})}.
%\end{equ}
%If $P^{(j)} < \theta^{(j)}_{k}$, 
%kill the $j$th particle by setting $N^{(j)} = 0$. Otherwise, set
%\begin{equ}[e:noff]
%N_j = \lfloor P_j - \xi_j \rfloor \vee 0\;,
%N^{(j)} = \max \bigl\{\lfloor P^{(j)} + u^{(j)} \rfloor , 1\bigr\}\;,
%\end{equ}
%where the $u^{(j)}$ are independent $\mathcal{U}(0,1)$ random variables.
%\item For  $j=1,\dots, N_{k}$, if $N^{(j)}>0$ set 
%$x^{(j,1)}_{k+1} = \tilde x^{(j)}_{k+1}$ and $\theta^{(j,1)}_{k+1} = \theta^{(j)}_{k} / P^{(j)}$.
%If $N^{(j)}>1$, set furthermore  
%\[
%x^{(j,i)}_{k+1} = \tilde x^{(j)}_{k+1}\quad\text{and}\quad  \theta^{(j,i)}_{k+1}\sim  \mathcal{U}((P^{(j)})^{-1},1).  
%\]
%for $i=2,\dots,N^{(j)}$. (These random variables are all independent.)
%\item  Finally set $N_{k+1} = \sum_{j=1}^{N_{k}} N^{(j)}$ and list the $N_{k+1}$ vectors 
%$\bigl\{x^{(j,i)}_{k+1}\bigr\}$  as $\bigl\{x^{(j)}_{k+1}\bigr\}_{j=1}^{N_{k+1}}$.
%\item
%At step $k$, produce the estimate
%$ \widehat  f _k  = \frac{1}{M}\sum_{j=1}^{k} f(x^{(j)}_k)$.
%\end{enumerate}
%\end{algorithm}

The aim of this article is to argue that one expects the application of 
Algorithm~\ref{tdmc} to \eref{e:rareEvent} to converge to a finite continuous-time limiting
particle process as $\eps \to 0$. Section~\ref{sec:heuristic} provides very detailed heuristic
as to why we expect this to be the case and gives a precise mathematical definition for the 
expected limiting object. 

In the special case where $d = 1$, $F = 0$, $\Sigma = 1$, and $V(x) = x$, the process in question 
 is built recursively by successive realisations of a Poisson point process in a space of excursions of $y_t$.  
A precise definition is given in  Definition~\ref{def:fan} and we call this object the \emph{Brownian fan}.
It is of particular interest since, similarly to the fact that every diffusion process looks locally like
a Brownian motion, one would expect the general limiting objects described in  Section~\ref{sec:heuristic} to
locally ``look like'' the Brownian fan.
The Brownian fan is interesting in its own right from a 
mathematical perspective and does not seem to have been studied before,
though it is very closely related to the ``Virgin island model'' introduced in \cite{VirginIsland}.

Loosely speaking, the Brownian fan is a branching process obtained in the following way. 
Start with a (usually finite) number of initial particles on $\R$ that furthermore come each with a 
tag $v\in \R$, such that the position $x$ satisfies $x > v$.
Each of these particles, which we call the ancestor particles, perform independent Brownian motions
until they hit the barrier $x=v$, where they get killed. Each of these particles independently produces 
offspring according to the following mechanism. Denoting by $\CE$ the space of excursions on $\R$,
let $\MM$ be the Poisson point process on $\R \times \CE$ with intensity measure
$a\,dt \otimes \Q$ for some $a>0$, where $\Q$ is It\^o's Brownian excursion measure \cite{ItoExc}. If $(t,w)$ is one
of the points of $\MM$ and the corresponding ancestor particle is alive at time $t$ and located at $x_t$, 
then it gives rise to an offspring which performs the motion described by $w$, translated so that
the origin of the excursion lies at $(t,x_t)$. This mechanism is then repeated recursively for each new particle
created in this way.

The Brownian fan has a number of very nice properties.  For example, as established by Theorem~\ref{theo:Nt}, the continuous-time analogue of the number of particles at time $t$, $N_t$, corresponding to the Brownian fan satisfies
\begin{equation}\label{expbndonN}
\sup_{t>0} \mathbf{E} \exp \left(\lambda_t N_t\right)< \infty
\end{equation}
 for some continuous decreasing function $\lambda_t>0$.  
As established in Proposition~\ref{prop:bfWmodcont}, the bound in \eqref{expbndonN} implies that the continuous-time analogue of the workload,
\[
\mathcal{W}_t = \int_0^t N_s ds
\]
is very nearly a differentiable function of time, 
\[
\sup_{t\leq T} \lim_{h\rightarrow 0} \frac{ \left| \mathcal{W}_{t+h} - \mathcal{W}_t\right|}{h \left| \log h\right|} < \infty.
\]
This is close to optimal, since there turns out to be a dense 
set of exceptional times at which there are infinitely many particles alive (see the argument
given at the end of Section~\ref{sec:numPart} below).

To conclude this round-up of its mathematical properties, we establish  in Proposition~\ref{prop:feller} 
that the Brownian fan is a Feller process in a suitable 
state space. When combined with the continuity of its sample paths established in 
Proposition~\ref{prop:Kolmogorov}, this implies that the Brownian fan is a strong Markov process.
Note that the construction of such a state space is a non-trivial endeavour due to the fact that, 
while the number $N_t$ of particles alive at time $t$ has finite moments of all orders, there exists
a \textit{dense} set of exceptional times for which $N_t = \infty$! Offspring are continuously being created at infinite rate, thus making the Brownian fan quite
different from a standard branching diffusion.

The ``meat'' of this article, Sections~\ref{sec:tightness} and \ref{sec:convFan}, is then
devoted to a rigorous proof of the convergence
of the output of Algorithm~\ref{tdmc} to the Brownian fan for any sufficiently light-tailed
one-step distribution for the random variables $\xi_k$. This result is formulated in Theorem~\ref{theo:finalConv}.
The usual method by which one attempts to characterise the continuous-time limit of a sequence of discrete-time 
processes involves studying the limit of the generators of the discrete-time processes.  In our case, 
as described in Section \ref{sec:noConv}, the discrete-time generators  do \textit{not} converge to the correct limit, at
least when applied to a class of very natural-looking test functions.
Surprisingly (and rather confusingly), they do actually converge
to the generator  of a Brownian fan, but unfortunately with the wrong parameters!  
 En route to our convergence result we 
establish a number of important and extremely encouraging results about the behaviour of the process generated 
by Algorithm~\ref{tdmc} for finite $\eps$.  For example, in Proposition~\ref{prop:numPart} we obtain that 
\[
\sup_{\eps < \eps_0} \mathbf{E} \left| N_{t}\right|^p < \infty\;,
\]
for any $p\geq 0$ and $\eps_0>0$ sufficiently small where here $N_t$ refers to the number of particles generated in Algorithm~\ref{tdmc} and not to the Brownian fan (for which we have the even stronger result in \eqref{expbndonN}).  In Corollary~\ref{cor:Kolmogorov} we also prove a uniform (in $\eps$) form of continuity of the processes.

%
%In addition, by focusing on a particular asymptotic regime (small time-discretisation parameter) 
%we are also able to rigorously establish several dramatic results concerning the stability of 
%Algorithm~\ref{tdmc} in settings in which the straightforward generalisation of DMC is unstable.  
%In Section~\ref{sec:contLimit}, we take a closer look at the situation \eref{e:rareEvent} used for
%rare event simulation. We provide a heuristic explanation why one would expect our algorithm to
%converge to a limiting process as the discretisation parameter converges to $0$ and we give a
%non-rigorous characterisation of the limiting object which we call the ``Brownian fan.'' We also provide a detailed study of the fine properties
%of this limiting process in the simplest case when $y_t$ is a random walk (rescaled so that it converges to 
%Brownian motion) and the biasing potential $V$ is linear.
%The Brownian fan is interesting in its own right from a 
%mathematical perspective and does not seem to have been studied before,
%though it is very closely related to the ``Virgin island model'' introduced in \cite{VirginIsland}.
%

\subsection{Notations}

For any Polish space $\CY$, we will denote by $\MM_+(\CY)$ the space of finite positive measures 
on $\CY$, endowed with the topology of weak convergence, together with the convergence of total mass.
Given a random variable $X$, we denote its law by $\CD(X)$, except in some cases where we 
introduce dedicated notations.

We will often use the notation $a \lesssim b$ as a shorthand for the inequality $a \le C b$ for some constant $C$.
The dependence of $C$ on other quantities will usually be clear from the context, and will be indicated when 
ambiguities may arise. We will also use the standard notations $a \wedge b = \min \{a,b\},$ $a \vee b = \max\{a,b\},$ and $\lfloor a\rfloor = \max\{ i\in \mathbb{Z}:\,i\leq a\}$.

\subsection*{Acknowledgements}

{\small
We would like to thank H.~Weber for numerous discussions on this article and S.~Vollmer for
pointing out a number of inaccuracies found in earlier drafts.  We would also like to thank J.~Goodman and E.~Vanden-Eijnden for helpful conversations and advice during the early stages of this work.
Financial support for MH was kindly provided by EPSRC grant EP/D071593/1, by the 
Royal Society through a Wolfson Research Merit Award and by the Leverhulme Trust through a Philip Leverhulme Prize.  JW was supported by NSF through award DMS-1109731.
}

\section{The continuous-time limit and the Brownian fan}
\label{sec:contLimit}

From now on, we restrict ourselves to the analysis of the case $\v(x,y) = V(y) - V(x)$ for
some ``potential'' $V$ defined on the state space of the underlying Markov process.
We argue that if the underlying process is obtained by approximating a diffusion process then, 
unlike in the case of the na\"\i ve generalisation of DMC (see \cite[Algorithm~DMC]{DMC}), 
our modification, Algorithm~\ref{tdmc}, converges to a 
non-trivial limiting process as the stepsize $\eps$ converges to $0$.

We will first provide a heuristic argument showing what kind of limiting process one would
expect to obtain. The remainder of this article will then be devoted to rigorously constructing
the limiting process and proving convergence in the simple case in which the underlying Markov chain
is a random walk (rescaled so that it converges to a standard Brownian motion) and the biasing
potential $V$ is linear. In this case the limiting process is a very natural object that does not seem
to have been studied in the literature so far. We call this object, which is closely related to the construction in 
\cite{VirginIsland}, the \textit{Brownian fan} (see Section~\ref{sec:hairy}).
It also has a flavour very similar to the construction of the Brownian web \cite{BWeb} and the
Brownian net \cite{MR2408586}, although 
there does not seem to be an obvious transformation linking these objects.

\subsection{Heuristic derivation of the continuous-time limit}
\label{sec:heuristic}

Throughout this section the underlying Markov chain will be given just like in the introduction by 
the following approximation to a diffusion:
\begin{equ}[e:stepEuler]
y_{(k+1)\eps} =y_{k\eps} + \eps F(y_{k\eps}) + \sqrt \eps \Sigma(y_{k\eps}) \,\xi_{k+1}\;,\qquad y_{k\eps} \in \R^n\;,
\end{equ}
where the $\xi_k$ are a sequence of i.i.d.\ (not necessarily Gaussian) random variables with law $\nu$
and the identity on $\R^n$ as their
covariance matrix.
The functions $F$ and $\Sigma$ are sufficiently ``nice'' functions, but since this section is only heuristic, we do
not state specific regularity, growth or non-degeneracy assumptions.

\begin{remark}
We have slightly changed our notations by writing $y_{k\eps}$ instead of $y_k$ for the position of the Markov chain
after $k$ steps. This is in order to make explicit the fact that as $\eps \to 0$, one has convergence to a 
continuous-time process. The corresponding notational changes in Algorithm~\ref{tdmc} are straightforward.
\end{remark}

Concerning the function $\v$, we take $\v(x,y) = V(y) - V(x)$ for some regular potential 
$V \colon \R^n \to \R$. Recall now that, as long as a particle is alive, its ticket $\theta$ evolves under
Algorithm~\ref{tdmc} as 
\begin{equ}
\theta^{(j)}_{k\eps} = \theta^{(j)}_{k\eps} \exp \bigl(V(\tilde x^{(j)}_{(k+1)\eps})-V(x^{(j)}_{k\eps})\bigr)\;.
\end{equ}
It is natural therefore to replace $\theta$ by the quantity $v$ given by
\begin{equ}
\exp\bigl(-v^{(j)}_{t}\bigr) =  \theta^{(j)}_{t}\exp \bigl(-V(x^{(j)}_{t}) \bigr)\;.
\end{equ}
In this way, the new ``tag'' $v$ does not change over time, but is assigned to a particle at the moment
of its birth. Translating Steps~3 and 4 of the algorithm into this slightly different setting, we see that
if a particle performs a step from $x$ to $y$ such that $V(y) < V(x)$,
then it can potentially spawn one or more descendants. The tags $v$ of the descendants are then 
distributed according to
\begin{equ}[e:lawv]
e^{-v} \eqlaw \CU\bigl(e^{-V(x)}, e^{-V(y)}\bigr)\;,
\end{equ}
and a particle with tag $v$ lives as long as it stays within the region $\{x\,:\,  V(x)\le v\}$.

\subsubsection{Description of the limit}
\label{sec:limit}

For very small values of $\eps$, the process described above has the following features.
Taking the limit $\eps \to 0$ in \eref{e:stepEuler}, we observe that each particle follows
a diffusion process, solving the equation
\begin{equ}[e:diff]
dy_t = F(y_t)\,dt + \Sigma(y_t)\,dB_t\;,
\end{equ}
where $B_t$ is a standard $d$-dimensional Brownian motion. If the particle has tag $v$, then this 
process is killed as soon as it exits the sublevel set $\{x\,:\, V(x) \le v\}$.

Consider the following representation of the object produced  by Algorithm \ref{tdmc}.  Denote by $\Q^\eps_{x,v}$
the law of the $\eps$-discretization of \eqref{e:diff} generated by \eqref{e:stepEuler} starting at $x$ and killed  upon exiting the set $\{y\,:\, V(y) \le v\}$.  Let $\tau$ and  $\{w_{k\eps}\}_{k=0}^{\tau/\eps}$ be, respectively, the lifetime and trajectory of the original particle.  The trajectories of the offspring of this initial trajectory are very nearly given (it will be true in the small $\eps$ limit) by a realisation $\mu^{\eps,1}$ of a Poisson point process with intensity
 \begin{equ}[e:discintMeasure]
\CQ^\eps(w,\cdot) =  \sum_{k=0}^{\tau/\eps-1} A^\eps (k)\, \Theta_{k\eps}^\star \int \sqrt{\eps}\Q^\eps_{w_{k\eps},V(w_{k\eps})+\delta}\,\eta(k,d\delta)\;,
\end{equ}
where, according to the rule for generating new offspring in Algorithm \ref{tdmc},
\[
A^\eps (k) = \begin{cases} \frac{1}{\sqrt{\eps}}\bigl(e^{-(V(w_{(k+1)\eps})-V(w_{k\eps}))}-1\bigr) & \text{if } V(w_{(k+1)\eps})<V(w_{k\eps})\\
0 & \text{if }V(w_{(k+1)\eps})\geq V(w_{k\eps})
\end{cases}
\]
and, according to the rule for generating offspring tickets in Algorithm \ref{tdmc},
\[
\int f(\delta)\eta(k,d\delta) =\int_0^1 f\bigl( -\log\bigl( 1 + u(e^{-(V(w_{(k+1)\eps})-V(w_{k\eps}))}-1)\bigr)\bigr)\,du.
\]
% \exp\bigl(V(w_{(k+1)\eps})-V(w_{k\eps})\bigr)
Here $\Theta_t$
is the map that shifts trajectories forward by time $t$.
Since each offspring behaves independently just like the original particle, this suggests
that the $n$th generation $\mu^{\eps,n}$ of offspring is obtained recursively 
as a realisation of the Poisson point process  with intensity
given by
\begin{equ}
\CG^{\eps,n}(\cdot) = \int \CQ^\eps(\tilde w,\cdot)\,\mu^{\eps,{n-1}}(d\tilde w)\;.
\end{equ}

At each ``microscopic'' step, the probability of creating a descendent is of order $\sqrt \eps$ so that,
in the limit $\eps \to 0$, each particle spawns  descendants at infinite rate. However, 
any such descendant is created at distance $\CO(\sqrt \eps)$ of the ``barrier'' $V(x) = v$. 
As a consequence, the probability that it survives for a time of order $1$ before being killed
is itself only of order $\sqrt \eps$. Therefore, the rate at which a particle creates descendants
that actually survive for a time $\tau$ of order $1$ is finite, but tends to infinity as $\tau \to 0$.
 
Now we will consider the small $\eps$-limit of the object we have constructed.  The trajectory $w_t$ becomes  a sample path of \eqref{e:diff} exiting the set $\{y\,:\, V(y) \le v\}$  at time $\tau$.  Denote by $\Q_{x,v}$ the law of the diffusion
\eref{e:diff} starting at $x$ and killed upon exiting the set $\{y\,:\, V(y) \le v\}$, which is a probability measure
on some space of excursions in $\R^n$.
The characterisation of the
standard It\^o excursion measure (see for example \cite[Theorem~4.1]{RevYor} and \cite{PitmanYor}) 
then suggests that, for every $x \in \R^n$ such that
$\nabla V \neq 0$ and $\Sigma$ is non-degenerate, the limit
\begin{equ}
\Q_x = \lim_{\delta\to 0^+} {1\over \delta} \Q_{x,V(x)+\delta}\;,
\end{equ}
exists as a $\sigma$-finite measure in the sense that  $ {1\over \delta} \Q_{x,V(x)+\delta}$  restricted to
the set of 
excursions longer than a fixed length 
converges weakly to $\Q_x$ restricted to the same set.  

The  discussion so far suggests that for the limiting object, the trajectories of the first generation of
offspring are given
by a realisation $\mu^1$ of the Poisson point process with intensity measure
\begin{equ}[e:intMeasure]
\CQ(w,\cdot) = \int_0^\tau A(w_t)\, \Theta_t^\star \Q_{w_t}\,dt\;,
\end{equ}
for some intensity $A\colon \R^n \to \R_+$ yet to be determined and
that the $n$th generation $\mu^n$ of offspring is obtained recursively 
as a realisation of a Poisson point process with intensity
given by
\begin{equ}
\CG^n(\cdot) = \int \CQ(\tilde w,\cdot)\,\mu^{n-1}(d\tilde w)\;,
\end{equ}
with $\CQ$ as in \eref{e:intMeasure}. 
In order to fully characterise the limiting object, it remains to provide an expression
for the intensity function $A$.

Let us start by replacing $\Q^\eps_{x,v}$ in equation \eqref{e:discintMeasure} by $\Q_{x,v},$ i.e. by assuming that for small $\eps$, excursions of the discrete process are very similar to excursions of its continuous time limit.
We then apply the  relations
\begin{equation}\label{Papprox1}
\Q_{x,V(x)+\delta} \approx \delta \Q_x\;
\end{equation}
and
\[
V(w_{(k+1)\eps}) - V(w_{k\eps}) \approx \sqrt{\eps}\scal{\nabla V(w_{k\eps}),\Sigma(w_{k\eps}) \xi_{k+1}}
\]
with the $\xi_{k+1}$ as in \eref{e:stepEuler}.
We then formally obtain
\begin{multline*}
\CQ^\eps(w,\cdot) \approx \eps \sum_{k=0}^{\tau/\eps-1} \mathbf{1}_{\scal{\nabla V(x),\Sigma(x) \xi_{k+1}} < 0}\scal{\nabla V(w_{k\eps}),\Sigma(w_{k\eps}) \xi_{k+1}}\\
\times
\int_0^1 u  \scal{\nabla V(w_{k\eps}),\Sigma(w_{k\eps}) \xi_{k+1}} du\,
 \Theta_{k\eps}^\star \Q_{w_{k\eps}}.
\end{multline*}
Our arguments so far therefore suggest that 
\begin{equ}[e:wrongGuess]
A(x) = \frac{1}{2} \int_{\scal{\nabla V(x),\Sigma(x) z} < 0}  \langle \nabla V(x), \Sigma(x) z \rangle^2\, \nu(dz)\;,
\end{equ}
where the distribution $\nu$ has mean 0 and identity covariance matrix.   Assuming that $\nu$ is symmetric this becomes
\begin{equ}[e:intensity]
A(x) = c\scal{\nabla V(x),\Sigma(x) \Sigma^T(x)\nabla V(x)}\;,
\end{equ}
where $c= {1\over 4}$. If $\nu$ is not symmetric, one might even expect a prefactor $c$ that depends on $x$.

In fact, as we will see in a specific case in the remainder of this section, the correct value is $c = {1\over 2}$, whether
the law of $\xi$ is symmetric or not. The reason for this discrepancy is that the relation
\[
\Q^\eps_{x,V(x)+\delta} \approx \delta \Q_x\;
\]
used in our derivation
 is only valid if
$\delta \gg \sqrt \eps$. In our case however, one precisely has $\delta \sim \sqrt \eps$, which 
introduces a correction factor that eventually gives rise to the value $c = {1\over 2}$. 
The aim of the next subsection is to show in more detail how this factor ${1\over 2}$ arises in the
simplest situation where $F=0$ and $\Sigma = 1$.

\subsubsection{The case of Brownian motion}
\label{sec:BMformal}

We now consider the one-dimensional case, where the limiting underlying process is simple Brownian motion.
Regarding the underlying discrete problem, we consider the Markov chain defined recursively by 
\begin{equ}[e:discrete]
y_{(k+1)\eps} = y_{k\eps} + \sqrt \eps \xi_{k+1}\;,
\end{equ}
for an i.i.d.\ sequence of centred random variables $\xi_k$ with law $\nu$ and variance $1$. For the potential
function $V$, we choose $V(x) = -ax$ for some $a > 0$.

In order to show that the constant $c$ appearing in \eref{e:intensity} 
is equal to ${1\over 2}$, we will now argue that
if we denote by $\Q$ the standard It\^o excursion measure (which we normalise in such a way that
 $\Q = \lim_{\eps \to 0} {1\over \eps} \Q_\eps$, where $\Q_\eps$ is the law of a 
standard Brownian motion starting at $\eps$ and killed when it hits the origin)
and by $\Q^\eps_z$ the law of the random 
walk \eref{e:discrete} starting at $\sqrt \eps z$ and stopped as soon as it takes negative values,
then there exists a function $G$ such that
\begin{equ}[e:fcnG]
 \Q^\eps_z \approx \sqrt{\eps}\,G(z) \Q\;,
\end{equ}
as $\eps\to 0$ when both sides are restricted to excursions that survive for at least  some fixed amount of time.
We will see that the function $G$ behaves like $G(z) \approx z$ for large values of $z$, but has a non-trivial behaviour
for values of order $1$.
In terms of our notation from the previous subsection (since $y_t$ is  spatially homogeneous and $V$ is linear) this implies
that the approximation \eref{Papprox1} should really have been replaced by
\begin{equ}
 \Q^\eps_{x,V(x)+\delta} \approx a \sqrt{\eps}\,G\left(\frac{\delta}{a\sqrt{\eps}}\right) \Q\;.
\end{equ}

Since we assumed $a > 0$, our process creates offspring only when it performs a step
towards the right, i.e.\ when $\xi_{k+1} > 0$.  The probability that a new particle is created in the $k$-th step is approximately $a\xi_{k+1}$. Furthermore, the small $\eps$ rule for the generation of tags implied by
Algorithm \ref{tdmc} is
\begin{equ}[e:unifv]
\frac{\delta}{a\sqrt{\eps}} = \frac{a^{-1}v+x}{\sqrt\eps}  \sim \CU \bigl(0, \xi_{k+1} \bigr).
\end{equ}
 As a consequence, once we have identified the function $G$ in
\eref{e:fcnG}, our arguments in the previous section lead to the formula
\[
A(x) = \int_0^\infty (a z) \biggl( \frac{1}{z} \int_0^z a G(y)\,dy\biggr)\,\nu(dz)
\]
where $\nu$ is the law of the steps $\xi_{k}$.
If we can  show that
\begin{equ}[e:idenG]
\int_0^\infty \int_0^z G(y)\,dy\,\nu(dz) = {1\over 2}\;,
\end{equ}
then we will have
\[
A(x) = \frac{a^2}{2}\;,
\]
a formula consistent with a choice of $c=\frac{1}{2}$ in \eqref{e:intensity}.

In order to identify $G$, we note that if a random walk starting from $\sqrt \eps \delta$ 
survives for some time of order $1$ before becoming negative then, with overwhelming probability,
it will have reached a height of at least $\eps^{1/4}$ (say). Furthermore, if we condition the
random walk $\Q^\eps_z$ to reach a level $\sqrt \eps \gamma$ with $1 \ll \gamma \ll \eps^{-1/2}$,
 one would expect its 
law to be well approximated by $\sqrt \eps \gamma \Q$ when restricted to excursions that 
survive for a time of order $1$. 
%This is because It\^o's excursion measure is 
%obtained as the limit $\Q = \lim_{\delta \to 0} \delta^{-1} \Q_\delta$, where $\Q_\delta$ is the law
%of a Brownian motion starting at $\delta$ and stopped when it hits the origin.

As a consequence, we expect that
\begin{equ}
\Q^\eps_z \approx \bar P_{z,\gamma} \sqrt \eps \gamma \Q\;,  \qquad \gamma \gg 1\;,
\end{equ}
where $\bar P_{z,\gamma}$ denotes the probability that the simple random walk \eref{e:discrete} with $\eps = 1$
started at $z$ reaches the level $\gamma$ before becoming negative.

The remainder of this section is devoted to the proof of the fact that if we define $\bar P_{z,\gamma}$ in this way,
then under some integrability assumptions for the one-step probability $\nu$, the limit
\begin{equ}
G(z) = \lim_{\gamma \to \infty} \gamma \bar P_{z,\gamma} \;,
\end{equ}
exists and does indeed satisfy \eref{e:idenG}, independently of the choice of $\nu$. Actually, we will prove
these statements for the quantity $P_{z,\gamma} = \bar P_{z, \gamma+z}$, which we interpret as the probability
that the random walk starting at the origin reaches $[\gamma,\infty)$ before reaching $(-\infty,-z]$.
Our first result is as follows:

\begin{proposition}\label{prop:defG}
Assume that the law $\nu$ satisfies $\nu(\{|x| \ge K\}) \le C \exp(- c K^\beta)$ for all $K \ge 0$ and
some strictly positive
constants $c$, $C$ and $\beta$.
Then, the limit 
\begin{equ}[e:defG]
G(s) = \lim_{\gamma \to \infty} \gamma P_{s,\gamma}
\end{equ}
exists and satisfies the relations
\begin{equ}[e:relG]
G(s) = \int_{-s}^\infty G(s+z) \,\nu(dz)\;,\quad s \ge 0\;,\qquad \lim_{s \to \infty} {G(s) \over s} = 1\;.
\end{equ}
Furthermore, for every $\delta > 0$ there exists $C$ such that the bound
\begin{equ}[e:quantitative]
\bigl|(\gamma+s)P_{s,\gamma} - G(s)\bigr| \le C {1+s \over \gamma^{{1\over 2}-\delta}}\;,
\end{equ}
holds uniformly for all $s \ge 0$ and $\gamma \ge 1\vee s$.
\end{proposition}
\begin{proof}
Denote by $y_k$ the $k$th step of the random walk starting at the origin.
Our main tool is the quantitative convergence result \cite{ConvRateBM}, 
which states that the supremum distance between a Wiener process and 
the diffusively rescaled random walk over $n$ steps is of order $n^{-1/4}$.

As a consequence we claim first that, for every $\delta > 0$ there exists a constant $C$ such that, for every $a \in [{1\over 3},3]$, we have the bound
\begin{equ}[e:boundPkk]
\Bigl|P_{a\gamma,\gamma} - {a\over 1+a}\Bigr| \le {C \over \gamma^{{1\over 2}-\delta}}\;,
\end{equ}
valid for every $\gamma \ge 1$. Indeed, for any $n \ge 1$, it follows from the previously quoted convergence result that
there exists a Brownian motion $B$ such that $|B_t - y_{\lfloor t \rfloor}| \le n^{1/4+\delta}$ for all $t \in [0,n]$
with probability greater than $1 - C/n^{q}$. Here, $\delta > 0$ and $q \ge 1$ are arbitrary, but the constant $C$ of course 
depends on them.

Take $n$ such that $n^{1/4+\delta} \le \gamma$.
If $y$ hits $[\gamma,\infty)$ before $(-\infty,-a\gamma]$, then either $\sup_{t \le n} |B_t - y_{\lfloor t \rfloor}| > n^{1/4+\delta}$,
or $\sup_{t \le n} |B_t| \le 3\gamma$, or
$B$ hits $[\gamma-n^{1/4+\delta},\infty)$ before it hits $(-\infty,-a\gamma - n^{1/4+\delta}]$. As a consequence, 
\begin{equ}
P_{a\gamma,\gamma} \le {a\gamma + n^{1/4+\delta} \over (1+a)\gamma} + {C\over n^{q}} + \exp(- c n/\gamma^2) \;.
\end{equ}
Reversing the roles of $\gamma$ and $a\gamma$, we thus obtain the bound
\begin{equ}
\Bigl|P_{a\gamma,\gamma} - {a\over 1+a}\Bigr| \lesssim {n^{1/4+\delta} \over \gamma} + {1\over n^{q}} + \exp(- c n/\gamma^2) \;.
\end{equ}
Choosing $\delta$ small enough and $n = \gamma^{2+\delta}$, the claim \eref{e:boundPkk} then follows.

In order to obtain the convergence of the right hand side in \eref{e:defG}, we make use of the fact that,
for $\bar \gamma > \gamma$, one has the identity
\begin{equ}[e:idenPk]
P_{s,\bar\gamma} = P_{s,\gamma} \int_0^\infty P_{s+\gamma + z,\bar \gamma - \gamma - z }\,\nu_\gamma(dz)\;,
\end{equ}
where $\nu_\gamma$ is the law of the ``overshoot'' $y_n - \gamma$ at the first time $n$ such that 
$y_n \ge \gamma$, conditioned on never reaching below the level $-s$. 
Since $P_{ s+\gamma + z,\bar \gamma - \gamma - z}$ is an increasing function of $z$,
we immediately obtain the lower bound
\begin{equ}
P_{s,\bar\gamma} \ge P_{s,\gamma} P_{s+\gamma,\bar \gamma - \gamma }\;,
\end{equ} 
If we choose $\gamma = a \bar \gamma$ for $a \in [{1\over 4},{1\over 2}]$, it then follows from \eref{e:boundPkk} that
\begin{equ}
P_{s,\bar\gamma} \ge  P_{s,\gamma} \Bigl({\gamma + s \over \bar \gamma + s} - {C\over  \gamma^{{1\over 2}-\delta}}\Bigr) \;,
\end{equ} 
for all $\gamma$ sufficiently large and uniformly over all $s \in [0,\gamma]$. Setting $Q_{s,\gamma} = (\gamma + s)P_{s,\gamma}$, it thus follows that one has the bound
\begin{equ}[e:boundF]
Q_{s,\bar\gamma} \ge  Q_{s,\gamma} \Bigl(1 - {C\over  \gamma^{{1\over 2}-\delta}}\Bigr)\;,
\end{equ}
possibly for a different constant $C$. Let $\gamma_0 \ge 1$ be such that the factor on the right of this equation 
is greater than $1/2$.
By \eref{e:boundPkk}, there then exists $s_0\ge \gamma_0$ such that for $s \ge s_0$ and 
$\gamma \in [s,2s]$ one has $Q_{s,\gamma} \ge C(1+s)$. Furthermore, for $s \le s_0$ and 
$\gamma \in [(1\vee s), 2(\gamma_0\vee s)]$, there exists a non-zero constant such that $Q_{s,\gamma} \ge C$.
Iterating \eref{e:boundF}, we then conclude that there exists a constant $C>0$ such that the bound
\begin{equ}
Q_{s,\gamma} \ge C(1+s)\;,
\end{equ}
holds uniformly over all $s > 0$ and all $\gamma \ge 1 \vee s$.

On the other hand, for arbitrary $\alpha>0$, one has from \eref{e:idenPk} the lower bound
\begin{equ}
P_{s,\bar\gamma} \le \bigl(\nu_\gamma(\{x > \gamma^\alpha\}) + P_{s+\gamma + \gamma^\alpha,\bar \gamma - \gamma - \gamma^\alpha }\bigr)P_{s,\gamma} \;.
\end{equ}
In order to bound $\nu_\gamma(\{x > \gamma^\alpha\})$, we note that 
this event can happen only if either one of the first $\gamma^3$ increments exceeds $\gamma^\alpha$,
or the random walk never exceeds the value $\gamma$ within these $\gamma^3$ steps.
Similarly to before, it then follows that
\begin{equ}
\nu_\gamma(\{x > \gamma^\alpha\}) \lesssim {1\over P_{s,\gamma}} \bigl(\gamma^3 \exp(-c \gamma^{\alpha\beta}\bigr)
+ \gamma^{-q} + \exp(-c \gamma) \bigr)\;,
\end{equ}
for every $q > 0$ and uniformly over $s \le \gamma$. 
It follows from the lower bound on $P_{s,\gamma}$ obtained previously that 
$\nu_\gamma(\{x > \gamma^\alpha\}) \lesssim \gamma^{-q}$ for any power $q > 0$, so that
we obtain the upper bound 
\begin{equ}[e:upBound]
P_{s,\bar\gamma} \le  P_{s,\gamma} \Bigl({\gamma+s \over \bar \gamma+s} + {C\over  \gamma^{{1\over 2}-\delta}}\Bigr) \;,
\end{equ}
with the same domain of validity as before.
Using a very similar argument as before, we obtain a constant $\bar C$ such that 
$\gamma P_{s,\gamma} \le \bar C (1+s)$ uniformly over $s > 0$ and $\gamma \ge 1\vee s$.

Combining the bounds we just obtained, we obtain
\begin{equ}
\big|\bar \gamma P_{s,\bar\gamma} - \gamma P_{s,\gamma}\big| \le {C \over \gamma^{{1\over 2} + \delta}}\;,
\end{equ}
uniformly over $\bar \gamma > \gamma > s$, from which it follows immediately that 
the sequence $\{\gamma P_{s,\gamma}\}_{\gamma \ge 1}$ is Cauchy,
so that it has a limit $G(s)$.

It remains to show that $G$ has the desired properties. The first one follows immediately from the
identity
\begin{equ}
\gamma P_{s,\gamma} = \int_\R \gamma P_{ s+z,\gamma-z}\,\nu(dz)\;,
\end{equ}
which holds provided that we define the integrand to be $1$ for $z > \gamma$ and $0$ for $z < -s$.

In order to show that $G(s) / s \to 1$, we fix some (large) value $s$ and choose $\gamma_n = 2^n s$.
It then follows from \eref{e:boundPkk} that
\begin{equ}
|Q_0 - s| \lesssim s^{{1\over 2}+\delta}\;,
\end{equ}
where we used the notation $Q_n = (\gamma_n + s)P_{s,\gamma_n}$ as a shorthand. 
Furthermore, it follows immediately from \eref{e:upBound} that there exists a constant $C$
independent of $s$ such that $|Q_n| \le C s$ uniformly in $n$.
As a consequence, we obtain the recursive bound
\begin{equ}
\bigl|Q_n-Q_{n-1}\bigr| \le {C s \over \gamma_n^{{1\over 2}-\delta}}\;.
\end{equ}
Summing over $n$ yields $|Q_n - s| \lesssim s^{{1\over 2}+\delta}$, uniformly in $n$, so that the claim follows.
The quantitative error bound \eref{e:quantitative} follows in the same way.
\end{proof}

\begin{corollary}\label{cor:boundRW}
In the same setting as above, one has the bound
\begin{equ}
\Bigl|P_{s,\gamma} - {s \over \gamma+s}\Bigr|\lesssim {1+s^{{1\over 2}+\delta} \over \gamma}\;,
\end{equ}
uniformly for all $s \ge 0$ and $\gamma \ge 1\vee s$.
\end{corollary}

\begin{proof}
Combine \eref{e:quantitative} with the bounds on $G(s) - s$ obtained at the end of the proof above.
\end{proof}

%The proof is somewhat lengthy but does not provide much new insight, so we postpone it to Appendix~\ref{app:boundTP}.
Somewhat surprising is the fact that the function $G$ obtained in the Proposition~\ref{prop:defG} does indeed satisfy \eref{e:idenG}, independently
of the choice of transition probability $\nu$,
provided that we assume that $\nu$ has some exponential moment.

\begin{proposition}\label{prop:expG}
Let $G$ be as in Proposition~\ref{prop:defG} and assume that the law $\nu$ satisfies 
\begin{equ}
\int_\R e^{c|z|}\nu(dz) < \infty\;,
\end{equ}
for some $c > 0$. Then, one has the identity $\int_0^\infty \nu([s,\infty))\,G(s)\,ds = {1\over 2}$.
\end{proposition}

\begin{remark}
Note that, by Fubini's theorem,
\begin{equ}
\int_0^\infty \nu([s,\infty))\,G(s)\,ds = \int_0^\infty \int_0^s G(y)\,dy\,\nu(ds)\;,
\end{equ}
so that we do obtain \eref{e:idenG}.
\end{remark}

\begin{proof}
Integrating \eref{e:relG} from $0$ to an arbitrary value $K > 0$ and applying Fubini's theorem, we obtain the identity
\begin{equ}[e:intG]
\int_0^K G(s)\,ds = \int_0^\infty G(z)\, \nu([z-K,z])\,dz\;.
\end{equ}

In this proof we denote by $\CI = \int_0^\infty G(z)\nu([z,\infty))\,dz$ the quantity of interest.
Simple algebraic manipulations then yield from \eref{e:intG}
\begin{equs}
\CI &= \int_0^\infty G(z)\,\nu([z-K,\infty))\,dz -  \int_0^\infty G(z)\,\nu([z-K,z])\,dz\\
 &= \int_0^\infty G(z)\,\nu([z-K,\infty))\,dz -  \int_0^K G(z)\,dz\\
 &= \int_0^\infty G(z) \bigl(\nu([z-K,\infty)) - \one_{z < K}\bigr)\,dz\;.
\end{equs}
Since this identity holds for every $K>0$, it follows in particular that one has
\begin{equs}
\CI &= \eps \int_0^\infty G(z) \int_0^\infty \bigl(\nu([z-K,\infty)) - \one_{z < K}\bigr) e^{-\eps K}\,dK\,dz\\
&= \int_0^\infty G(z) \Bigl(\eps \int_0^\infty \nu([z-K,\infty)) e^{-\eps K}\,dK - e^{-\eps z}\Bigr)\,dz\;,\label{e:boundII}
\end{equs}
for every $\eps > 0$. At this stage, we note that one has the identity
\begin{equ}
\eps \int_0^\infty \nu([z-K,\infty)) e^{-\eps K}\,dK = e^{-\eps z} \E e^{\eps \xi} - \eps e^{-\eps z} \int_{z}^{\infty} e^{\eps K} \nu([K,\infty))\,dK\;,
\end{equ}
where $\xi$ denotes an arbitrary random variable with law $\nu$.
Since $\nu$ has some exponential moment by assumption, $\nu([K,\infty))$ decays exponentially so that 
the second term in this identity satisfies
\begin{equ}
\Bigl|\eps e^{-\eps z} \int_{z}^{\infty} e^{\eps K} \nu([K,\infty))\,dK\Bigr| \le C \eps e^{-\gamma z}\;,
\end{equ}
for some constants $\gamma, C>0$, provided that $\eps$ is small enough. 
Inserting this into \eref{e:boundII}, it follows that
\begin{equ}
\CI = \int_0^\infty G(z) e^{-\eps z} \E \bigl(e^{\eps \xi} - 1\bigr)\,dz  + \CO(\eps)\;.
\end{equ}
At this stage, we use again the fact that $\xi$ has exponential moments to deduce that
\begin{equ}
\E e^{\eps \xi} - 1 = {\eps^2 \over 2} + \CO(\eps^3)\;,
\end{equ}
where we used the fact that $\E \xi = 0$ and $\E \xi^2 = 1$, so that
\begin{equ}
\CI = {\eps^2 \over 2} \int_0^\infty G(z) e^{-\eps z} \,dz + \CO(\eps)\;.
\end{equ}
It then follows from the fact that $\lim_{s \to \infty} G(s)/s = 1$ and the dominated convergence theorem that
\begin{equ}
\CI = \lim_{\eps \to 0} {\eps^2 \over 2} \int_0^\infty z e^{-\eps z} \,dt = {1\over 2}\;,
\end{equ}
which is precisely the desired expression. 
\end{proof}

\begin{remark}
It is clear that these results should hold under much weaker integrability conditions on $\nu$. However,
since we need some exponential moments on $\nu$ at several places in the sequel, we did not try to improve on this.
\end{remark}

\subsection{Some properties of the limiting process}
\label{sec:hairy}

In this section, we provide a rigorous definition of the limiting process loosely defined in Section~\ref{sec:limit},
and we study some of its properties.
In order to be able to use existing results on Brownian excursions, we restrict ourselves to the
same situation as in Section~\ref{sec:BMformal}, namely
the case where the underlying diffusion is a Brownian motion and the potential
$V(x) = -ax$ is linear. We call the resulting object the \textit{Brownian fan}.

\subsubsection{Recursive Poisson point processes}

Before we give a formal definition of the Brownian fan, we 
define a ``recursive Poisson point process''. Loosely speaking, this is a Crump-Mode-Jagers process \cite{CMJ}
with Poisson distributed offspring, but where the number of offspring of any given individual
is allowed to be almost surely infinite. Note again that our construction
is very similar to the one given in \cite{VirginIsland}.
Given a Polish space $\CX$ and a function $F \colon \CX \to \R_+$, we denote throughout
this section by $\MM_+^F(\CX)$ the space of $\sigma$-finite measures $\mu$ on $\CX$ such that 
\begin{equ}
\mu \bigl(F^{-1}(0)\bigr) = 0\;,\qquad \mu \bigl(\{x\,:\, F(x) > \eps\}\bigr) < \infty\;,
\end{equ}
for all $\eps > 0$. We endow this with the topology of convergence in total variation on each set of the form
$\{x\,:\, F(x) > 1/n\}$.
Given a (measurable) map $\CQ$ from $\CX$ to $\MM_+^F(\CX)$,
we can then build for every $x \in \CX$
a point process as follows.

Define $\mu^0_x = \delta_x$ and, for $n \ge 1$,
 define $\mu^n_x$ recursively as a (conditionally independent of the $\mu^{\ell}_x$ with $\ell < n$) 
realisation of a Poisson point process
with intensity measure 
\begin{equ}
\CQ_n = \int_\CX \CQ(y)\,\mu^{n-1}_x(dy)\;,
\end{equ}
where we view $\mu^n_x$ as a random $\sigma$-finite 
positive integer-valued measure on $\CX$.
(In principle, it may happen that $\CQ_n\bigl(\{x\,:\, F(x) > \eps\}\bigr) = \infty$ for some $\eps>0$. In this case,
our construction stops there.) 

When then set
\begin{equ}[e:PPP]
\mu^{[n]}_x = \sum_{\ell=0}^n \mu^n_x\;,
\end{equ}
and we call $\mu^{[n]}_x$ the \textit{recursive Poisson point process} of depth $n$ with kernel $\CQ$.
We will occasionally need to refer the Brownian fan spawned by an initial Brownian motion $w$.  For this purpose we will use the symbol $\mu^{[n]}_w$ (or $\mu^{n}_w$ for a specific generation) and rely on the context to differentiate $\mu^{[n]}_x$ 
and $\mu^{[n]}_w$.  If, in these symbols we omit the subscript entirely then it is assumed that $x=0.$
In general, there is no reason to expect the sequence $\mu^{[n]}_x$ to converge to a finite limit. 
However, one has the following simple
criterion ensuring that this is the case:

\begin{lemma}\label{lem:intPPP}
Let $F$ and $\CQ$ be as above and assume that
there exists $c < 1$ such that $\int F(y) \CQ(x,dy) \le c F(x)$ for every $x \in \CX$.
Then, for every $x\in \CX$, there exists a random $\sigma$-finite 
measure $\mu^{[\infty]}_x$ on $\CX$ such that $\lim_{n \to \infty} \E \int F(y) \bigl(\mu^{[\infty]}_x - \mu^{[n]}_x\bigr)(dy) = 0$.
\end{lemma}

\begin{proof}
Fix $x \in \CX$ and denote by $\FF_n$ the $\sigma$-algebra generated by $\mu^{[n]}_x$.
It then follows from the definition of the $\mu^{[n]}_x$ that 
\begin{equ}
\E \Bigl(\int_\CX F(y)\,\mu^{\ell+1}_x(dy)\,\Big|\, \FF_{\ell}\Bigr) = \int_\CX \int_\CX F(y) \CQ(z,dy)\,\mu^{\ell}_x(dz)
\le c \int_\CX F(z)\,\mu^{\ell}_x(dz)\;.
\end{equ}
As a consequence, one has $\E \int_\CX F(y)\,\mu^{\ell}_x(dy) \le c^\ell F(x)$, and the claim follows.
\end{proof}

\begin{remark}
A useful identity is the following. Denote by $\{\tilde \mu^{[\infty]}_y\}_{y \in \CX}$ a collection of 
\textit{independent} copies of recursive Poisson point processes with ``initial conditions'' $y$. 
Then one has the identity in law
\begin{equ}[e:recursion]
\mu^{[\infty]}_x \eqlaw \delta_x + \int_\CX \tilde \mu^{[\infty]}_y\, \mu^{1}_x(dy)\;,
\end{equ}
where, as before, $\mu^{1}_x$ is a realisation of a Poisson point process with intensity $\CQ(x,\cdot)$, which is 
itself independent of the $\tilde \mu^{[\infty]}_y$.
This identity makes sense since the integral on the right is really just a countable sum.
\end{remark}

%\begin{remark}
%This construction is very close to that of a general branching process. The difference is that since the 
%measures $\CQ(x,\dot)$ are allowed to have infinite mass, we consider situations where an individual can
%have infinitely many offspring with probability one.
%\end{remark}

\subsubsection{Construction of the Brownian fan}
\label{sec:BFan}

We now denote by $\CE$ the set of
excursions with values in $\R$. We consider elements of $\CE$ as
triples $(s,t,y)$ Where $s < t \in \R \cup \{+\infty\}$, and $y \in \CC(\R,\R)$ has the property that 
$y_{\tau}  = y_t$ for $\tau \ge t$ and $y_{\tau}  = y_{s} $ for $\tau \le s$. 
We also write $\CE_0$ for the subset of those triples $(s,t,y)$ such that $s = 0$.

Denoting a generic excursion by $w$, we write 
$\s(w)$ for its starting time and $\e(w)$ for its end time, \ie, $\s(s,t,y) = s$, 
and $\e(s,t,y) = t$. We also denote by $\l(w)$ the lifetime of the excursion, which is the 
interval $\l(w) = [\s(w), \e(w)]$.
In order to keep notations compact, we will also identify an excursion with its path 
component, making the abuse of notation
$w_t = y_t$. 
There is a natural metric on $\CE$ given by
\begin{equ}[e:distCE]
d(\tilde w,w) = d_\l(\tilde w,w) + \sum_{k \ge 1} 2^{-k} \Bigl(1 \wedge \sup_{|t| \le 2^k} |\tilde w_t - w_t|\Bigr)\;,
\end{equ}
where the distance $d_\l$ between the supports is given by
\begin{equ}
d_\l(\tilde w,w) = 1\wedge \bigl(|\s(\tilde w) - \s(w)| + |\tanh \e(\tilde w) - \tanh \e(w)|\bigr)\;. 
\end{equ}
The reason for this particular choice of metric is that it ensures that $\CE$ is a Polish space, while still
allowing for infinite excursions. 

For $\tau \in \R$ and $v \in \CE$,
we denote by $\Theta_{v,\tau}\colon \CE \to \CE$ the shift map given by
\begin{equ}
\Theta_{v,\tau} \colon (s,t,w) \mapsto \bigl(s + \tau, t + \tau, w_{ \cdot+\tau} + v_{\tau}\bigr)\;,
\end{equ}
which essentially changes the coordinate system so that the origin $(0,0)$ is mapped to $(\tau, v_{\tau})$.
Denoting as before by $\Q$ the standard It\^o excursion measure, we now give the following definition:

\begin{definition}\label{def:fan}
The \textit{Brownian fan} with intensity $a>0$ is the recursive Poisson point process on $\CE$ with
kernel
\begin{equ}[e:defHairy]
\CQ(w,\cdot) = {a \over 2} \int_{\s(w)}^{\e(w)} \Theta_{w,\tau}^\star \Q\, d\tau\;,
\end{equ}
and initial condition given by a realisation of Brownian motion, starting at the origin and killed when it reaches the level $-L$, where $L$ is exponentially
distributed with mean $a$.
\end{definition}

\begin{figure}
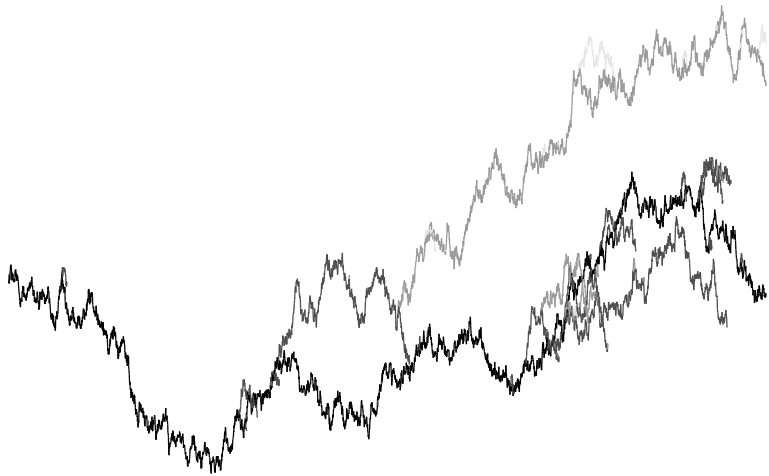

\begin{center}
\mhpastefig{sample}
\end{center}
\caption{A typical realisation of the Brownian fan with intensity $a={3\over 2}$ over the time interval $[0,1]$. Successive generations are draw in lighter shades of grey.}
\end{figure}

\begin{remark}
The reason for killing the original Brownian motion at this particular level is natural, due to the distribution
of the initial tag in Step~1 of the algorithm. It is however essentially irrelevant to the mathematical construction.
\end{remark}

\begin{remark}
Formally, the Brownian fan is a particular case of the Virgin Island Model \cite{VirginIsland}
with $a$ playing the same role in both models, $h = a$, and $g = 1/2$. The differences are twofold.
First, the case of constant non-vanishing $a$ actually doesn't fall within the framework of \cite{VirginIsland} since
the author there uses the standing assumption that $a(0) = 0$. The other difference is mostly one of perspective. while we
have so far defined the Brownian fan as a point process on a space of excursions, one of the 
purposes of this article is to show that it is also well-behaved as an actual Markov process with
values in a suitable state space of (possibly infinite) point configurations.
\end{remark}

\begin{remark}
By only keeping track of the genealogy of the particles and not their precise locations, one can construct
a ``real tree'' on top of which the Brownian fan is constructed (loosely speaking) by attaching a Brownian
excursion to each branch. This is very similar in spirit to Le Gall's construction of the Brownian map \cite{BrownianMap,BrownianMap2}, 
starting from Aldous's continuous random tree \cite{CRT}. The scaling properties of the Brownian fan
however are quite different. In particular, Theorem~\ref{theo:Nt} below implies that the underlying tree has Hausdorff dimension
$1$, as opposed to the CRT which has Hausdorff dimension $2$.
\end{remark}

Before we proceed, let us show that it is possible to verify the assumptions of Lemma~\ref{lem:intPPP}, so that 
this object actually exists for every $a > 0$:

\begin{proposition}
The kernel $\CQ$ defined in \eref{e:defHairy} satisfies the assumption of Lemma~\ref{lem:intPPP} with the choice
\begin{equ}
F(w) = e^{-\eta \s(w)} \bigl(1 - e^{-\eta |\l(w)|}\bigr)\;,
\end{equ}
provided that $\eta$ is large enough.
\end{proposition}

\begin{proof}
It follows from the properties of $\Q$ that there exists a constant $C$ independent of $\gamma$ and $a$ such that
\begin{equs}
\int_\CE F(\tilde w) \CQ(w,d\tilde w) &= {a\over \sqrt{2\pi}} \int_{\s(\tilde w)}^{\e(\tilde w)} e^{-\eta \tau} \int_0^\infty \bigl(1 - e^{-\eta s}\bigr) s^{-3/2}\,ds\,d\tau \\
&\le {Ca  \over \sqrt \eta}  \bigl(e^{-\eta \s(\tilde w)} - e^{-\eta \e(\tilde w)}\bigr) = {Ca  \over \sqrt \eta} F(v)\;,
\end{equs}
where $C$ is a constant independent of $a$ and $\eta$.
The claim then follows by choosing $\sqrt \eta > Ca$.
\end{proof}

The remainder of this section is devoted to a study of the basic properties of the Brownian fan.
In particular,we will show that there exists a suitable space $\CX$ 
of particle configurations such that it can be viewed as a $\CX$-valued Markov process with continuous
sample paths that satisfies the Feller property.

\subsubsection{Number of particles and workload rate}
\label{sec:numPart}

Define the set $\CN_t \subset \CE$ of excursions that are ``alive at time $t$'' by 
\begin{equ}
\CN_t = \{w \in \CE\,:\, \s(w) < t < \e(w)\}\;.
\end{equ}
With this notation, the number of particles alive at time $t$ for the Brownian fan is given by
\begin{equ}
N_t = \mu^{[\infty]}(\CN_t)\;,
\end{equ}
which is in principle allowed to be infinite.

\begin{theorem}\label{theo:Nt}
There exist a constant $C>0$ and a strictly positive continuous decreasing function $\lambda \colon \R_+ \to \R_+$ such that
\begin{equ}[e:boundNt]
\E \exp \bigl( \lambda_t N_t\bigr) \le C\;,
\end{equ}
holds uniformly over all $t>0$.
\end{theorem}

\begin{remark}
We will see in the proof that one can choose $\lambda$ of the form
\begin{equ}[e:goodlambda]
\lambda_t = K^{-1} e^{-Kt}\;,
\end{equ}
for $K$ sufficiently large.
\end{remark}

\begin{proof}
For $\lambda > 0$ and $s,t \ge 0$, set
\begin{equ}
N^\lambda_{s,t} = \log \E \exp \bigl(\lambda \mu^{[\infty]}_w(\CN_t)\bigr)\;,
\end{equ}
where $w$ is any excursion starting from $0$ with lifetime $s$. Since, by the definition \eref{e:defHairy},
the value of $N^\lambda_{s,t}$  does not depend on the precise choice of excursion, we do not include it in the notation.
It also follows from the construction of $\mu^{[\infty]}_w$ that the function
$t \mapsto N^\lambda_{t,t}$ is increasing in $t$ and that $N^\lambda_{s,t} = N^\lambda_{t,t}$ for $s \ge t$.

We also define $M^\lambda_{t}$ by
\begin{equ}[e:defGlambda]
M^\lambda_{t} = \log \E \exp \Bigl(\lambda \int_{\CE} \mu^{[\infty]}_w(\CN_t)\, \cM(dw)\Bigr)\;,
\end{equ}
where $\cM$ is a Poisson random measure with intensity measure $\Q$ and the realisations $\mu^{[\infty]}_w$ are independent of $\cM$ and of each other.
While $N^\lambda$ measures the
total number of offspring alive at time $t$ due to an excursion starting at time $0$, $M^\lambda$ measures
the rate at which these offspring are created. 

Indeed, combining \eref{e:recursion} with the definition of $\mu_{w}^1$ and the superposition principle
for Poisson point processes, we have the identity 
\begin{equ}[e:relNG]
N^\lambda_{s,t} = \lambda \one_{s \ge t} + \frac{a}{2} \int_0^{s \wedge t} M^\lambda_{t-r}\,dr\;.
\end{equ}
It therefore remains to obtain suitable bounds on $M_t^\lambda$.

It follows from \eref{e:relNG} and standard properties of Poisson point processes (see for example \cite[Theorem~6.3]{MR2356959}) that one has the identity
\begin{equs}
M^\lambda_{t} &= \int_\CE \bigl(\E \exp \bigl(\lambda \mu^{[\infty]}_w(\CN_t)\bigr) - 1\bigr)\,\Q(dw)
 = \int_\CE \bigl(e^{N^\lambda_{\e(w),t}} - 1\bigr)\,\Q(dw)\\
 &\le \int_\CE \Bigl( \exp \Bigl(\lambda\one_{\CN_t}(w) + \frac{a}{2} \int_0^{\e(w)\wedge t} M^\lambda_{t-s}\,ds\Bigr) - 1\Bigr)\,\Q(dw)\;.
 \end{equs}
At this stage, we note that the integrand appearing in this expression depends on $w$ only
through $\e(w)$. It is then
convenient to break the integral into a contribution coming from $\e(w) > t$, as well as its complement.
Since, under $\Q$, the quantity $\e(w)$ is distributed according to the measure ${s^{-3/2}\over \sqrt{2\pi}}ds$, 
this yields
\begin{equs}
M^\lambda_{t} &\le \Bigl(  \exp \Bigl(\lambda + \frac{a}{2}\int_0^{t} M^\lambda_r\,dr\Bigr)  - 1\Bigr)\,\Q\bigl(\e(w) \ge t\bigr)\\
&\quad + \int_0^t \Bigl(  \exp \Bigl(\frac{a}{2}\int_0^{s} M^\lambda_{t-r}\,dr\Bigr)  - 1\Bigr)\,{s^{-3/2}\over \sqrt{2\pi}}\,ds\;.
\end{equs}
We now assume that both $\lambda$ and $t$ are sufficiently small so that
\begin{equ}[e:assumption]
\lambda + \frac{a}{2} \int_0^{t} M^\lambda_r\,dr \le 1\;.
\end{equ}
This assumption allows us to use the bound $e^t - 1 \le 2t$, so that we obtain the more manageable expression
\begin{equs}
M^\lambda_{t} &\le {4 t^{-1/2} \over \sqrt{2\pi}}\Bigl(\lambda + \frac{a}{2}\int_0^{t} M^\lambda_r\,dr\Bigr)
+ a \int_0^t \int_0^{s} M^\lambda_{t-r}\,dr\,{s^{-3/2}\over \sqrt{2\pi}}\,ds\\
&= {4 \lambda t^{-1/2} \over \sqrt{2\pi}} +
{2 a t^{-1/2} \over \sqrt{2\pi}}\int_0^{t} M^\lambda_r\,dr
+ a \int_0^t \int_s^{t} M^\lambda_r\,dr\,{(t-s)^{-3/2}\over \sqrt{2\pi}}\,ds\\
&= {4\lambda t^{-1/2} \over \sqrt{2\pi}} + {2 a t^{-1/2} \over \sqrt{2\pi}}\int_0^{t} M^\lambda_r\,dr
+ {2a \over \sqrt{2\pi}} \int_0^t (t-s)^{-1/2}\, M^\lambda_s\,ds\\
&\le {4\lambda t^{-1/2} \over \sqrt{2\pi}} 
+ {4a \over \sqrt{2\pi}} \int_0^t (t-s)^{-1/2}\, M^\lambda_s\,ds\;.
\end{equs}
Writing $H^\lambda_t = t^{1/2} M^\lambda_{t}$, we thus obtain the bound
\begin{equ}[e:Hlambda]
H^\lambda_t \le {4\lambda\over \sqrt {2\pi}} + {4a \over \sqrt{2\pi}} t^{1/2} \int_0^t (t-s)^{-1/2}s^{-1/2}\, H^\lambda_s\,ds\;.
\end{equ}
We can now apply the fractional version of Gronwall's lemma  \cite[Lemma~7.6]{Nualart:2002p6052} (with $b = {4a \over \sqrt{2\pi}} $,
$a = {4\lambda \over \sqrt {2\pi}}$,
and $\alpha = {1\over 2}$), so that there exists a constant $C>0$ depending on $a$ but
independent of $\lambda$ such that
\begin{equ}
H^\lambda_t \le C\lambda \exp(Ct)\;.
\end{equ}
From this, we immediately deduce from \eref{e:relNG} and the definition of $H^\lambda$
a similar bound on $N^\lambda_{t,t}$.
Choosing $\lambda = K^{-1}e^{-Kt}$ for sufficiently large $t$ then allows to satisfy \eref{e:assumption}
and to obtain $N^\lambda_{t,t}\le 2$, thus completing the proof.
\end{proof}

%\begin{remark}
%A toy model for the behaviour of $N$ which is very close to the behaviour actually observed for our process
%is the following. Let $\mu_n$ be a Poisson process on $\R$ with intensity $2^n$, let $U_n_t = \int \chi_n(t-s)\,\mu_n(ds)$
%with $\chi_n = \one_{[0,2^{-2n}]}$, and set $N_t = \sum_{n \ge 1} U_n_t$.
%Elementary properties of Poisson processes then show that
%\begin{equ}
%\log \E \exp (\lambda U_n_t) = 2^n\int_{\R} \bigl(e^{\lambda \chi_n(t-s)}-1\bigr)\,ds = 2^{-n} \bigl(e^\lambda -1\bigr)\;.
%\end{equ}
%As an immediate consequence of the independence of the $U_n$, we then indeed have
%$\log \E \exp (\lambda N_t) = e^{\lambda}-1$.
%\end{remark}

As a corollary, we obtain a rather sharp bound on the modulus of continuity of 
the total workload process $\CW_t = \int_0^t N_s\,ds$. One has

\begin{proposition}\label{prop:bfWmodcont}
For every $T>0$, one has 
\begin{equ}[e:boundModW]
\sup_{t \le T}\lim_{h \to 0} {|\CW_{t+h} - \CW_t| \over h |\log h|} < \infty\;,
\end{equ}
almost surely.
\end{proposition}

\begin{proof}
It follows from the generalised Young inequality that, for every $a,b \in \R_+$, and every $\lambda, \eta > 0$, one
has the inequality
\begin{equ}
ab \le {\eta \over \lambda} \bigl(e^{\lambda a} - 1 - \lambda a + (1+ b/\eta)\log(1+b/\eta) - b/\eta \bigr)\;,
\end{equ}
so that 
\begin{equ}
N_t \le {\eta \over \lambda} \bigl(e^{\lambda N_t} + (1+ 1/\eta)\log(1+1/\eta)\bigr)\;.
\end{equ}
It follows immediately that
\begin{equ}
|\CW_{t+h} - \CW_{t}| \le  {\eta \over \lambda} \int_t^{t+h} e^{\lambda N_s}\,ds + {h (1+ \eta)\over \lambda} \log(1+1/\eta)\;.
\end{equ}
Setting $\eta = h$, we obtain the bound
\begin{equ}
|\CW_{t+h} - \CW_t| \le  {h \over \lambda} \int_0^{T+1} e^{\lambda N_s}\,ds + C_\lambda h |\log h|\;,
\end{equ}
uniformly over all $h \le 1$ and all $t \in [0,T]$.
The claim now follows immediately from Theorem~\ref{theo:Nt}.
\end{proof}

%\begin{remark}\label{rem:infinite}
Although the number of particles alive at any deterministic time has exponential moments, there
exists a \textit{dense} set of exceptional times for which $N_t = \infty$. For one, this follows
from the fact that, under $\Q$, $\e(w)$ is distributed proportionally to $s^{-3/2} ds$, so that
every particle creates an infinite number of offspring in every time interval. 

Actually, one has the even stronger statement that there is a dense set of exceptional times
at which the number of particles belonging to the \textit{first} generation of offspring is infinite.
Indeed, if we denote by $\MM$ a Poisson random measure on $\R_+^2$ with density
$c s^{-3/2}\,dr\,ds$ for a suitable constant $c$, then the number 
$N^1_t$ of particles in the first generation of offspring
is given by
\begin{equ}
N^1_t = \MM(A_t)\;,\qquad A_t = \{(r,s) \in [0,t] \times \R_+ \,:\, s \ge t-r\}\;.
\end{equ}
For $k \ge 0$ and $d \in \{1,\ldots 2^k\}$, we then set
\begin{equ}
A_{k,d} = [(d-1)2^{-k}, d2^{-k}] \times [4^{-k}, 4^{1-k}]\;,
\end{equ}
so that, by the scaling properties of $\MM$, 
the random variables $N_{k,d} = \MM(A_{k,d})$ form a sequence of i.i.d.\ Poisson random variables.
For any given point $(r,s)$, we set $B_{(r,s)} = [r, r+s]$, which is the set of times $t$ such that $(r,s) \in A_t$, and we set 
\begin{equ}
D_{(r,s)} = \{(k,d) \,:\, B_{(r',s')} \subset B_{(r,s)}\, \forall (r',s') \in A_{k,d}\}\;.
\end{equ}
Since the set $D_{(r,s)}$ is infinite for every $(r,s)$, we can then build a sequence $(r_n, s_n)$ recursively in the following way. Start with $(r_0, s_0) = (0,1)$ and then, given $(r_n, s_n)$ for some $n\ge 0$, define
$(k_n, d_n)$ as the first (in lexicographic order) element $(k,d) \in D_{(r_n, s_n)}$ such that
$N_{k,d} \ge 1$. We then set $(r_{n+1}, s_{n+1})$ to be one of the points of $\MM$ located in 
$A_{k_n,d_n}$. By construction, one then has $\cap_{n \ge 1} B_{(r_n,s_n)} = \{t\}$ 
for some $t \in [0,1]$, and  $\MM(A_t) \ge \sum_n \MM(A_{k_n, d_n}) = \infty$, as stated.
Of course, the interval $[0,1]$ in this procedure is arbitrary. If we want to show that there exists
an exceptional time within any deterministic time interval $[t_0, t_1]$, it suffices to start the algorithm
 we just described with $r_0 = t_0$ and $s_0 = t_1 - t_0$.

\section{The Brownian fan as a Markov process}

In this section, we slightly shift our perspective. We no longer consider the Brownian fan as a point
process of excursions, but we consider it as an evolving system of particles.
 Our system will therefore be described by a Markov process in some space 
of integer-valued measures
on a subset of $\R^2$ corresponding to the admissible combinations of ``position + tag''. The problem
is that, as we have seen in the previous section, there are exceptional times at which the limiting process 
consists of infinitely many particles. The first challenge is therefore to construct a space $\CX$
of integer-valued measures with
a sensible topology which can still accommodate these ``bad'' configurations in such a way 
that the limiting process is continuous both as a function of time and as a function of its 
initial configuration.

Once this space is defined, we show that the Brownian fan possesses the Feller property in $\CX$
(i.e.\ the corresponding Markov semigroup leaves the space of bounded continuous functions invariant). 
In fact, we will show that it preserves the space of Lipschitz continuous functions.
The continuity property of the discrete-time process established below in Proposition~\ref{prop:Kolmogorov} 
also implies the time continuity of
the Brownian fan (see Theorem~\ref{theo:finalConv} below). These properties  allow us to conclude
that the Brownian fan  $t \mapsto \mu_t$ is in fact a strong 
Markov process.

In Section~\ref{sec:noConv} below, we will furthermore compute its generator $A$
on a class of ``nice'' test functions. We will also present a ``negative'' result showing that if we denote by
$T_\eps$ the one-step Markov operator corresponding to the evolution of Algorithm~\ref{tdmc},
then one has $A \neq \lim_{\eps \to 0} \eps^{-1} \bigl(T_\eps - 1\bigr)$.
This is in stark contrast with, for example, Euler approximations to stochastic differential equations, where
such an equality would hold, at least when applied to sufficiently regular test functions.

\subsection{State space}
\label{sec:statespace}

Our construction is essentially the Wasserstein-$1$ analogue of the construction given in \cite{GigliFigalli}.
Let $\CM\subset \R^n$ be a convex open set with boundary $\d\CM$.
For $p \in (0,1]$, we then denote by $\ell^p(\CM)$ the set of all integer-valued measures
$\mu$ on $\CM$ such that 
\begin{equ}[e:normp]
\|\mu\|_p = \int_{\CM} d^p(y,\d \CM) \,\mu(dy) < \infty\;,
\end{equ}
where $d(y,\d \CM)$ denotes the (Euclidean) distance from $y$ to the boundary of $\CM$.
Note that since this quantity vanishes at the boundary, there are elements $\mu\in\ell^p(\CM)$ 
such that $\mu(\CM) = \infty$.

We endow $\ell^p(\CM)$ with a slight modification of the Wasserstein-$1$
metric by setting:
\begin{equ}[e:defDistp]
\|\mu-\nu\|_p = \sup_{f \in \Lip_p^0(\CM)} \Bigl(\int f(y)\mu(dy) - \int f(y)\nu(dy)\Bigr)\;,
\end{equ}
where we denoted by $\Lip_p^0(\CM)$ the set of all functions 
$f \colon \CM\to \R$ such that
\begin{equ}[e:boundLipf]
|f(x) - f(y)| \le  |x-y|^p\;,
\end{equ}
for all $x,y \in \CM$, and  $f(y) = 0$ for all $y \in \d \CM$.

\begin{remark}
Our notation is consistent in the sense that if we take for $\nu$ the null measure in \eref{e:defDistp},
then we precisely recover \eref{e:normp}. This can be seen by taking $f(x) = d^p(x,\d \CM)$,
which is optimal by \eref{e:boundLipf} and the triangle inequality.
\end{remark}

 If $\mu$ and $\nu$ happen to have the same (finite) mass, then the expression \eref{e:defDistp} does not change
when one adds a constant to $f$. In this case, we are thus reduced to the usual Wasserstein-$1$ distance between 
$\mu$ and $\nu$, but with respect to the modified distance function
\begin{equ}
d_p(x,y) = |x-y|^p  \wedge \bigl(d^p(x,\d\CM) + d^p(y,\d\CM)\bigr)\;.
\end{equ}
Note that the completion of $\CM$ under the distance function $d_p$ consists 
of $\CM \cup \{\Delta\}$, where $\Delta$ is a single ``point on the boundary''
such that $d_p(x,\Delta) = d^p(x,\d\CM)$ for every $x \in \CM$.

If one has $\mu(\CM) < \nu(\CM) < \infty$, then the distance $\|\cdot\|_p$ 
reduces to the Wasserstein-$1$ distance (again with respect to $d_p$) between
$\bar \mu$ and $\nu$, where $\bar\mu$ is obtained from $\mu$ by placing a mass $\nu(\CM) - \mu(\CM)$ on 
the boundary $\Delta$.
The following alternative characterisation of \eref{e:defDistp} in the case of purely atomic measures
is a version of the Monge-Kantorovich duality in this context:

\begin{lemma}
Consider a situation where $\mu = \sum_{i=1}^N \delta_{x_i}$ and $\nu = \sum_{i=1}^M \delta_{y_i}$. 
Then, 
\begin{equ}
\|\mu-\nu\|_p = \inf_{\sigma \in S_{N+M}} \sum_{i=1}^{N+M} d^p(x_i, y_{\sigma(i)})\;,
\end{equ}
where $S_{N+M}$ is the group of permutations of $N+M$ elements and we
set $x_j = \Delta$ for $j > N$ and $y_j = \Delta$ for $j>M$.
\end{lemma}

\begin{proof}
See for example \cite{MR2369050}.
\end{proof}

This characterisation suggests the following ``interpolation'' procedure between elements in $\ell^p(\CM)$.
Let $\mu = \sum_{i=1}^N \delta_{x_i}$ and $\nu = \sum_{i=1}^N \delta_{y_i}$, where we assumed that both measures
charge the same number of points (this is something that we can always achieve by possibly adding points on $\Delta$).
assume furthermore that these points are ordered in such a way that 
\begin{equ}
\|\mu-\nu\|_p = \sum_{i=1}^{N} |x_i- y_i|^p\;.
\end{equ}
Again, this can always be enforced by suitably reordering the points and possibly adding points on the boundary.
We then define, for $t \in (0,1)$, the ``linear interpolation'' $L_t(\mu,\nu)$ by
\begin{equ}[e:linearInterp]
L_t(\mu,\nu) = \sum_{i=1}^N \delta_{z_i}\;,\qquad z_i = ty_i + (1-t)x_i\;.
\end{equ}
Note that this procedure is not necessarily unique, but it is easy to resolve this ambiguity by optimising over the possible
pairings $\{(x_i, y_i)\}$ realising the above construction, according to some arbitrary criteria.

In any case, one can check that this construction has the property that
\begin{equ}[e:boundInter]
\|L_s(\mu,\nu)-L_t(\mu,\nu)\|_p \le |t-s|^p \|\mu-\nu\|_p\;,
\end{equ}
for any $s,t \in [0,1]$, which will be a useful fact in the sequel.

\subsection{Definition of the process}
\label{sec:defFanMP}

For the remainder of this section, we set
\begin{equ}
\CM = \{(x,v) \in \R^2\,:\, v > -ax\}\;,
\end{equ}
which is the natural configuration space for our process.  We will use capital letters to distinguish points in $\CM$ from points in $\R$.
By Theorem~\ref{theo:Nt}, we already know that, for any fixed time $t$, the Brownian fan almost surely has 
only finitely many particles alive at time $t$. Define now the evaluation map 
$E_t \colon \CE \to \CM \cup \{\Delta\}$ by
\begin{equ}
E_t(w) = 
\left\{\begin{array}{cl}
	\bigl(w_t, -a w_{\s(w)}\bigr) & \text{if $t \in \l(w)$,} \\
	\Delta & \text{otherwise.}
\end{array}\right.
\end{equ}
For a given ``initial condition'' $(x,v) \in \CM$, we then set
\begin{equ}
\mu_t = E_t^\star \mu^{[\infty]}_w \;,
\end{equ}
which is an $\ell^p(\CM)$-valued random variable.
Here, $w$ is a realisation of a Brownian motion starting at $x$ and killed when it hits $-v/a$,
and $\mu^{[\infty]}_w$ is the corresponding realisation of the Brownian fan. (Just so that $E_t$ 
has the correct effect
on $w$, one can for example set $\s(w) = -1$ and make sure that $w(-1) = -v/a$.)
As a consequence of Theorem~\ref{theo:Nt} and of our definition of the Brownian fan, 
we then indeed have $\mu_t \in \ell^p(\CM)$ for every $p \le 1$. 

Note at this stage that we can simply discard the, typically infinite, mass on $\Delta$ by identifying measures
that only differ on $\Delta$. As already mentioned earlier, this is consistent with the identification 
$\Delta \sim \d \CM$ already made
in the interpretation of the construction of $\ell^p(\CM)$.

This construction can be extended to any initial condition in $\ell^p(\CM)$ with finite total mass, by considering independent
Brownian fans for each particle.
As a consequence
of the Markov property of the Brownian excursion and the independence properties of Poisson
point processes, it is then straightforward to verify that $t \mapsto \mu_t$ is indeed a Markov process.

Actually, by Proposition~\ref{prop:numPart}, we know
that for any fixed collection of deterministic times $\{t_1,\ldots, t_k\}$, one has $\mu_{t_k}(\CM) < \infty$ almost surely, 
so that our construction determines a probability measure on $\bigl(\ell^p(\CM)\bigr)^{\R_+}$ by Kolmogorov's extension theorem.
At this stage however, we know absolutely nothing about the continuity properties of this process, and this
is the subject of the remainder of this section.

\subsection{Feller property}

We now show that the Brownian fan constructed in Sections~\ref{sec:BFan} and \ref{sec:defFanMP}
has the Feller property in $\ell^p(\CM)$ for every $p \le 1$.

As in \cite{GigliFigalli}, we could have defined spaces $\ell^p(\CM)$ in a natural way for $p > 1$.
However, the Feller property would fail in this case because of the following simple heuristic argument. 
For any $p$, we can change the initial condition by an amount less than $\delta$ in $\ell^p$ by
creating $N$ particles at distance $\eps = (\delta/N)^{1/p}$ from the boundary of $\CM$.
For $\eps$ small,  the probability that any such particle survives up to 
time $1$ (say) is bounded from below by $c\eps$ for some $c>0$. On average, the number of survivors 
will thus be on the order of $\eps N \sim \delta^{1/p} N^{(p-1)/p}$.
Furthermore, at time $1/2$, each of these surviving particles will be at a distance of order $1$ of the 
boundary of $\CM$. As a consequence, by increasing $N$ but keeping $\delta$ fixed (or even
sending $\delta$ to $0$ sufficiently slowly), the law of the process at time $1$ with an initial condition
arbitrarily close to $0$ can be at arbitrarily large distance of $0$, so that the Feller property fails.

For $p \le 1$ on the other hand, we have

\begin{proposition}\label{prop:feller}
For any $p \le 1$, the Brownian fan gives rise to a Feller process in $\ell^p(\CM)$. 
Even more, the corresponding Markov semigroup preserves the space of bounded 
Lipschitz continuous functions. 
\end{proposition}

\begin{proof}
For any two initial conditions $\mu_0$ and $\bar \mu_0$, write as before
\begin{equ}
\mu_0 = \sum_{j=1}^N \delta_{X_0^{(j)}}\;,\qquad \bar \mu_0 = \sum_{j=1}^N \delta_{\bar X_0^{(j)}}\;,
\end{equ} 
with the $X_0^{(j)}\in\CM$ and $\bar X_0^{(j)}\in \CM$ chosen in such a way that
\begin{equ}[e:equaldp]
\|\mu_0- \bar \mu_0\|_p = \sum_{j=1}^N d_p(X_0^{(j)},\bar X_0^{(j)})\;.
\end{equ}
Our aim now is to show that there exists a constant $C$ such that
\begin{equ}
\E \|\mu_t- \bar \mu_t\|_p \le C \|\mu_0-\bar \mu_0\|_p\;,
\end{equ}
independently of $t \le 1$, where the pair $(\mu_t,\bar \mu_t)$ is a particular coupling between the Brownian fans
starting from $\mu_0$ and $\bar \mu_0$ respectively.
Denote by $\mu_t^{(j)}$ the contribution to $\mu_t$ originating from the initial particle $X_0^{(j)}$ and
similarly for $\bar \mu_t^{(j)}$. Then, by the triangle inequality,
one obtains the bound
\begin{equ}
\E \|\mu_t- \bar \mu_t\|_p \le \sum_{j \ge 1} \E \|\mu_t^{(j)}- \bar \mu_t^{(j)}\|_p \;,
\end{equ}
so that the claim follows from \eref{e:equaldp} if we can show that
\begin{equ}
\E \|\mu_t^{(j)}- \bar \mu_t^{(j)}\|_p \le C d_p(X_0^{(j)}, \bar X_0^{(j)})\;.
\end{equ}
In other words, it suffices to consider the special case when both $\mu_0$ and $\bar \mu_0$ consist of
one single particle, which we denote by $X_0 = (x_0, v_0)$ and $\bar X_0 = (\bar x_0, \bar v_0)$ respectively.

One then constructs a coupling between the two processes $\mu_t$ and $\bar \mu_t$ by running both particles
with the same Brownian motion and spawning children according to the same Poisson process (as long as the corresponding particle is alive). 
We denote by $X_t$ and $\bar X_t$ the evolutions of the two initial particles in $\CM$, driven by the
same realisation of a Brownian motion, and stopped when they reach $\d\CM$. 
We can assume without loss of generality that $v_0 + ax_0 < \bar v_0 + a\bar x_0$, so that the 
particle $X$ dies before the particle $\bar X$. Denoting by $\tau$ and $\bar \tau$ the respective lifetimes
of these particles, one thus has $\tau \le \bar\tau$.

Denote now by $\MM_t$ the (random) measure on $[0,t]\times \CM$ which is such that,
for $I \subset [0,t]$ and $A \subset \CM$, $\MM_t(I\times A)$ is the number of particles
in $A$ at time $t$ that are offspring of a particle created from the ``ancestor particle'' $\bar X$ at some time
$s\in I$. With this notation, if we denote by $\Xi\colon \CM \to \CM$ the map 
\begin{equ}
\Xi(x,v) = \bigl(x+ x_0 - \bar x_0, v - a ( x_0 - \bar x_0)\bigr)\;,
\end{equ}
then one has the decompositions
\begin{equs}
\mu_t(A) &= \one_{\tau \ge t} \delta_{X_t}(A) + \Xi^\star \MM_t([0,\tau\wedge t] \times A)\;,\\
\bar \mu_t(A) &= \one_{\bar \tau \ge t} \delta_{\bar X_t}(A) + \MM_t([0,\bar\tau\wedge t] \times A)\;.
\end{equs}
Denote now by $\delta$ the \textit{Euclidean} distance between the two initial particles,
so that their $\ell^p$-distance is $\delta^p$. 
It then follows immediately from the above decomposition that one has the bound
\begin{equs}
\|\mu_t - \bar \mu_t\|_p &\le \delta^p + \one_{t \in [\tau,\bar \tau]}d_p(\bar X_t,\d \CM) +  \delta^p \MM_t([0,\tau\wedge t] \times \CM) \\
&\quad + \int_\CM d_p(Y,\d \CM)\MM_t([\tau \wedge t,\bar \tau\wedge t] \times dY)\;.
\end{equs}
Since $d(\bar X(\tau),\d \CM) = \delta$ by the definition of $\tau$ and $\delta$, it follows 
from Jensen's inequality and the Martingale
property of (stopped) Brownian motion 
that one has the bound $\E \one_{t \in [\tau,\bar \tau]}d_p(\bar X_t,\d \CM) \le \delta^p$. It 
also follows from Theorem~\ref{theo:Nt} that $\E \MM_t([0,\tau\wedge t] \times \CM) < \infty$,
independently of $\delta$. Finally, it follows from an argument very similar to the proof of
Theorem~\ref{theo:Nt} that 
\begin{equ}
\E \int_\CM d_p(Y,\d \CM)\MM_t(ds \times dY)
\le C\, ds\;,
\end{equ}
uniformly over $s \in [0,t]$. It follows that
\begin{equ}
\E \int_\CM d_p(Y,\d \CM)\MM_t([\tau \wedge t,\bar \tau\wedge t] \times dY)
\le C \E |\bar \tau\wedge t - \tau \wedge t| \le C \delta\;,
\end{equ}
where we used the fact that if $\tau_\delta$ is the first hitting time of $0$ by a Brownian motion
starting at $\delta$, then $\E (\tau_\delta \wedge 1) \le C\delta$.
Combining these bounds completes the proof.
\end{proof}

\subsection{Lack of convergence of the generators}
\label{sec:noConv}

One standard method to prove convergence of a sequence of Markov processes to a limiting
process, once tightness has been established, is to show that the corresponding generators converge
in a suitable sense. In our situation, one actually does  \textit{not} expect the generator
of the approximate process to converge to that of the limiting process, when testing it
on ``nice'' test functions. We first argue at a technical level why this is the case, before
providing an intuitive explanation.

Inspired by \cite{EthKur86MP,Dawson}, we consider test functions of the form
\begin{equ}[e:shapeTF]
F(\mu) = \exp \bigl(\scal{\log f, \mu}\bigr)\;,
\end{equ}
where $f \colon \CM \to \R_+$ is a sufficiently smooth function such that $f(x,v) = 1$ for $(x,v) \in \d \CM$.
This boundary condition is required since elements $\mu \in \ell^p(\CM)$ can have infinite
mass (and, as we have already seen, the limiting process really does acquire infinite mass at some
exceptional times) accumulating near $\d \CM$. Being in $\ell^p(\CM)$ for $p \le 1$ does however
ensure that smooth functions such as above are integrable.

Exploiting the independence structure of the process as well as its space homogeneity, 
we can reduce ourselves to the case of an initial condition of the form
$\mu_0 = \delta_{(x,v)}$ for some $v < -ax$. In this case, for sufficiently small $\eps > 0$,
the probability that the original particle dies within the time interval $\eps$ is of the order $\eps^p$ for any $p>0$.
We therefore only need to take into account the possibility of creating some descendant(s), with the
killing mechanism being taken care of by the boundary condition of $f$. While the
average number of ``second generation'' descendants is of order $\eps$, any such descendant will typically
have travelled to a distance of order $\sqrt \eps$ from $\d\CM$, 
so that only the first generation has a chance of contributing to the generator.

We then have
\begin{equ}
{1\over \eps} \bigl(\E F(\mu_\eps) - F(\mu_0)\bigr) \approx A_0 f(x,v) + {f(x,v) \over \eps} \E \scal{f-1,\MM_\eps}\;,
\end{equ}
where
\begin{equ}
A_0 = {1\over 2}\d_x^2\;,
\end{equ}
is the generator of Brownian motion, and where
$\MM_\eps$ is the (projection to time $\eps$ of the) Poisson point process yielding the first generation of offsprings.
Note now that since these offspring will be created near $\d \CM$ and since $f=1$ there, we can 
further approximate this expression by
\begin{equ}
{1\over \eps} \bigl(\E F(\mu_\eps) - F(\mu_0)\bigr) \approx A_0 f(x,v) + {f(x,v) \over \eps} f'(x)\, \E \int_\CM (\bar x + a^{-1}\bar v) \MM_\eps(d\bar x,d\bar v)\;.
\end{equ}
Here, we wrote $f'(x)$ as a shortcut for $\d_x f(x,v)\big|_{v = -ax}$.
Denoting by $e_s$ the position, relative from its starting point, 
of an excursion of length $s$ and making use of the formula \eref{e:defHairy} for the intensity 
measure of $\MM_\eps$, we obtain for the last term in this equation the expression
\begin{equ}[e:Gen]
{1\over \eps} \E \int_\CM (\bar x + a^{-1}\bar v) \MM_\eps(d\bar x,d\bar v) = {a\over 2\sqrt{2\pi} \eps} \int_0^\infty \int_0^{s\wedge \eps} \E e_s(t)\,dt\, s^{-3/2}\,ds \;.
\end{equ}
At this stage we note that, for a  Brownian excursion of length $s$, we have for $t \le s$ the identity
\begin{equ}
\E e_s(t) =  \sqrt{{8 t \over \pi s} (s-t)}%,\qquad \int_0^s \E e_s_t\,dt =  {1\over 4}\sqrt{2\pi} s^{3/2}\;.
\end{equ}
which can be computed using the explicit formula given for the transition probabilities of the Brownian excursion on page 59 of \cite{BMHandbook}.
Inserting this into \eref{e:Gen}, a tedious but straightforward calculation then yields
\begin{equ}
\lim_{\eps \to 0} {1\over \eps} \E \int_\CM (\bar x + a^{-1}\bar v) \MM_\eps(d\bar x,d\bar v) = {a\over 2}\;,
\end{equ}
so that we finally obtain for the generator $A$ the expression
\begin{equ}[e:genBF]
A F(\mu_0) =  A_0 f(x,v) + {a \over 2} f(x,v) f'(x) \;.
\end{equ}
Recall that this is for the particular case where $\mu_0 = \delta_{(x,v)}$. In the general case, we can
use the independence structure of the process to obtain
\begin{equ}
A F(\mu_0) =  F(\mu_0) \int_{\CM} \Bigl({A_0 f(x,v) \over f(x,v)} +  {a \over 2} f'(x) \Bigr)\,\mu_0(dx,dv)\;.
\end{equ}

\begin{remark}
Compare this with the generator of a usual branching diffusion, where the term $a f'(x)/2$ would be replaced
by $a(f(x) - 1)$, with $a$ the branching rate.
\end{remark}

On the other hand, if we denote by $T_\eps$ the Markov operator describing one step 
of Algorithm~\ref{tdmc}, we might expect that one also obtains $A$ as the limit $\eps^{-1} \bigl(T_\eps - 1\bigr)$
as $\eps \to 0$. This would indeed be the case if there was no branching or if branching did only
occur at a finite rate. Considering again initial conditions of the form $\mu_0 = \delta_{(x,v)}$
for some $v < -ax$. By reasoning similar to that in the beginning of Section~\ref{sec:contLimit}, we obtain 
%\begin{equ}[e:discrete2]
%\eps^{-1} \bigl(T_\eps F - F\bigr)(\mu_0) \sim \CL_0 f(x,v) + {f(x,v) \over \eps} f'(x) \int_0^\infty y \eta_\eps(dy)\;,
%\end{equ}
%where, informally, $\eta_\eps(dy)$ denotes the probability that the 
%particle creates an offspring $(\bar x, \bar v)$ such that
%$a^{-1}v + \bar x \in [y,y+dy]$. As already seen in Section~\ref{sec:contLimit}, one has 
%\begin{equ}
%\int_0^\infty \psi(y)\, \eta_\eps(dy) \sim a\sqrt \eps\int_0^\infty \int_0^z \psi(\sqrt \eps y)\, dy\,\nu(dz)\;, 
%\end{equ}
%for fixed test functions $\psi$, so that 
\begin{equ}[e:finalWrong]
\eps^{-1} \bigl(T_\eps F - F\bigr)(\mu_0) \approx A_0 f(x,v) + {a \over 2} f(x,v) f'(x) \int_0^\infty y^2\,\nu(dy) \;,
\end{equ}
which is always \textit{different} from \eref{e:genBF}, and is actually what we would have obtained from
the wrong guess \eref{e:wrongGuess}.

A possible  reason for this discrepancy is that, while the Markov semigroup of the Brownian fan does
indeed preserve test functions of the type \eref{e:shapeTF}, we do not expect this to be true of the
Markov operator $T_\eps$. Instead, the ``correct'' space of test functions for $T_\eps$ is 
of the same type, but the function $f$ should have a ``boundary layer'' near $\d\CM$.
% which 
%would then be responsible for replacing $y$ by $\eps G(y/\eps)$ in \eref{e:discrete2}, thus eventually leading
%to the correct constant in \eref{e:finalWrong}.

\begin{remark}
Another reason why the generator of the Brownian fan is not such a useful object is that many 
seemingly innocent observables, like for example the total number $\CN$ of particles, do \textit{not} belong 
to its domain. This follows from the fact that if we consider again a simple initial 
condition $\mu_0$ as above, then $\E \CN(\mu_\eps) - 1 \approx \sqrt \eps$ for small $\eps$. The reason
 the total number of particles nevertheless remains finite (at least for fixed times) is that there are
exceptional states where one or more particles are very near $\d\CM$ and for which
$\E \CN(\mu_\eps) - 1 \approx -\CO(1)$. 
\end{remark}

%\section{Convergence to the Brownian fan}
%\label{sec:convergence}

The remainder of the article is devoted to providing a rigorous proof of the fact that, in the situation of the
previous two sections, the process given by Algorithm~\ref{tdmc}
converges to the Brownian fan in $\CC([0,T], \ell^p(\CM))$ for any $T>0$ and $p \le 1$.
The overall strategy of the proof is classical: we first prove a tightness result in Section~\ref{sec:tightness}
and then show that finite-dimensional marginals converge to those of the Brownian fan in Section~\ref{sec:convFan}. 

Difficulties arise on two fronts. First, to prove the tightness result, it is convenient to have
uniform moment bounds on the number of particles at fixed time for the approximating system.
These turn out to be much more difficult to obtain for the approximating system than for the 
Brownian fan, which is mainly due to a lack of uniform exponential bounds. 
A second difficulty arises in the proof that finite-dimensional distributions converge to those of the 
Brownian fan. While it is intuitively clear that those excursions that survive for times of
order $\CO(1)$ do converge to suitably normalised Brownian excursions, this result is rather technical
and, surprisingly, does not seem to appear in the literature.
Furthermore, no convergence result holds for the typical excursions which die very early. 
We therefore also need to argue that, both at the level
of Algorithm~\ref{tdmc} and at the level of the Brownian fan, these small excursions do not matter 
in the limit.

\section{Tightness}
\label{sec:tightness}

As in the previous section, we restrict ourselves to the particular case when the underlying Markov process is
given by a rescaled random walk, namely
\begin{equ}[e:RW]
y_{(k+1)\eps} = y_{k\eps}  + \sqrt \eps \xi_{k+1}\;,
\end{equ}
where the $\xi_k$ are i.i.d.\ random variables with distribution $\nu$ having some exponential moments, 
and where the potential $V$ is given by a linear function, $V(x) = -ax$. Our aim
is to show that as $\eps \to 0$, the sequence of birth and death processes obtained by running
Algorithm~\ref{tdmc} is tight in a state space $\CX$, which we will now describe.

\subsection{Formulation of the tightness result}

With the construction of the previous section in mind, we choose as our state space
 $\CX = \ell^p(\CM)$, where  $\CM = \{(x,v,n) \in \R^2 \times \N \,:\, v > -ax\}$, 
and $p \le 1$ is arbitrary. Here, the coordinate $n$ is used to keep track of the generation of a particle:
direct offspring of a particle from the $n$th generation belong to the $(n+1)$st generation.
We extend the Euclidean distance to $\R^2 \times \N$ by additionally 
postulating that the distance between particles
belonging to different generations is given by the sums of the distances of the two particles to $\d\CM$. 
The boundary $\d \CM$ is given as before by $\d \CM = \{(x,v,n)\,:\, v = -ax\}$.  We will use capital letters for elements of $\CM$ to differentiate them from elements of $\R$.

In order to formulate our result, we will make use of the following notation. For $t = k\eps$ with $k$ an
integer, we denote by $\mu_t^\eps$ the empirical measure of the particles alive at time $t$, and by $N_t$
the number of such particles. Sometimes, it will be convenient to consider the
particles instead as a collection of elements  $X_t^{(j)}\in\CM$, so that we write
\begin{equ}
\mu_t^\eps = \sum_{j=1}^{N_t} \delta_{X^{(j)}_t}\;.
\end{equ}
We do not specify how exactly we order the particles, as this is completely irrelevant for our purpose.
For $t \in (k\eps, (k+1)\eps)$, we define $\mu_t^\eps$ by using the ``linear interpolation'' procedure \eref{e:linearInterp}, 
setting
\begin{equ}
\mu_t^\eps = L_s(\mu_{k\eps}^\eps, \mu_{(k+1)\eps}^\eps)\;,\qquad s = \eps^{-1}t - k\;.
\end{equ}
The interpolation procedure $L_s$ is a very minor modification from the one described above, in the sense
that we only connect particles belonging to the same generation. 
In this way, the process $t \mapsto \mu_t^\eps$ has continuous trajectories for every $\eps$.
The main result of this section is as follows:

\begin{theorem}\label{theo:tight}
Let $p \le 1$ and denote by $\CL^\eps$ the law of the process $t \mapsto \mu_t^\eps$ 
described above, viewed as a family of probability measures on $\CC([0,1],\CX)$.
Assume furthermore that there exists $c> 0$ such that $\int e^{c|y|}\nu(dy) < \infty$. 

Then, for any single particle initial condition $\mu_0^\eps = \delta_{X_0}$ with $X_0 \in \CM$,
there exists $\eps_0>0$ such that the family $\{\CL^\eps\}_{\eps \le \eps_0}$ is tight.
\end{theorem}

\begin{proof}
Combining \cite[Theorem~3.6.4]{Dawson} and \cite[Theorem~8.3]{Billingsley}, we
see that, in order to obtain tightness, it is sufficient to show that:
\begin{claim}
\item For every $\delta > 0$, there exists a compact set $K_\delta \subset \CX$ such that
$\P(\mu_t^\eps \in K_\delta) > 1-\delta$, uniformly over $t \in [0,1]$ and $\eps < \eps_0$.
\item There exists $\alpha > 0$ and $C>0$ such that $\E \|\mu_t^\eps - \mu_s^\eps\|_p^q \le C|t-s|^{\alpha q}$,
uniformly over all $s,t\in [0,1]$ and $\eps < \eps_0$.
\end{claim}
The first claim then follows from Proposition~\ref{prop:compact} below,
while the second claim is the content of Proposition~\ref{prop:Kolmogorov}.
\end{proof}

The proof of this result is the content of the remainder of this subsection and goes roughly as follows.
In Section~\ref{sec:boundsNPart}, we obtain a moment bound on the number of particles alive at any fixed time $t \in [0,1]$,
which is uniform in $\eps > 0$. This then allows to obtain the compactness at fixed time in Section~\ref{sec:compactFixed}.
The verification of Kolmogorov's continuity criterion is the content of Section~\ref{sec:Kolmo}.

\subsection{Moment bounds on the number of particles}
\label{sec:boundsNPart}

In the sequel, we will denote by $\K_p(X)$ the $p$th cumulant of a random variable $X$ and by $\K_p(X\,|\, \FF)$ the
same cumulant, conditioned on the $\sigma$-field $\FF$. We will use the important property that $\K_p(X+Y\,|\,\FF) = \K_p(X\,|\,\FF) + \K_p(Y\,|\,\FF)$,
provided that $X$ and $Y$ are independent, conditionally on $\FF$. We will also use the fact that if $X$ is a positive random variable, then
$\K_p(X) \le \E X^p$ and there exists a constant $C$ such that the bound
\begin{equ}[e:boundMoments]
\E X^p \le C\sum_{q \le p} \bigl(\K_q (X)\bigr)^{p/q}\;,
\end{equ}
holds.

Our aim now is to obtain a bound on the cumulants of the number of particles alive at time $t$ which is independent of $\eps$.
We start with an initial configuration containing only one particle, which belongs to generation $0$,
and we set $N^0_t \in \{0,1\}$, depending on
whether or not this particle is still alive at some subsequent time $t$. 
We also define $N^n_t = \mu_t^\eps(\CM^n)$, where $\CM^n = \{(x,v,\ell) \in \CM\,:\, \ell=n\}$, which is
 the number of particles in the $n$th generation that are
alive at time $t$. We also denote by $N^n_{s,t}$ the number of such particles that were created at time $s \le t$.

For any $X_0 = (x,v)$ with $v > -ax$, we write $\E_{X_0}$ for expectations of observables for the process
generated by starting Algorithm~\ref{tdmc} with underlying dynamic \eref{e:RW}, started with one single initial particle 
in generation $0$ at location $x$ with tag $v$.
Finally, for $x \in \R$, we write $\E_x$ for the same expectation, but where the initial particle has tag $v = +\infty$,
meaning that it is ``immortal''.
We then have the following result:

\begin{proposition}\label{prop:numPart}
Consider the situation of Theorem~\ref{theo:tight}.
For every $p \ge 0$, there exist $C_p>0$ and $\eps_0 > 0$ such that, under the rules of Algorithm \ref{tdmc}, the number $N_t = \mu_t^\eps(\CM)$ of particles alive 
at time $t$ satisfies $\E_{x} |N_t|^p \le 2\exp(C_p t)$, uniformly over all $\eps \le \eps_0$ and $x \in \R$.
Furthermore, there exists $\rho > 0$ such that 
 \begin{equ}[e:expOffspring]
 \E_{x} N^n_t \le (n+1)^{t/\rho} 2^{-n}\;,
 \end{equ}
uniformly over all $\eps \le \eps_0$, $n \ge 0$, and $t>0$.

Finally, denoting by $R^\gamma_t$ the number of offspring alive at time $t$ that have never been at distance more than $\gamma$
from $\d \CM$, we have the bound
\begin{equ}
 \E_{x} R^\gamma_t \le C \gamma\;,
\end{equ}
uniformly over $t \in [0,1]$, $\gamma \in (0,1]$, and $\eps \le \eps_0\wedge\gamma^2$.
\end{proposition}

The proof will make use of the following elementary fact where, for $\lambda = n + p$ with $n \in \N$ and $p \in (0,1]$,
we denote by $\CI(\lambda)$ the law of a random variable $Y$ such that $Y= n$ with probability $1-p$ and $Y = n+1$ with probability $p$.
(This is so that $\E Y = \lambda$.)

\begin{lemma}\label{lem:unif}
Let $Y$ be a random variable with law $\CI(\lambda)$.
Then, for any $q \ge 1$ and any $\lambda>0$, one has the bound $\E Y^q \le \lambda + (2\lambda)^q$.
\end{lemma}

\begin{proof}
By inspection, one has
\begin{equ}
\E Y^q = p (n+1)^q + (1-p) n^q\;.
\end{equ}
In the case $n=0$, one then has $\E Y^q = \lambda$, so the statement is true. For $n \ge 1$, one uses the 
fact that both $n+1$ and $n$ are bounded by $2\lambda$, and the claim follows at once.
\end{proof}

\begin{proof}[of Proposition~\ref{prop:numPart}]
We restrict ourselves to times that are integer multiples of $\eps$. Furthermore, from now on, we fix 
an initial condition $x$,
so that we just write $\E$ instead of $\E_x$.
We also denote by $\FF^n$ the $\sigma$-algebra containing all information pertaining to particles
in generations up to (and including) $n$. 
With this notation at hand, we obtain for $N^n_t$ the bound
\begin{equs}
\E\bigl( N^n_t\bigr)^p &\lesssim \E \sum_{q=1}^{p} \bigl(\K_q (N^n_t\,|\, \FF^{n-1})\bigr)^{p/q}\\
&= \E \sum_{q=1}^{p} \Bigl(\sum_{\eps \ell \le t} \K_q (N^n_{\eps\ell, t}\,|\, \FF^{n-1})\Bigr)^{p/q}\\
&\lesssim \sum_{q=1}^{p} \E \Bigl(\sum_{\eps \ell \le t} \E \bigl(|N^n_{\eps\ell, t}|^q\,|\, \FF^{n-1}\bigr)\Bigr)^{p/q}\;,\label{e:boundNkp}
%&\le \sum_{q=1}^{p} \Bigl({\eps \over t}\Bigr)^{{p\over q}-1} \sum_{\eps \ell \le t} \E \bigl( \E (N_k^q(\eps\ell, t)\,|\, \FF_{k-1})\bigr)^{p/q}\;,
\end{equs}
where we used \eref{e:boundMoments} in the first step and the independence of the offspring in the second step.
In the above expression, $\ell$ takes only integer values.
Note now that 
\begin{equ}[e:boundSum]
\sum_{\eps \ell \le t} {\eps \over \sqrt{t(t + \eps -\eps\ell)}}  \le C\;,
\end{equ}
uniformly over all $t\ge \eps$. As a consequence, for any positive sequence $a_n$ and any power $r \ge 1$, one has the bound
\begin{equs}
\Bigl(\sum_{\eps \ell \le t} a_{\ell} \Bigr)^r &= \Bigl(\sum_{\eps \ell \le t}{\eps \over \sqrt{t(t + \eps -\eps\ell)}}  \,\tilde a_{\ell} \Bigr)^r 
\lesssim \sum_{\eps \ell \le t}{\eps \over \sqrt{t(t + \eps -\eps\ell)}}\,  \tilde a^r_{\ell} \\
&=  \sum_{\eps \ell \le t}{t^{r-1\over 2}(t + \eps -\eps\ell)^{r-1\over 2} \over \eps^{r-1}}\,  a^r_{\ell} \;,
\end{equs}
where we have set $\tilde a_\ell = \eps^{-1}  \sqrt{t(t + \eps -\eps\ell)}\, a_\ell$ in the intermediate steps. Applying this inequality to \eref{e:boundNkp}, we obtain
\begin{equ}[e:myBound]
\E\bigl( N^n_t\bigr)^p \lesssim  \sum_{q=1}^{p} \sum_{\eps \ell \le t} {t^{p-q\over 2q}(t + \eps -\eps\ell)^{p-q\over 2q} \over \eps^{p-q\over q}} \E \Bigl(\E \bigl( |N^n_{\eps\ell, t}|^q\,|\, \FF^{n-1}\bigr)\Bigr)^{p/q}\;.
\end{equ}

Denote now by $M^{n,j}_s$ the number of particles in the $n$th generation created at time $s$ 
by the $j$th particle from the $n-1$st generation. We write 
$\GG_s$ for the $\sigma$-algebra generated by this additional data. 
Each of these particles yields a contribution to $N^n_{s,t}$ of either $1$ or $0$, depending whether it
survives or not.  
Furthermore, these contributions, which we will denote by $S^{n,j,i}_{s,t}$, 
are all independent and, for the same value of $j$, they are also identically distributed. By definition, we thus have the identity
\begin{equ}
N^n_{\eps\ell, t} =  \sum_{j=1}^{N^{n-1}_{\eps(\ell-1)}} \sum_{i=1}^{M^{n,j}_{\eps\ell}} S^{n,j,i}_{\eps\ell,t}\;.
\end{equ}

We now twice make use of the inequality
\begin{equ}[e:Holder]
\Bigl(\sum_{j=1}^m a_j\Bigr)^q \le m^{q-1} \sum_{j=1}^m a_j^q\;,
\end{equ}
which is valid for any $q \ge 1$, $m \ge 0$ and sequence of positive numbers $a_j$.
This yields the bound 
\begin{equ}[e:boundN]
\bigl(N^n_{\eps\ell,t}\bigr)^q \le (N^{n-1}_{(\ell-1)\eps})^{q-1} \sum_{j=1}^{N^{n-1}_{(\ell-1)\eps}} \bigl(M^{n,j}_{\eps\ell}\bigr)^{q-1} \sum_{i=1}^{M^{n,j}_{\eps\ell}} S^{n,j,i}_{\eps\ell,t}\;.
\end{equ}
 Note that since $S^{n,j,i}_{\eps\ell,t}$ can only take
the values $0$ or $1$, raising it to the power $q$ makes no difference.
By the ``gambler's ruin theorem'' \cite[Thm~5.1.7]{RWLawler}, we have  the bound
\begin{equ}[e:boundSurv]
\E \bigl(S^{n,j,i}_{\eps\ell,t}\,|\, \FF^{n-1}\vee \GG_{\eps\ell}\bigr) \le C\sqrt{\eps}{\xi^j +1 \over \sqrt{t+\eps - \ell \eps}}\;,
\end{equ}
where $\sqrt{\eps}\xi^j$ denotes the step performed by the $j$th particle of the $(n-1)$st generation
between times $\eps(\ell-1)$ and $\eps\ell$.
Regarding the number of offspring $M^{n,j}_{\eps\ell}$, it follows from the definition of the algorithm
that its distribution is given by $\CI(\exp(a\sqrt{\eps}\xi^j_+) - 1)$, where $\xi^j_+$ denotes the positive part
of $\xi^j$. Combining this with \eref{e:boundN}, Lemma~\ref{lem:unif}, and \eref{e:boundSurv}, we
thus obtain the bound
\begin{equ}
\E \bigl(|N^n_{\eps\ell,t}|^q \,|\, \FF^{n-1}\bigr) \lesssim |N^{n-1}_{(\ell-1)\eps}|^{q-1}\sum_{j=1}^{N^{n-1}_{(\ell-1)\eps}} \sqrt{\eps} { \bigl(e^{a\sqrt{\eps}\xi^j_+} - 1\bigr) + \bigl(e^{a\sqrt{\eps}\xi^{j}_+} - 1\bigr)^q  \over \sqrt{t+\eps - \eps \ell} }\bigl(\xi^j + 1\bigr)\;.
\end{equ}
In order to simplify this expression, we use the fact that, for $x \ge 0$, 
there exists a constant $C$ depending on $q$ such that
\begin{equ}
e^x - 1 \le x e^x\;,\qquad (e^x - 1)^q \le C x e^{qx}\;,
\end{equ}
for every $q \ge 1$.
This yields
\begin{equ}
\E \bigl(|N^n_{\eps\ell,t}|^q \,|\, \FF^{n-1}\bigr) \lesssim |N^{n-1}_{(\ell-1)\eps}|^{q-1}  \sum_{j=1}^{N^{n-1}_{(\ell-1)\eps}} \eps\, {\one_{\xi^j \ge 0} e^{aq\sqrt{\eps}\xi^j}  \over \sqrt{t+\eps - \eps \ell} }\bigl(\xi^j + 1\bigr)^2\;.
\end{equ}
Using \eref{e:Holder} once again, we get the bound
\begin{equ}
\bigl(\E \bigl(|N^n_{\eps\ell,t}|^q \,|\, \FF^{n-1}\bigr)\bigr)^{p\over q} \lesssim |N^{n-1}_{(\ell-1)\eps}|^{p-1}  \sum_{j=1}^{N^{n-1}_{(\ell-1)\eps}} \eps^{p\over q} {\one_{\xi^j \ge 0}  e^{ap\sqrt{\eps}\xi^j}  \over (t+\eps - \eps \ell)^{p\over 2q} }\bigl(\xi^j + 1\bigr)^{2p\over q}\;.
\end{equ}
At this stage, we note that, conditional on the state of the system at time $\eps (\ell-1)$ 
the steps $\xi^j$ are all independent and identically distributed with law $\nu$.
Setting
\begin{equ}
P_\eps = \int_0^\infty e^{ap \sqrt \eps z}  (z+1)^{2p\over q}\nu(dz)\;, 
\end{equ}
it follows that 
\begin{equ}
\E \bigl(\E \bigl(|N^n_{\eps\ell,t}|^q \,|\, \FF^{n-1}\bigr)\bigr)^{p\over q} \lesssim {\eps^{p\over q}  P_\eps \over (t+\eps - \eps \ell)^{p\over 2q}}\E |N^{n-1}_{(\ell-1)\eps}|^p \;.
\end{equ}
Inserting this into \eref{e:myBound} yields for $t \le 1$ the bound
\begin{equ}
\E\bigl( N^n_t\bigr)^p \lesssim  P_\eps \sum_{\eps \ell \le t} {\eps \over \sqrt{t + \eps -\eps\ell}} \E |N^{n-1}_{(\ell-1)\eps}|^p\;.
\end{equ}
On the other hand, since we assumed that the initial particle is immortal, 
$N^{n-1}_s$ is stochastically increasing as a function of $s$, so that we obtain from \eref{e:boundSum} the bound
\begin{equ}
\E\bigl( N^n_t\bigr)^p \lesssim  P_\eps  \sqrt t \,\E \bigl(N^{n-1}_t\bigr)^p\;.
\end{equ}
Since, by the exponential integrability assumption on $\nu$, 
there exists $C > 0$ such that, for $\eps$ small enough, $P_\eps \le C$ and since $N^0_t = 1$, we conclude that there exists some $\lambda_p$ such that,
for all $n \ge 1$, one has the bound
\begin{equ}[e:boundNkt]
\E\bigl( N^n_t\bigr)^p \lesssim  \bigl(\lambda_p t\bigr)^{n/2} \;.
\end{equ}
It follows that there exists $t_\star > 0$ (depending on $p$) such that $\E (N_t)^p \le 2$ (say), for every $t \le t_\star$.
To show that $\E (N_t)^p < \infty$ for every $t$, we use again \eref{e:Holder}, combined with the Markov property of the process
to conclude that $\E (N_{t+t_\star})^p \le 2 \E (N_t)^p$, from which the first claim follows.

To get \eref{e:expOffspring}, observe that if we choose $\rho$ small enough so that $\lambda_1 \rho \le 1/4$, the bound with $t \le \rho$ 
follows from \eref{e:boundNkt}. Denote now by $\CF_t$ the filtration generated by all events up to time $t$.
For $t$ and $\rho$ that are multiples of $\eps$  it follows from the Markov property that 
\begin{equ}
\E \bigl(N^n_t \,|\, \CF_{t-\rho}\bigr) \le \sum_{\ell=0}^n N^\ell_{t-\rho}\, \E N^{n-\ell}_{\rho}\;
\end{equ}
where the expectation on the right is taken with respect to an initial condition with one single immortal particle.
The claimed bound for arbitrary $t>0$ therefore follows by induction.

It remains to obtain the bound on $R^\gamma_t$. Denote by $R^\gamma_{s,t}$ the number of offspring
contributing to $R^\gamma_t$ that are created at time $s < t$, so that $R^\gamma_t = \sum_{\eps k \le t} R^\gamma_{k\eps,t}$.
Denote now by $Q_{z,t}^{\gamma}$ the probability that, after time $t$, the random walk \eref{e:RW} with initial condition $\sqrt \eps z$ 
has never exited the interval $[0,\gamma]$. It then follows that
\begin{equ}[e:boundR]
\E \bigl(R^\gamma_{k\eps,t}\,|\, \CF_{k\eps}\bigr) = a\sqrt \eps N_{k\eps} \int_0^\infty \int_0^z e^{a\sqrt \eps z} Q^{\gamma}_{y,t-k\eps}\,dy\,\nu(dz)\;.
\end{equ}
This is because, by Step~3 of Algorithm~\ref{tdmc}, the expected number of offspring created by a particle performing an upward step of size $\sqrt \eps z$ is given by $e^{a\sqrt \eps z} - 1$ while, if we denote by
$\sqrt \eps y$ the distance between the starting point of the offspring and the ``wall'' below which it is killed,
the law of $y$ is given by $c^{-1} \int_0^z e^{a\sqrt \eps y} \,dy$ where $c$ is a normalizing constant.
Since $c = (e^{a\sqrt \eps z} - 1)/a\sqrt \eps$, \eref{e:boundR} follows.

Using again the gambler's ruin theorem, combined with the Markov property of the random walk, we obtain for $Q_{z,t}^{\gamma}$
the bound
\begin{equ}
Q_{z,t}^{\gamma} \le C {(z+1) \sqrt \eps \over \sqrt{t+\eps}} \bar Q^\gamma_{t/2}\;,
\end{equ}
where $\bar Q^\gamma_t$ is the probability that the random walk \eref{e:RW}, starting at the origin, stays within $[-\gamma, \gamma]$ up to time $t$.
Using the scaling properties of Brownian motion, combined with \cite{ConvRateBM}, we conclude that one has the bound
\begin{equ}
\bar Q^\gamma_t \le 1 \wedge {C\delta^{2q} \over t^q}\;,
\end{equ}
for every $q > 0$, so that 
\begin{equ}
Q_{z,t}^\gamma \lesssim {(z+1) \sqrt \eps \over \sqrt{t+\eps}} \Bigl(1 \wedge {\delta^{2q} \over t^q}\Bigr)\;.
\end{equ}
Inserting this bound into \eref{e:boundR}, using the previously obtained bounds on $N_{k\eps}$, and summing over all $k$, the claim follows at once.
\end{proof}

\begin{remark}
It follows from Theorem~4.1 in \cite{DMC} that, in the particular case when the initial tag $v$ is distributed
according to the logarithm of a uniformly distributed random variable, one has the identity
$\E N_t = \E e^{a y_t}$, where $y_t$ is the rescaled random
walk with steps $\nu$. The (at least one-sided) exponential integrability assumption on $\nu$ is therefore a necessary assumption
in order to obtain any kind of moment bounds on $N_t$.
\end{remark}

\begin{remark}
Even if we assume that $\nu$ has Gaussian tails and despite the result previously 
obtained for the Brownian fan in Theorem~\ref{theo:Nt},
it is \textit{not} true in general that $N_t$ defined by Algorithm \ref{tdmc}, has uniform exponential
moments as $\eps \to 0$. This is because even for the first step, the probability that the original particle 
performs a step of order $\eps^{-p}$
is of order $\exp(-\eps^{-2p-1})$. If this were to happen, the number of offspring created in this way would be of order $\exp(\eps^{-p})$,
which immediately shows that exponential moments blow up as $\eps \to 0$. 
\end{remark}

\subsection{Compactness at any fixed time}
\label{sec:compactFixed}

We now show that we can find a compact set $K_\delta$ such that  $\mu_t^\eps$
belongs to $K_\delta$ with high probability, uniformly over $\eps$ and $t\in[0,1].$
Our first ingredient for this is the following moment bound:

\begin{proposition}\label{prop:moments}
Consider the setting of Theorem~\ref{theo:tight}.
Then, there exist constants $C$ and $\eps_0$ such that, for every $t \in [0,1]$ and every $\eps \le \eps_0$, the bound
\begin{equ}[e:wantedBound]
\E_{X_0}\Bigl(\int_\CM |x-x_0|^{2p} \mu_t^\eps(dx,dv,dk)\Bigr) \le C t^p\;,
\end{equ}
holds uniformly over all initial conditions $X_0 = (x_0, v_0, 0) \in \CM$.
\end{proposition}

\begin{proof}
We restrict ourselves to the case when $t$ is an integer multiple of $\eps$, since the bound on the remaining times
easily follows from our interpolation procedure.
Furthermore, we can restrict ourselves to the case when the initial particle is immortal, which formally corresponds
to setting $v = +\infty$. By translation invariance, we also restrict ourselves to the case where the initial particle
is located at the origin, and we denote the corresponding expectation by $\E_0$.

It follows from Theorem~4.1 in \cite{DMC} that if we choose $v = {1\over a} \log u$, where $u$ is drawn uniformly from $[0,1]$
and denote by $\tilde \mu_t^\eps$ the corresponding process, then one has the identity
\begin{equ}
\E_{0}\Bigl(\int_\CM x^{2p} \tilde\mu_t^\eps(dx,dv,dk)\Bigr) = \E_0 \Bigl(e^{a y_t} \bigl(y_t\bigr)^{2p}\Bigr)\;,
\end{equ}
where $y_t$ is the simple random walk \eref{e:RW} started at the origin. It follows immediately from the 
exponential integrability of $\nu$ that this quantity is bounded by $C t^p$ for $t \le 1$ and for
$\eps$ small enough. 

On the other hand, one can realise the process $\mu_t^\eps$, which starts with an immortal initial particle, in the following way:
\begin{enumerate}
\item Consider the process $\tilde \mu_t^\eps$, where the $v$-component is as above.
\item When the initial (generation $0$) particle is killed, replace it instantly by an immortal particle starting at the current location.
\end{enumerate}
Let $x^0_t$ denote the trajectory of the initial particle and let $s$ be the time at which the initial particle is killed and replaced.
Let $P_t = \P(s \le t)$.
By the construction just outlined, we have the recursion relation
\begin{equ}
F_{t,0} = E_t + P_t\, \E \bigl(F_{t-s, x^0_s}\,|\, s \le t \bigr)\;,
\end{equ}
where we set
\begin{equs}
F_{t,x^0_s} &= \E_{0}\Bigl(\int_\CM (x + x^0_s)^{2p} \mu_t^\eps(dx,dv,dk)\Bigr) \;,\\
E_t &= \E_{0}\Bigl(\int_\CM x^{2p} \tilde\mu_t^\eps(dx,dv,dk)\Bigr)\;. 
\end{equs}
(Remember that the difference between $E_t$ and $F_{t,0}$ is that in order to compute $F$, we start with an immortal particle.)
Setting $F_t = F_{t,0}$, using the fact that, for every $\delta > 0$ one can find $C_\delta$ such that 
$(x+x^0_s)^{2p} \le C_\delta (x^0_s)^{2p} + (1+\delta)x^{2p}$, and 
recalling that $E_t \le Ct^p$, we deduce that
\begin{equ}
F_t \le Ct^p + C_\delta P_t\E \bigl((x^0_s)^{2p} \E_{x^0_s} \bigl(N_{t-s}\bigr)\,|\, s \le t\bigr)\bigr) + (1+\delta) P_t\, \E\bigl(F_{t-s}\,|\, s \le t\bigr)\;,
\end{equ}
where $N_t$ is the number of particles alive at time $t$ for the system started with an immortal 
particle. Note now that, for $t \le 1$, we know from Proposition~\ref{prop:numPart} that the expected number of particles
alive at any given time is bounded by some constant uniform in $\eps$. Since this bound is also
uniform in the initial condition, we have
\begin{equ}
P_t \E \bigl((x^0_s)^{2p} \E_{x^0_s} \bigl(N_{t-s}\,|\, s \le t\bigr)\bigr) \lesssim P_t \E \bigl((x^0_s)^{2p}\,|\, s\le t\bigr) = \E \bigl((x^0_s)^{2p} \one_{s \le t}\bigr)\;.
\end{equ}
%Furthermore, 
%recall that the random variable $y$ is nothing but the location of the random walk $Y$ when it is less than $a^{-1}v$
%for the first time.
 Defining $\tilde s = s \wedge t$, we then obtain the trivial bound
\begin{equ}
\E \bigl((x^0_s)^{2p} \one_{s \le t}\bigr) \le \E_0 \bigl|y_{\tilde s}\bigr|^{2p} \le \E_0 \bigl|y_t\bigr|^{2p} 
\lesssim t^p\;,
\end{equ}
where we made use of the fact that $|y_t|^{2p}$ is a submartingale to obtain the second inequality.

Setting now $\bar F_t = \sup_{s \le t} F_s$, we can combine all of these bounds to get the inequality
\begin{equ}
\bar F_t \le C_\delta t^p + (1+\delta) P_t \bar F_t\;.
\end{equ}
Since one can check (for example by again using the fact that the random walk approximates a Brownian 
motion for small $\eps$) that $\sup_{\eps \le 1} \sup_{t\le 1} P_t = \sup_{\eps \le 1} P_1 < 1$, the claim follows at once
by choosing $\delta$ sufficiently small.
\end{proof}

This result can now be used to deduce the announced uniform tightness result over fixed times:

\begin{proposition}\label{prop:compact}
Consider the setting of Theorem~\ref{theo:tight}.
Then, for every $\delta>0$ there exists a compact set $K_\delta \subset \CX$ such that 
$\P_{X_0}(\mu_t^\eps \in K_\delta) > 1-\delta$, uniformly over $t \in [0,1]$, $\eps \le \eps_0$
and $X_0 \in \CM^0$.
\end{proposition}

\begin{proof}
For any $n\in \Z_+$ and $R  \in \R_+$, denote
\begin{equ}
 \CM^{[n]} = \{(x,v,\ell) \in \CM\,:\, \ell \le n\}\;,\quad
B_R = \{(x,v,\ell) \in \CM\,:\, |x| \vee |v| \le R\}\;.
\end{equ}
For any such $n$ and $R$ and for $m \in \Z_+$, we then denote by $K_{n,m,R} \subset \CX$ the set of all
integer-valued measures $\eta$ on $\CM$ such that $\eta(\CM \setminus  \CM^{[n]}) = 0$, $\eta(\CM) \le m$, and $\eta(\CM \setminus B_R) = 0$.
Since these sets are obviously compact, it suffices
to find, for every $\delta > 0$, sufficiently large values $n$, $R$ and $m$ so that $\P_{X_0}(\mu_t^\eps \in K_{n,m,R}) > 1-\delta$
uniformly over all $\eps \le \eps_0$ and $t \in [0,1]$.

Note that $K_{n,m,R} \subset K_n^1 \cap K_m^2 \cap K_R^x \cap K_R^v$ where, for
$K_n^1$ and $K_m^2$ we only enforce the conditions involving $n$ and $m$ respectively. The set
$K_R^x$ consists of configurations such that the $x$-coordinate of every particle is less than $R$
in absolute value, while $K_R^v$ enforces that the $v$-coordinate be less than $R$. 

It follows immediately from \eref{e:expOffspring} that there exists $\gamma>0$ and a constant $C$ (depending in principle on the
 time we consider, but it can be chosen uniformly over $t \in [0,1]$), such that
\begin{equ}
\P_{X_0}(\mu_t^\eps  \not\in K_n^1) \le C e^{-\gamma n}\;,
\end{equ}
for every $n \ge 0$. Similarly, it follows from the moment bounds on $N_t$ obtained in Proposition~\ref{prop:numPart} that, for every $p>0$, there exists $C$ such that
\begin{equ}
\P_{X_0}(\mu_t^\eps  \not\in K_m^2) \le {C\over m^p}\;.
\end{equ}
In order to get a bound on the $x$-coordinate of the particles, we 
combine Proposition~\ref{prop:moments}
with Chebyshev's inequality, so that
\begin{equ}
\P_{X_0}(\mu_t^\eps  \not\in K_R^x) \le {C\over R^2}\;.
\end{equ}
It remains to obtain a bound on the probability of not being in $K_R^v$. For this, we use the fact that 
on the one hand, the label of a particle always satisfies $v > -ax$. On the other hand, 
any descendant of the initial particle will always satisfy $v \le -a\sup_{t \le 1} x^{0}_t$, where here we denote by $x^{0}_t$
the position of the original particle at time $t$. Since this particle was assumed to start at the origin, 
we obtain $v^2 \le a^2 x^2 +  a^2\sup_{t \le 1} (x^0_t)^2$, so that we obtain the bound
\begin{equ}
\P_{X_0}(\mu_t^\eps  \not\in K_R^v) \le {C\over R^2}\;,
\end{equ}
just as above. Combining all of these bounds, the claim follows by choosing
$n$, $R$ and $m$ large enough.
\end{proof}

\subsection{Kolmogorov criterion}
\label{sec:Kolmo}

The aim of this section is to obtain the following bound on the time regularity of our process:

\begin{proposition}\label{prop:Kolmogorov}
For every $p \le 1$ and $q \ge 1$, there exists a constant $C$ such that
\begin{equ}[e:KolmoBound]
\E_{X_0} \|\mu_\delta^\eps- \delta_{X_0}\|_p^q \le C\delta^{p q/2}\;,
\end{equ}
where $X_0\in \CM^0$ is an initial condition with only one particle in the system, and $\delta < 1$ is an integer multiple of $\eps$.
\end{proposition}

\begin{remark}
As usual, the precise location of the 
particle is irrelevant by translation invariance, so the above bound is uniform over all choices of $X_0$.
\end{remark}

Before we proceed to the proof of Proposition~\ref{prop:Kolmogorov}, we observe that the bound \eref{e:KolmoBound}
implies that Kolmogorov's criterion holds for the process over a fixed interval of time:

\begin{corollary}\label{cor:Kolmogorov}
For every $p \le 1$ and $q \ge 1$,
one has the bound
\begin{equ}[e:KolmoBound2]
\E_{X_0} \|\mu_{t+\delta}^\eps- \mu_t^\eps\|_p^q \le C\delta^{p q/2}\;,
\end{equ}
uniformly over all $\delta, t, \eps \in [0,1]$.
\end{corollary}

\begin{proof}
Note first that we can restrict ourselves to the case when $t$ and $\delta$ are integer multiples of $\eps$. Indeed, 
it follows from the definition of $\|\cdot\|_p$ and from \eref{e:boundInter} that, if $k\eps \le s < t \le (k+1)\eps$,
then
\begin{equ}
\|\mu_{s}^\eps- \mu_t^\eps\| \le \eps^{-p} |t-s|^p \|\mu_{k\eps}^\eps- \mu_{(k+1)\eps}^\eps\|\;.
\end{equ}

We then obtain from Proposition~\ref{prop:Kolmogorov} the bound
\begin{equs}
\E_{X_0} \|\mu_{t+\delta}^\eps- \mu_t^\eps\|_p^q &= \E_{X_0} \Bigl(\E_{X_0} \bigl( \|\mu^\eps_{t+\delta}- \mu^\eps_t\|_p^q\,|\, \CF_t\bigr)\Bigr) \\
&\le \E_{X_0} \Bigl(|N_t|^{q-1}\sum_{j=1}^{N_t}\E_{X_t^{(j)}} \|\mu_\delta^\eps- \delta_{X_t^{(j)}}\|_p^q\Bigr) \\
&\le C \delta^{pq/2} \E_{X_0} \bigl(|N_t|^{q}\bigr) \le C \delta^{pq/2}\;,
\end{equs}
where we made use of the Markov property, \eref{e:Holder}, and \eref{e:KolmoBound}.
\end{proof}

\begin{proof}[of Proposition~\ref{prop:Kolmogorov}]
Denote by $X_t$ the location at time $t$ of a single particle starting at $X_0 = (x_0,v_0,0)$ and evolving under the
rescaled random walk stopped when it reaches the boundary of $\CM$.   Denote by $x_t$ the position in $\R$ corresponding to $X_t$.
From the properties of the random walk and the definition of $\|\cdot\|_p$, 
for every $q>1$ there exists a constant $C$ such that
the bound
\begin{equ}
\E_{X_0} \|\delta_{X_\delta}- \delta_{X_0}\|_p^q \le C\delta^{qp/2}\;,
\end{equ}
holds, independently of the initial condition and independently of $\eps \le 1$.

Let us now bound the contribution from the descendants of the initial particle.
Ordering the particles alive at time $t$ in such a way that the original particle has label $1$
(if it is not alive anymore, we consider it as being located on the boundary, where it was stopped), we
have the bound
\begin{equ}
\|\delta_{X_\delta}- \mu_\delta^\eps\|_p \le \sum_{j = 2}^{N_{\delta}} d_p(X_\delta^{(j)}, \d \CM)\;,
\end{equ}
so that
\begin{equs}
\E_{X_0} \|\delta_{X_\delta}- \mu_\delta^\eps\|_p^q &\le \E_{X_0} \Bigl(\sum_{j = 2}^{N_{\delta}} d_p(X_\delta^{(j)}, \d \CM)\Bigr)^q \\
&\le \biggl(\Bigl(\E_{X_0} (N_{\delta}-1)^{2q-1}\Bigr) \Bigl(\E_{X_0} \sum_{j = 2}^{N_{\delta}} d_p^{2q}(X_\delta^{(j)}, \d \CM)\Bigr)\biggr)^{1/2} \\
&\le  C\Bigl(\E_{X_0} \sum_{j = 2}^{N_{\delta}} d_p^{2q}(X_\delta^{(j)}, \d \CM)\Bigr)^{1/2} \;,
\end{equs}
where the second inequality follows from Proposition~\ref{prop:numPart}.
Recall now that if $X_\delta^{(j)} = (x,v,n)$, then one has $d_p(X_\delta^{(j)}, \d \CM) = 
|x-v/a|^p$, where $v/a$ is guaranteed to take values between $\inf_{s \le \delta} x_s$ and $x$.
As a consequence, we have the bound
\begin{equ}
|x-v/a|^p \le |x - \inf_{s \le \delta} x_s|^p \le |x-x_0|^p +  \sup_{s \le \delta} |x_s-x_0|^p\;,
\end{equ}
so that 
\begin{equs}
\E_{X_0} \sum_{j = 2}^{N_{\delta}} d_p^{2q}(X_\delta^{(j)}, \d \CM)
&\lesssim \E_{X_0}\Bigl( N_{\delta} \sup_{s \le \delta} |x_s-x_0|^{2pq}\Bigr)\\
&\quad + \E_{X_0} \int |x-x_0|^{2pq} \mu_\delta^\eps(dx,dv,dk) \lesssim \delta^{pq}\;,
\end{equs}
where the last bound is a consequence of Propositions \ref{prop:numPart} and \ref{prop:moments},
as well as standard bounds on the supremum of a random walk.
The claim now follows at once.
\end{proof}

\section{Convergence of fixed-time distributions}
\label{sec:convFan}

In this section, we show that any limiting process obtained from the tightness result of the previous section
necessarily coincides with the Brownian fan constructed in Section~\ref{sec:defFanMP}. 
With the notations of that section at hand, our convergence result can be formulated as follows.

\begin{theorem}\label{theo:finalConv}
Consider the setting of Theorem~\ref{theo:tight} with an initial condition $X_0 = (x,v,0) \in \CM$,
and consider $\mu_t$ as above with the ``initial condition'' for $\mu^{[\infty]}$
given by a Brownian motion starting at $x$, killed when it reaches the level $-v/a$.

Then, for every $p \in (0,1]$, there exists a version of the process $\{\mu_t\}_{t \ge 0}$ which is a 
continuous Markov process with
values in $\ell^p(\CM)$. Furthermore, denoting the law of its restriction to the time interval $[0,1]$ by $\CL$, 
the sequence of measures $\CL^\eps$ converges weakly to $\CL$ in $\CC([0,1], \ell^p(\CM))$.
\end{theorem}

\begin{proof}
It suffices to show that, for any fixed collection of times $\{t_1,\ldots, t_k\}$, the law of $\{\mu_{t_i}^\eps\}_{i \le k}$
converges weakly to that of $\{\mu_{t_i}\}_{i \le k}$. Indeed, Corollary~\ref{cor:Kolmogorov} then  implies that
the process $\mu_t$ satisfies Kolmogorov's continuity criterion and therefore has a continuous version.
By Theorem~\ref{theo:tight}, we deduce weak convergence in $\CC([0,1], \ell^p(\CM))$ from the convergence 
of marginals. Using the Markov property, the superposition property of the process, and
Proposition~\ref{prop:numPart}, we reduce ourselves to the case $k=1$ with an initial condition consisting of one single
particle.

Denote now by $\mu^{\eps,[\infty]}$ the random integer-valued measure on $\CE$ obtained 
by running Algorithm~\ref{tdmc} until time $1$. 
Observe that $\mu^{\eps,[\infty]}$ can be built in the following way. For an excursion $w \in \CE$,
we build a random measure $\QQ^\eps(w)$ by the following procedure. For every $k \in \N$ with $\eps k > \s(w)$
and $\eps(k+1) < \e(w) \wedge 1$ we set 
\begin{equ}
\Delta w_k = w_{(k+1)\eps}-w_{k\eps}\;.
\end{equ}
If $\Delta w_k > 0$, we then draw a random variable $N^1_{k}$ with law $\CI(\exp(a\Delta w_k)-1)$
(see \eref{e:noff} and Lemma~\ref{lem:unif}). For $j=1,\ldots, N^1_{k}$, we build i.i.d.\ excursions $w^{k,j} \in \CE$
by the following procedure. First, draw a uniform random variable $u \sim \CU(\exp(-a \Delta w_k),1)$
and set $v = \log u -a w_{k\eps}$. Then, denote by $\{y^{k,j}_{\ell\eps}\}_{\ell=0}^L$ an instance of the random walk 
\eref{e:RW}, started at $w_{(k+1)\eps}$ and stopped just before it becomes smaller than $-v/a$ 
(so that $y^{k,j}_{L\eps} > -v/a$).
The excursion $w^{k,j}$ is then given by $\s(w^{k,j}) = k\eps$, $\e(w^{k,j}) = (L+k+2)\eps$, 
$w^{k,j}_{\eps (\ell+k+1)} = y^{k,j}_{\ell\eps}$ for $\ell \in \{0,\ldots,L\}$,
$w^{k,j}(\ell \eps) = -v/a$ for the remaining integer values of $\ell$, and linear interpolation 
in between integer values. We then set
\begin{equ}
\QQ^\eps(w) = \sum_{k=1}^\infty \sum_{j=1}^{N^1_k} \delta_{w^{k,j}}\;,
\end{equ}
which is the point measure describing the children of the particle with trajectory $w$.
Similarly to before, we build $\mu^{\eps,[\infty]}$ recursively by the following procedure:
\begin{claim}
\item Build an excursion $w^0 \in \CE$ as above, starting at $x$ and stopped at $-v/a$, where $(x,v,0)\in \CM$
is the initial condition appearing in the statement. Set $\mu^{\eps,0} = \delta_{w^0}$.
\item Given $\mu^{\eps,\ell}$, define $\mu^{\eps,\ell+1}$ by
\begin{equ}
\mu^{\eps,\ell+1} = \int_\CE \QQ^\eps(w)\,\mu^{\eps,\ell}(dw)\;,
\end{equ} 
where the $\{\QQ^\eps(w)\}_{w \in \CE}$ are all independent (and independent of the $\mu^{\eps,\ell'}$ with
$\ell' \le \ell$). Note that the integral is actually a finite sum, so the construction makes sense.
\item Set $\mu^{\eps,[\ell]} = \sum_{\ell'=0}^\ell \mu^{\eps,\ell'}$ for positive $\ell$ (including the case $\ell = \infty$).
\end{claim}
If we set
\begin{equ}
\mu_t^\eps = E_t^\star \mu^{\eps,[\infty]}\;,
\end{equ}
where $E_t$ is defined as in Section~\ref{sec:defFanMP},
then the process $\mu_t^\eps$ is indeed equal in law to the process considered in 
Section~\ref{sec:tightness}\footnote{Strictly speaking, the two processes agree only 
at times that are multiples of $\eps$ since the two
interpolation procedures we used may differ when the trajectories of 
two particles from the same generation cross each other. This is irrelevant.}.

Denote now by $\CE_\gamma$ the set of excursions of height at least $\gamma$, namely
\begin{equ}
\CE_\gamma = \{w \in \CE\,:\, \exists t \in \l(w) \;\text{with}\; w_t \ge w(\s(w)) + \gamma\}\;.
\end{equ}
We also write $\mu^{\eps,[n]}_\gamma$ for the measure obtained exactly like $\mu^{\eps,[n]}$, 
but where we replace $\QQ^\eps(w)$ by its restriction $\QQ^\eps_\gamma(w)$ to the set $\CE_\gamma$ at every step, so that $\mu^{\eps,[n]}_\gamma \le \mu_\eps^{[\infty]}$
almost surely. In words, $\QQ^\eps_\gamma$ is obtained from $\QQ^\eps$ by discarding all excursions of
height less than $\gamma$, as well as the descendants of any such excursions.

Combining Propositions~\ref{prop:numPart} and \ref{prop:moments}, we see that, 
for $\eps_0$ small enough, there exist constants $C$ and $\alpha > 0$ such that
one has the bound
\begin{equ}[e:boundUnif]
\sup_{\eps \le \eps_0} \P \bigl(E_t^\star\mu^{\eps,[\infty]} \neq E_t^\star\mu^{\eps,[n]}_\gamma\bigr) \le C \bigl(e^{-\alpha n} + \gamma^p\bigr)\;,
\end{equ}
uniformly over $\eps$ and $t \in [0,1]$.

Following an argument along the lines of the proof of Theorem~\ref{theo:Nt}, 
a similar bound can be shown to hold for 
$\P \bigl(E_t^\star\mu^{[\infty]} \neq E_t^\star\mu^{[n]}_\gamma\bigr)$,
where $\mu^{[n]}_\gamma$ is the recursive Poisson process of depth $n$ constructed like $\mu^{[\infty]}$, 
but with $\CQ(w,\cdot)$ replaced by its restriction $\CQ_\gamma(w,\cdot)$ to $\CE_\gamma$.
As a consequence, we conclude that it is sufficient to show that, for every fixed $\gamma > 0$ and $n \ge 1$,
\begin{equ}[e:wantedConv]
\lim_{\eps \to 0} \CD(\mu^{\eps,[n]}_\gamma) = \CD(\mu^{[n]}_\gamma)\;,
\end{equ}
where $\CD(\cdot)$ denotes the law of a random variable and convergence is to be understood
in the sense of weak convergence on the space $\MM_+(\CE)$ endowed with the Wasserstein-$1$
distance (see \cite{Villani} and Section~\ref{sec:convRPP} below). 

Our aim then is to make use of Theorem~\ref{theo:convPPP} below which gives a general
convergence result for recursive Poisson point processes. The drawback is that 
$\mu^{\eps,[n]}_\gamma$ is itself not a recursive Poisson point process, due to the fact that $\QQ^\eps_\gamma(w)$ is not a
realisation of a Poisson point process. However, it would be one if, in the construction of $\QQ^\eps(w)$, we 
had drawn $N^1_k$ according to a Poisson distribution with mean $\exp(a\Delta w_k)-1$.
Denote by $\QQt(w)$ the Poisson point process obtained in this way, by $\QQtd$ its restriction to $\CE_\gamma$, and by 
$\tilde \mu^{\eps,[n]}_\gamma$ the recursive Poisson point process of depth $n$ obtained
by iterating $\QQtd$. We then claim that it is possible to find a coupling between $\tilde \mu^{\eps,[n]}_\gamma$
and $\mu^{\eps,[n]}_\gamma$ such that
\begin{equ}[e:boundTV]
\P \bigl(\tilde \mu^{\eps,[n]}_\gamma \neq \mu^{\eps,[n]}_\gamma\bigr) \le C_{n,\gamma} \sqrt \eps\;,
\end{equ}
where the constant $C_{n,\gamma}$ depends on $n$ and $\gamma$, but not on $\eps$. Indeed, if we denote by $U_\lambda$ a random variable with law $\CI(\lambda)$ and by $\bar U_\lambda$ a Poisson random variable with mean $\lambda$, then it is straightforward to check that there exists a constant $C$ and a coupling between $U_\lambda$ and $\bar U_\lambda$ such that
\begin{equ}
%\P(U_\lambda \neq \bar U_\lambda) \le C(1\wedge  \lambda^2)\;.
\P(|U_\lambda - \bar U_\lambda| = 1) \le C(1\wedge  \lambda^2)\;,\quad
\P(|U_\lambda - \bar U_\lambda| > 1) \le C (1\wedge \lambda^3)\;.
\end{equ}
Furthermore, by Proposition~\ref{prop:defG}, the probability that the random walk started at $x$ reaches $\gamma$ before becoming negative is bounded from 
above by $C_\gamma (x+\sqrt \eps)$ for some constant $C_\gamma$ depending on $\gamma$.
As a consequence, one can construct a coupling between $\QQ^\eps_\gamma(w)$ and $\QQtd(w)$ such that 
\begin{equ}[e:bd1]
\P \bigl(\QQ^\eps_\gamma(w) \neq \QQtd(w)\bigr) \le C_\gamma \sum_{k\eps \in [0,1]} \bigl|\Delta w_k\bigr|^2 \bigl(\bigl|\Delta w_k\bigr| + \sqrt \eps\bigr)\;.
\end{equ}
Similarly, regarding the total mass of $\QQ^\eps_\gamma$, one has the bound
\begin{equ}[e:bd2]
\E \QQ^\eps_\gamma(w,\CE) = \E \QQ^\eps(w,\CE_\gamma) \le C_\gamma  \sum_{k\eps \in [0,1]} \bigl|\Delta w_k\bigr|^2 e^{c |\Delta w_k|}\;, 
\end{equ}
for some constant $c$. If $w$ is an excursion created by the procedure above, the expected values of 
\eref{e:bd1} and \eref{e:bd2} are bounded by $C_\gamma \sqrt \eps$ and $C_\gamma$ respectively, for a possibly
different constant depending on $\gamma$. Proceeding as in Lemma~\ref{lem:intPPP}, it follows that, if we set
\begin{equ}
F_\eps(w) = 1 + {1\over \sqrt \eps}\sum_{k\eps \in [0,1]} \bigl|\Delta w_k\bigr|^3 + \sum_{k\eps \in [0,1]} \bigl|\Delta w_k\bigr|^2 e^{c |\Delta w_k|}\;,
\end{equ}
we obtain the bound $\E \int_\CE F_\eps(w) \mu^{\eps,[n]}_\gamma(dw) < C_{n,\gamma}$, uniformly over $\eps$.
(But this constant might potentially grow very fast as $\gamma \to 0$ and $n \to \infty$!)
Combining this with \eref{e:bd1}, the bound \eref{e:boundTV} then follows at once.

As a consequence, it is sufficient to show, instead of \eref{e:wantedConv}, that
\begin{equ}[e:convFD]
\lim_{\eps \to 0} \CD(\tilde \mu^{\eps,[n]}_\gamma) = \CD(\mu^{[n]}_\gamma)\;.
\end{equ}
This is the content of Proposition~\ref{prop:convFD} below, which completes the proof.
\end{proof}

The remainder of this section is devoted to the proof of \eref{e:convFD}.
The outline of the proof goes as follows. First, in Section~\ref{sec:convExc}, we show that the law of
a single excursion of the random walk \eref{e:RW}, conditioned on hitting a prescribed
level $\gamma$, converges as $\eps \to 0$ to the Brownian excursion, conditioned to reach $\gamma$.
In Section~\ref{sec:convRPP} we then provide a general criterion for the convergence of
recursive Poisson point processes. Finally, in Section~\ref{sec:convFD} we combine these results in
order to obtain \eref{e:convFD}.

\subsection{Convergence of excursion measures}
\label{sec:convExc}

As before, we denote by $y_t$ the rescaled random walk given by
\begin{equ}
y_{(k+1)\eps} = y_{k\eps}  + \sqrt \eps \xi_{k+1}\;,
\end{equ}
where  the $\xi_k$ are an i.i.d.\ sequence of random variables with law $\nu$. 
 As before, we extend this to arbitrary times 
by linear interpolation.
We also assume that $\nu$ has some exponential moment as before. The aim of this section is to show that if
we start $y_t$ at some initial condition $x \sim \sqrt \eps$ and condition it on reaching a
prescribed height $\gamma$ before becoming negative, then its law converges to that of an unnormalised
Brownian excursion, conditioned to reach height $\gamma$ (call it $w^\gamma$). Throughout this whole section, we will only consider 
the process on a fixed time interval, which for definiteness we choose to be equal to $[0,1]$.

A more precise description of the law $\Q_\gamma$ of $w^\gamma$ is given by the identity
\begin{equ}
\Q_\gamma(\,\cdot\,) = {\int_0^\infty s^{-3/2} g_\gamma(s) \Q_{s,\gamma}(\,\cdot\,)\,ds \over \int_0^\infty s^{-3/2} g_\gamma(s)\,ds}\;,
\end{equ}
where $\Q_{s,\gamma}$ denotes the law of a Brownian excursion of length $s$, conditioned to reach level $\gamma$, and 
\begin{equ}
g_\gamma(s) = \Q_{s,0}(\{w\,:\, \textstyle{\sup_{t \le s} w_t \ge \gamma}\}) = \Q(\{w\,:\, \textstyle{\sup_{t \le 1} w_t \ge \gamma/\sqrt s}\})\;,
\end{equ}
where $\Q$ is the standard It\^o excursion measure. Since $g_\gamma(s) \to 0$ exponentially fast for small $s$,
this does indeed define a probability measure on $\CC(\R_+, \R)$. We then turn this into a probability measure on $\CC([0,1],\R)$
by restriction.

We view $\CC([0,1],\R)$ as a subset of $\CE$ via the injection $\iota\colon \CC([0,1],\R)\hookrightarrow \CE$
given by
\begin{equ}
\s(\iota w) = 0\;,\qquad \e(\iota w) = 1 \wedge \inf\{t > 0\,:\, w(t) = 0\}\;,
\end{equ}
and by setting the path component of $\iota w$ equal to $w$, stopped when it reaches the time $\e(\iota w)$.
We endow $\CE$ as before with the distance $d$ given in \eref{e:distCE}. Since we only deal with excursions 
starting at $0$ and stopped before time $1$, the distance $d$ is equivalent on this set to the (pseudo-)distance
$\bar d$ given by
\begin{equ}
\bar d (w,w') = 1 \wedge \Bigl(|\e(w) - \e(w')| + \sup_{t \in [0,1]} |w_t - w'_t|\Bigr)\;.
\end{equ}

Regarding the space $\CC([0,1],\R)$, we endow it throughout this section with the metric
\begin{equ}
d(w,w') = 1 \wedge \sup_{t \in [0,1]} |w_t - w'_t|\;.
\end{equ}
We furthermore denote by $\|\cdot\|_d$ the Wasserstein-$1$ metric associated with any distance function $d$, which is 
just the dual of the corresponding Lipschitz (semi-)norm. The main theorem of this section is  the following:

\begin{theorem}\label{theo:convLaws}
Let $x > 0$ and denote by $\Q^{\eps}_{z,\gamma}$ the law of the random walk \eref{e:RW}, starting at $z = \sqrt \eps x$ and conditioned
to reach level $\gamma$ before becoming negative. Then, for every $\bar \beta < {1\over 16}$ there exists a constant $C$ such that 
\begin{equ}
\|\Q^{\eps}_{z,\gamma} - \Q_\gamma\|_d \le C \eps^{\bar \beta}\;,\qquad
\|\iota^\star\Q^{\eps}_{z,\gamma} - \iota^\star\Q_\gamma\|_{\bar d} \le C \eps^{2\bar \beta\over 5}\;,
\end{equ}
uniformly over all $\eps \in (0,1]$, all $x \in [0, \eps^{-1/3}]$, and all $\gamma \in [\eps^{1/16},1]$.
\end{theorem}

Our main abstract ingredient in the proof is the following criterion for the convergence
of conditional probabilities when the probability of the set on which the measures are conditioned converges to $0$:

\begin{lemma}\label{lem:convCond}
Let $\mu$ and $\pi$ be two 
 probability measures on some Polish space $\CY$ with metric $d$ and diameter $1$
 and let $\CD_\mu\colon \CY \to [0,1]$ and $\CD_\pi\colon \CY \to [0,1]$. For $\rho > 0$, set
 \begin{equs}
 A_\rho &= \{x \in \CY\,:\, \exists y \in \CY \;\text{with}\; \CD_\pi(y) > \CD_\pi(x) + d(x,y)/\rho\}\;,\\
 \bar A_\rho &= \{x \in \CY\,:\, d(x,A_\rho) \le \rho\}\;. 
 \end{equs}
Assume that $\delta$, $\eps_1$ and $\eps_2$ are such that
\begin{equ}
\int_\CY \CD_\pi(x) \pi(dx) \ge \delta \;,\quad \|\mu - \pi\|_d \le \eps_1\;, \quad \sup_x |\CD_\pi(x) - \CD_\mu(x)| \le \eps_2\;,
\end{equ}
where $\|\cdot\|_d$ denotes the Wasserstein-$1$ distance with respect to $d$. 

Define measures $\tilde\mu$ and $\tilde \pi$ by
\begin{equ}
\tilde \mu(A) = c_\mu\int_A \CD_\mu(x) \mu(dx)\;,\qquad \tilde \pi(A) = c_\pi \int_A \CD_\pi(x) \pi(dx)\;,
\end{equ}
where $c_\mu$ and $c_\pi$ are such that these are probability measures.

Then, the bound
\begin{equ}[e:boundCondMeas]
\|\tilde \mu - \tilde \pi\|_d \le {1\over \delta} \Bigl({3\eps_1\over \rho} + \eps_2 + 2\pi(\bar A_\rho)\Bigr)\;,
\end{equ}
holds for every $\rho \le 1$. In particular, one has $\int_\CY  \CD_\mu(x) \mu(dx)>0$ whenever the right hand side in \eref{e:boundCondMeas} is
strictly smaller than $1$, so that the bound is non-trivial.
\end{lemma}

\begin{proof}
Let $f \colon \CY \to \R$ be a test function such that $\Lip_d(f) \le 1$. Since the diameter of $\CY$ is assumed to be $1$,
we can assume without loss of generality (by possibly adding a constant to $f$) that $\sup_x |f(x)| \le {1\over 2}$. 
Since one has the identity
\begin{equ}
\|\tilde \mu - \tilde \pi\|_d = \sup_{\Lip_d(f) \le 1} \Bigl(c_\mu \int f(x)\CD_\mu(x)\,\mu(dx) - c_\pi \int f(x)\CD_\pi(x)\,\pi(dx)\Bigr)= \sup_{\Lip_d(f) \le 1} \CI_f\;,
\end{equ}
our aim is to bound $\CI_f$, uniformly over $f$. Note first that, by the bound on $f$, 
\begin{equs}
\CI_f &\le {c_\pi\over 2} \Bigl|{1\over c_\pi} - {1 \over c_\mu}\Bigr| + {c_\pi \over 2} \int |\CD_\mu(x)-\CD_\pi(x)|\,\mu(dx) \\
&\qquad + c_\pi \Bigl|\int f(x)\CD_\pi(x)\,\mu(dx)-\int f(x)\CD_\pi(x)\,\pi(dx)\Bigr|\;.
\end{equs}
Note  that the first term is nothing but a particular instance of the last term with $f = {1\over 2}$.
Since the second term is furthermore trivially bounded by $\eps_2 / (2\delta)$, it suffices to bound the last term.

The problem in bounding this term arises of course from the fact that $\CD_\pi$ is not Lipschitz continuous.
For any $\rho > 0$, we can however ``mollify'' this function by setting
\begin{equ}
\CD_\pi^\rho(x) = \sup_{y \in \CY} \Bigl(\CD_\pi(y) - {d(x,y)\over \rho}\Bigr)\;.
\end{equ}
It is then straightforward to check that $\Lip_d (\CD_\pi^\rho) \le \rho^{-1}$, that $\CD_\pi (x) \le \CD_\pi^\rho(x) \le \sup_y\CD_\pi (y)$, and that furthermore 
$\CD_\pi^\rho(x) = \CD_\pi(x)$ for all $x \not \in A_\rho$. It then follows from the definition of $\eps_1$ that
\begin{equ}
\Bigl|\int f(x)\CD_\pi^\rho(x)\,\mu(dx)-\int f(x)\CD_\pi^\rho(x)\,\pi(dx)\Bigr| \le {\eps_1} \bigl(1 + (2\rho)^{-1}\bigr)\;.
\end{equ}
Furthermore, 
\begin{equ}
\Bigl|\int f(x)\CD_\pi^\rho(x)\,\pi(dx)-\int f(x)\CD_\pi(x)\,\pi(dx)\Bigr| \le {\pi(A_\rho) \over 2}\;,
\end{equ}
and similarly for the term with $\pi$ replaced by $\mu$. In order to bound $\mu(A_\rho)$, we set as above
\begin{equ}
F^\rho(x) = \sup_{y \in \CY} \Bigl(\one_{A_\rho}(y) - {d(x,y)\over \rho}\Bigr) \;,
\end{equ}
so that
\begin{equ}
\mu(A_\rho) \le \int F^\rho(x) \,\mu(dx) \le {\eps_1 \over \rho} + \int F^\rho(x) \,\pi(dx) \le {\eps_1 \over \rho} + \pi(\bar A_\rho)\;,
\end{equ}
where we used the fact that $F^\rho$ vanishes outside of $\bar A_\rho$. Collecting all of these bounds, the claim follows.
\end{proof}

An alternative description of $w^\gamma$ is given by the following. Let $Y$ be a Bessel-$3$ process 
starting at the origin and let $\tau_\gamma$ be its first hitting time of $\gamma$, i.e.\ $\tau_\gamma = \inf\{t \ge 0\,:\, Y_t \ge \gamma\}$.
Let furthermore $B$ be a Brownian motion independent of $Y$, which is stopped when 
it reaches the level $\gamma$. Then, one has the decomposition
\begin{equ}[e:decompM]
w^\gamma_t = 
\left\{\begin{array}{cl}
	Y_t & \text{for $t \le \tau_\gamma$,} \\
	\gamma - B_{t-\tau_\gamma} & \text{for $t \ge \tau_\gamma$.}
\end{array}\right.
\end{equ}
This is a consequence of the symmetry of the Brownian excursion under time reversal, combined
with \cite[Theorem~49.1]{RogWil} for example.
%consequence of the general fact that a Markov process $X$, conditioned on some event $A$,
%is again a Markov process provided that $A$ is such that, for any time $t$, there exist events
%$A_t^-$ and $A_t^+$, measurable with respect to $\sigma(\{X_s\}_{s \le t})$ and $\sigma(\{X_s\}_{s \ge t})$
%respectively, such that $A = A_t^- \cap A_t^+$. Furthermore, when restricted to times less than $t$, $X$ conditioned 
%on $A$ is the same as $X$ conditioned on $A_t^-$, while for times greater than $t$ it has the same transition 
%probabilities as $X$ conditioned on $A_t^+$. These statements are still true when $t$ is replaced by a stopping time $\tau$,
%provided that $A_\tau^\pm$ is defined as
%\begin{equ}
%A_\tau^\pm \cap \{\tau = t\} = A_t^\pm \cap \{\tau = t\} \;,
%\end{equ}
%for every $t\in \R_+\cup \{+\infty\}$. See for example \cite{}.\comment{Rogers \&\ Williams}

We can use the above decomposition to obtain the following bound:
\begin{lemma}\label{lem:contMdelta}
For any $\beta < {1\over 4}$, there exists a constant $C$ such that, 
for every $\gamma,\gamma' \in (0,1]$, 
one has the bound
\begin{equ}
\|\Q_\gamma - \Q_{\gamma'}\|_d \le C |\gamma - \gamma'|^\beta\;.
\end{equ}
\end{lemma}

\begin{proof}
The decomposition \eref{e:decompM} suggests a natural coupling between $w^\gamma$ and $w^{\gamma'}$ by
building them from the same basic building blocks $Y$ and $B$. 
The characterisation of the Bessel-$3$ process as the norm of a $3$-dimensional Brownian motion, together with standard hitting estimates for
Brownian motion, imply that
\begin{equ}
\P(|\tau_\gamma - \tau_{\gamma'}| > \zeta) \le 1 \wedge {C |\gamma-\gamma'| \over \zeta^{3/2}}\;,
\end{equ}
so that in particular $\P(|\tau_\gamma - \tau_{\gamma'}| \ge \sqrt{|\gamma - \gamma'|}) \le C|\gamma - \gamma'|^{1/4}$.
The result now follows from the fact that both $B$
and $Y$ are almost surely $\alpha$-H\"older continuous for every $\alpha < {1\over 2}$.
\end{proof}

\begin{lemma}\label{lem:convBM}
Let $B^\gamma_z$ be a Brownian motion starting at $z$,
conditioned to hit $\gamma$ before $0$, and stopped upon
its return to $0$. Then, for every $\beta < 1$, there exists a constant $C$ depending on $\beta$ such that the bound
\begin{equ}
\|\CD(B^\gamma_\eps) - \Q_\gamma\|_d \le C_\gamma \eps^\beta\;,
\end{equ}
holds uniformly over $\eps \in (0,\gamma\wedge 1]$ and $\gamma > 0$.
\end{lemma}

\begin{proof}
Let $w^\gamma$  be as above and let $\tau_\eps$ be the first passage time 
of $w^\gamma$ through $\eps$. Then, it follows from the decomposition \eref{e:decompM} that one has the
exact identity 
\begin{equ}[e:idenM]
B^\gamma_\eps(\cdot) \eqlaw w^\gamma(\cdot- \tau_\eps)\;.
\end{equ}
It then follows from the small ball estimates of Brownian motion that, for every $\bar \zeta < 2$, one has the bound
\begin{equ}
\P (\tau_\eps \ge \eps^{\bar \zeta}) \lesssim \eps\;.
\end{equ}
The desired estimate then follows at once from the H\"older regularity of $w^\gamma$.
\end{proof}

\begin{proposition}\label{prop:convEpsAlpha}
Let $\alpha \in (0,{1\over 8})$.  Suppose further that $\gamma > 0$ is fixed and denote by $y^{\gamma}_t$
the  random walk $y_t$ conditioned to hit $[\gamma,\infty)$ before hitting $\R_-$ and stopped
when it then hits $\R_-$.
Assume that $y^{\gamma}_0 = \eps^\alpha$.
Then the law of $y^\gamma$ converges weakly to $\Q_\gamma$ as $\eps \to 0$.
Furthermore, for every $\beta > 0$, 
there exists a constant $C$ such that the bound
\begin{equ}
\|\CD(y^\gamma) - \Q_\gamma\|_d \le C \bigl(\sqrt \gamma \eps^{{1\over 8}-\alpha} + \eps^{\alpha-\beta}\bigr)\;,
\end{equ}
holds uniformly over all $\eps \le 1$ and $\gamma \in [\eps^\alpha,1]$.
\end{proposition}

\begin{proof}
By Lemma~\ref{lem:convBM}, it suffices to compare the law of $y^\gamma$
with that of $B^\gamma_{\eps^\alpha}$.  The result will then be a consequence of Lemma~\ref{lem:convCond}. 

To see this, we partition the state space
$\CY = \{w \in \CC([0,1],\R)\,:\, w_0 = \eps^\alpha\}$ into three sets in the following way. Given a continuous function $w$ with $w_0 \in (0,\gamma)$, we set
$\tau = 1 \wedge \inf\{t > 0\,:\, w_t \not \in [0,\gamma]\}$, and we define sets $A^{(i)}$ with $i\in \{1,2,3\}$ by
\begin{equs}[3]
{}&\mhpastefig[3/4]{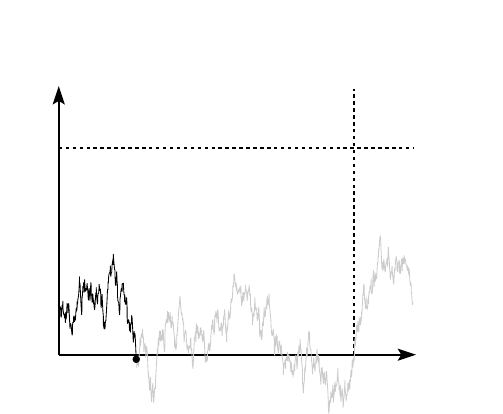}&&\mhpastefig[3/4]{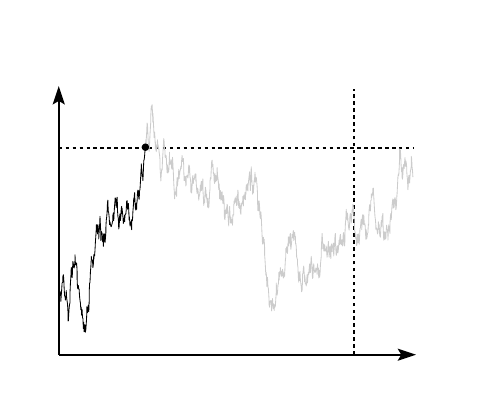}&&\mhpastefig[3/4]{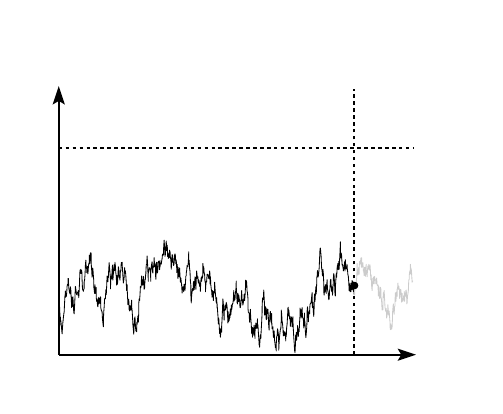} \\
&\hspace{0.6em} A^{(1)} = \{w\,:\, w_{\tau} = 0\}\;,&\quad &\hspace{0.6em} A^{(2)} = \{w\,:\, w_{\tau} = \gamma\}\;,&\quad&\hspace{0.6em} A^{(3)} = \{w\,:\, \tau = 1\}\;.
\end{equs}
Define furthermore functions $F_\gamma$ and $F_\gamma^\eps$ on $\CY$ by
\begin{equ}
F_\gamma(w) = {w_\tau \over \gamma} \;,\qquad F_\gamma^\eps(w) = 
\left\{\begin{array}{cl}
	 0 & \text{if $w \in A^{(1)}$,} \\
	 1 & \text{if $w \in A^{(2)}$,} \\
	 \bar P_{{\gamma\over \sqrt \eps},{w_1\over \sqrt \eps}} & \text{if $w \in A^{(3)}$,}
\end{array}\right.
\end{equ}
where $\bar P_{z,\gamma}$ is defined as in the discussion before Proposition \ref{prop:defG}. With these definitions at hand, if we set $\mu = \CD(y)$ with $y_0 = \eps^\alpha,$
$\pi = \CD(B_{\eps^\alpha})$, $\CD_\mu = F_{\gamma}^\eps$, and $\CD_\pi = F_\gamma$, then we are precisely in the setting of
Lemma~\ref{lem:convCond} with $\tilde \mu = \CD(y^\gamma)$ and $\tilde \pi = \CD(B_{\eps^\alpha}^\gamma)$.

Note first that, since $F_\gamma$ is precisely the probability that a Brownian motion started from $w_1$ hits $\gamma$ before $0$,
we have 
\begin{equ}
\delta = \int \CD_\pi(w) \pi(dw) = {\eps^\alpha \over \gamma}\;.
\end{equ}
Furthermore, it follows from Corollary~\ref{cor:boundRW} (and from the fact that $F_\gamma$ and $F_\gamma^\eps$ coincide outside of $A^{(3)}$) that
\begin{equ}
\eps_1 = \sup_{w} |F_\gamma(w) - F_\gamma^\eps(w)| \lesssim {\eps^{{1\over 4}} \over \gamma}\;.
\end{equ}
(We could have replaced ${1\over 4}$ by any exponent less than ${1\over 2}$ here.) Regarding the distance between the 
unconditioned measures, we obtain from \cite{ConvRateBM} the bound
\begin{equ}
\eps_2 = \|\mu - \pi\|_d \lesssim \eps^{1\over 8}\;.
\end{equ}

It thus remains to obtain a bound on $\bar A_\rho$. By the definition of $A_\rho$ and of $F_\gamma$, $w \in A_\rho$ implies that
either $w \in A^{(1)} \cup A^{(3)}$ and $d(w,A^{(2)}) \le \rho$ or $w\in A^{(1)}$ and $d(w,A^{(3)}) \le \rho$.
This implies that
\begin{equ}
A_\rho \subset \Bigl\{w\,:\, \sup_{t \in [0,1]} w_t \in [\gamma-\rho,\gamma]\Bigr\} \cup \Bigl\{w\,:\, \inf_{t \in [0,1]} w_t \in [-\rho,0]\Bigr\} \;,
\end{equ}
so that 
\begin{equ}
\bar A_\rho \subset \Bigl\{w\,:\, \sup_{t \in [0,1]} w_t \in [\gamma-2\rho,\gamma+\rho]\Bigr\} \cup \Bigl\{w\,:\, \inf_{t \in [0,1]} w_t \in [-2\rho,\rho]\Bigr\} \;.
\end{equ}
Since (by the reflection principle) the law of the extremum of Brownian motion over a finite time 
interval has a smooth density with respect to 
Lebesgue measure,
there exists a constant $C$ independent of $\eps$ and $\delta$ such that $\pi(\bar A_\rho) \le C\rho$.

Inserting these bounds into Lemma~\ref{lem:convCond}, we thus obtain the bound
\begin{equ}
\|\CD(y^\gamma) - \CD(B_{\eps^\alpha}^\gamma)\|_d \lesssim {\gamma \over\eps^\alpha} \Bigl({\eps^{1\over 4} \over \rho \gamma} + \eps^{1\over 8} + \rho\Bigr)\;.
\end{equ}  
Setting $\rho = \eps^{1\over 8} \gamma^{-{1\over 2}}$ completes the proof.
\end{proof}

We now have all the ingredients required for the proof of Theorem~\ref{theo:convLaws}.

\begin{proof}[of Theorem~\ref{theo:convLaws}] Assume as in the previous proof that $y^\gamma_t$ is the random walk $y_t$  conditioned to hit $[\gamma,\infty)$ before hitting $\R_-$ and stopped
when it then hits $\R_-$.  In contrast to the above setup we now assume that $y^\gamma_0 = z = x\sqrt \eps$ for some $x\geq 0$.
Let $k_0$ be given by
\begin{equ}
k_0 = \inf\big\{k > 0\,:\, y^\gamma_{k\eps} \ge \eps^{1\over 16}\big\}\;.
\end{equ}
It then follows from \cite{ConvRateBM}, combined with standard small ball estimates for Brownian motion that,
for every $\beta < {1\over 8}$ and every $p>0$ there exists a constant $C$ such that
\begin{equ}[e:boundk0]
\P \bigl(\eps k_0 > \eps^{\beta}\bigr) \lesssim \eps^p\;,
\end{equ}
uniformly over $\eps \le 1$, for all $x$ such that $x\sqrt \eps \le \eps^{1\over 16}$.
 Furthermore, the probability that $y^\gamma_{k_0\eps} > 2\eps^{1\over 16}$ (say)
is exponentially small in $\eps$, again uniformly over $x$. It then follows from Proposition~\ref{prop:convEpsAlpha} 
(choosing $\alpha = {1\over 16}$)
that, for every $\bar\beta < {1\over 16}$ one can construct a 
joint realisation of $y^\gamma$ and $w^\gamma$ such that
\begin{equ}
\E d\bigl(y^\gamma, w^\gamma(\cdot -\eps k_0)\bigr) \lesssim \eps^{{\bar \beta}}\;.
\end{equ}
(Here we extend $w^\gamma$ by setting it to $0$ for negative times.)
On the other hand,  the H\"older regularity of the sample paths of $w^\gamma$ together with 
\eref{e:boundk0} implies that 
\begin{equ}
\E d\bigl(w^\gamma, w^\gamma(\cdot -\eps k_0)\bigr) \lesssim \eps^{{\bar \beta}}\;,
\end{equ}
so that the bound on $\|\Q^{\eps}_{z,\gamma} - \Q_\gamma\|_d$ follows.

In order to obtain the bound on $\|\iota^\star\Q^{\eps}_{z,\gamma} - \iota^\star\Q_\gamma\|_{\bar d}$, we
make use of the same coupling between $y^\gamma$ and $w^\gamma$ as above,
so that $\E d(y^\gamma, w^\gamma) \lesssim \eps^{\bar \beta}$. 
It then remains to obtain a bound on 
\begin{equ}
|\e(\iota y^\gamma) - \e(\iota w^\gamma)| \;.
\end{equ}
For this, note first that, by Chebychev's inequality, one has the bound
\begin{equ}[e:Cheb]
\P \bigl(d(y^\gamma, w^\gamma) \ge \eps^\beta\bigr)\lesssim \eps^{\bar \beta-\beta}\;,
\end{equ}
valid for every $\beta \le \bar \beta$. Consider now any two paths $y^\gamma$ and $w^\gamma$
at distance less than $\eps^\beta$ and define 
\begin{equ}
\tau_1 = 1 \wedge \inf\{t > 0\,:\, w^\gamma(t) \le \eps^\beta\}\;,\qquad 
\tau_2 = 1 \wedge \inf\{t > \tau_1\,:\, w^\gamma(t) < -\eps^\beta\}\;.
\end{equ}
In this way, one has both $\e(\iota w^\gamma) \in [\tau_1,\tau_2]$ and 
$\e(\iota y^\gamma) \in [\tau_1,\tau_2]$, so that it remains to obtain
a bound on $\tau_2 - \tau_1$. The explicit expression for the hitting time of a line for a 
Brownian motion yields 
\begin{equ}
\P (\tau_2 - \tau_1 \ge \eps^\alpha) \le \eps^{\beta - {\alpha \over 2}}\;,
\end{equ}
for any $\alpha < 2\beta$. Choosing $\beta = {3\bar\beta\over 5}$ and $\alpha = {2\bar \beta \over 5}$
and combining this with \eref{e:Cheb}, we then obtain
\begin{equ}
\P (|\e(\iota y^\gamma) - \e(\iota w^\gamma)| \ge  \eps^{2\bar\beta \over 5})\lesssim \eps^{2\bar\beta \over 5}\;,
\end{equ}
from which the bound follows.
\end{proof}

\subsection{Convergence of recursive Poisson point processes}
\label{sec:convRPP}

The aim of this section is to provide a general result allowing us to bound the distance between
two recursive Poisson point processes of the same depth $n$ in terms of their respective kernels.
This result is the main abstract result on which the proof of the convergence result \eref{e:convFD} will 
be based. One difficulty that we have to overcome is that there is very little uniformity in the 
proximity of the kernel describing $\mu^{\eps,[n]}_\gamma$ to the one describing $\mu^{[n]}_\gamma$. 
%
%Given a map $\CQ\colon \CY \to \MM_+(\CY)$, with $\CY$ 
%some Polish space, and a ``depth'' $k$, we can build for every $x \in \CY$
%a point process in the following way.
%Define $\mu_0 = \delta_x$ and, for $k \ge 1$,
%we define $\mu_{k}$ recursively as a (conditionally independent of the $\mu_\ell$ with $\ell < k$) 
%realisation of a Poisson point process
%with intensity measure $\pi_{k} = \int_\CY \CQ(y)\,\mu_{k-1}(dy)$. 
%When then set
%\begin{equ}
%\mu^{(k)}_x = \sum_{\ell=0}^k \mu_k\;,
%\end{equ}
%and we call $\mu^{(k)}_x$ the \textit{recursive Poisson point process} of depth $k$ with kernel $\CQ$.

Throughout this section, given a Polish space $\CX$ with a distance function $d$ bounded by $1$, we define the Wasserstein-$1$
distance between any two positive measures with finite mass (and not just probability measures!) by 
\begin{equ}
\|\mu - \pi\|_1 = \sup_{\Lip f \le 1 \atop \|f\|_\infty \le 1} \Bigl(\int_\CX f(x) \mu(dx) - \int_\CX f(x) \pi(dx)\Bigr)\;.
\end{equ}
If $\mu$ and $\pi$ happen to have the same mass, then the additional constraint on the supremum norm $\|f\|_\infty$
of $f$ is redundant in the above expression, and we simply recover the usual Wasserstein-$1$ distance. In the case where the masses of $\mu$ and
$\pi$ are different however, our choice of definitions ensures that 
\begin{equ}[e:compW1]
\|\mu - \pi\|_1 \approx |\mu(\CX) - \pi(\CX)| + |\mu(\CX)|\, \Bigl\| {\mu \over \mu(\CX)} - {\pi \over \pi(\CX)}\Bigr\|_1\;,
\end{equ}
where $\approx$ denotes that both quantities are bounded by multiples of each other, with proportionality constants that
are independent of $\mu$ and $\pi$.

The main result of this section is the following:

\begin{theorem}\label{theo:convPPP}
Let $\CQ$ and $\bar \CQ \colon \CX \to \MM_+(\CX)$ be two measurable maps 
and assume that the Polish space $\CX$ is endowed with a metric
$d$ bounded by $1$.
Let $A \subset \CX$, $\eps \in (0,1]$ and $K \ge 1$ be such that the bounds
\minilab{e:ass}
\begin{align}
\sup_{x\in \CX} \CQ(x,\CX)  &\le K\;,&\qquad
%\sup_{x,y\in \CX}
 \|\CQ(x)- \CQ(y)\|_1 &\le K d(x,y)\;,\quad \label{e:ass1}\\
\sup_{x\in A} \bar\CQ(x,A^{{\text{c}}}) &\le \eps\;,&\qquad
\sup_{x\in A} \|\CQ(x) - \bar \CQ(x)\|_1 &\le \eps\;, \label{e:ass2}
\end{align}
hold, where we use the notation $A^{\text{c}} = \CX \setminus A$. 

Fix $n>0$, $\bar x \in A$ and $x \in \CX$, and denote by $\mu_x^{[n]}$ and $\bar \mu_{\bar x}^{[n]}$ the recursive
Poisson point processes with respective kernels $\CQ$ and $\bar \CQ$.
Then, there exists a coupling between $\mu_x^{[n]}$ and $\bar \mu_{\bar x}^{[n]}$ such that
\begin{equ}
\E \bigl(1 \wedge \|\mu_{x}^{[n]}-\bar \mu_{\bar x}^{[n]}\|_1\bigr) \lesssim  C \bigl(\sqrt \eps + d(x,\bar x)\bigr)\;, 
\end{equ} 
where the proportionality constant $C$ depends only on $K$ and  $n$.
\end{theorem}

\begin{remark}
One useful feature of the way that we set up the bounds in the statement is that we only require information
about the kernel $\bar \CQ$ on the set $A$. In the application we have in mind, the kernel $\bar \CQ$
will be the one describing $\tilde \mu^{\eps,[n]}_\gamma$, while the set $A$ will consist of trajectories exhibiting
``typical'' behaviour at small scales.
\end{remark}

Before we turn to the proof of this theorem, we show that if $\pi_n \to \pi$ in the Wasserstein-$1$ sense, then the  
(usual) Poisson point processes with these intensity measures also converge to each other weakly in the Wasserstein-$1$ distance:

\begin{proposition}\label{prop:distPPP}
Let $\pi$ and $\bar \pi$ be two finite positive measures on a Polish space $\CX$ endowed with a metric $d \le 1$
and let $\mu$ and $\bar \mu$ be the
corresponding Poisson point processes on $\CX$. 
Then, there exists a constant $C$ and a coupling between $\mu$ and $\bar \mu$ such that
\begin{equ}
\E \bigl(1 \wedge \|\mu - \bar \mu\|_1\bigr)  \le C \bigl(\|\pi - \bar \pi\|_1 \wedge 1\bigr)\;.
\end{equ}
\end{proposition}

\begin{proof}
The proof relies on the fact that, if $\CP(\lambda)$ denotes the Poisson distribution with parameter $\lambda$, one has the total variation
bound
\begin{equ}[e:boundTVP]
d_{\TV} \bigl(\CP(\lambda), \CP(\bar \lambda)\bigr) \le |\lambda - \bar \lambda|\wedge 1\;,
\end{equ}
see for example \cite{MR2221228}.

We can construct $\mu$ (and similarly for $\bar \mu$) in the following way. First, draw a Poisson random variable $N$ with
parameter $\pi(\CX)$. Then, draw $N$ independent random variables $\{X_1,\ldots,X_N\}$ with law $\pi / \pi(\CX)$ and
set
\begin{equ}[e:defMu]
\mu = \sum_{k=1}^N \delta_{X_k}\;.
\end{equ}
By \eref{e:boundTVP}, we can now construct a Poisson random variable $\bar N$ with parameter $\bar \pi(\CX)$ such that
\begin{equ}[e:diffN]
\P(\bar N \neq N) \le |\pi(\CX) - \bar \pi(\CX)|\;.
\end{equ}
Assuming that $\bar N = N$, we can then draw random variables $\{\bar X_1,\ldots,\bar X_N\}$ in such a way that the pairs
$(\bar X_k, X_k)$ are distributed according to a coupling between $\pi / \pi(\CX)$ and $\bar \pi / \bar\pi(\CX)$ that minimises 
their expected distance. If $\bar N \neq N$, then we simply draw $\{\bar X_1,\ldots,\bar X_N\}$ according to $\bar \pi / \bar\pi(\CX)$,
independently of the $X_k$.

If we then define $\bar \mu$ similarly to \eref{e:defMu}, it follows that
\begin{equ}
\|\mu - \bar \mu\|_1 \le 
\left\{\begin{array}{cl}
	N \bigl\| {\pi \over \pi(\CX)} - {\bar \pi \over \bar \pi(\CX)}\bigr\|_1\ & \text{if $N = \bar N$,} \\[0.4em]
	N + \bar N & \text{otherwise.}
\end{array}\right.
\end{equ}
As a consequence, we obtain the bound
\begin{equ}
\E \bigl(1 \wedge \|\mu - \bar \mu\|_1\bigr) \le \P(N \neq \bar N) + \Bigl\| {\pi \over \pi(\CX)} - {\bar \pi \over \bar \pi(\CX)}\Bigr\|_1 \E N\;,
\end{equ}
so that the claim follows from \eref{e:diffN}, the definition of $N$, and \eref{e:compW1}.
\end{proof}

\begin{proof}[of Theorem~\ref{theo:convPPP}]
Note first that, by combining the first bound in \eref{e:ass1} with the second bound in  \eref{e:ass2}, we obtain the bound
\begin{equ}
\sup_{x \in A} \bar \CQ(x,\CX) \le K+\eps\;.
\end{equ}
It follows that we have the recursive bound
\begin{equation*}
\E \bigl(\bar \mu^n_{\bar x}(A) \,|\, \bar \mu_{\bar x}^{[n-1]}(A^{\text{c}}) = 0\bigr)
\le 2K \E \bigl(\bar \mu^{n-1}_{\bar x}(A) \,|\, \bar \mu_{\bar x}^{[n-2]}(A^{\text{c}}) = 0\bigr)\;,
\end{equation*}
so that, since $\bar \mu^0_{\bar x}(A) = 1$ by assumption,
\begin{equation}\label{e:bound1}
\E \bigl(\bar \mu^n_{\bar x}(A) \,|\, \bar \mu_{\bar x}^{[n-1]}(A^{\text{c}}) = 0\bigr) \le (2K)^n\;.
\end{equation}
On the other hand, defining the  positive measures $\pi^n$ and $\bar \pi^n$ by
\[
\pi^{n+1} = \int_{\CX} \CQ(y) \mu^n_{x} (dy)\qquad \text{and}\qquad \bar \pi^{n+1} = \int_{\CX} \bar\CQ(y) \bar \mu^n_{\bar x} (dy)\]
we have the inequality
\begin{align*}
\P(\bar \mu_{\bar x}^{[n]}(A^{\text{\tiny c}}) > 0) &\le \P(\bar \mu_{\bar x}^{[n-1]}(A^{\text{c}}) > 0) + \P(\bar \mu^n_{\bar x}(A^{\text{c}}) > 0\,|\,  \bar \mu_{\bar x}^{[n-1]}(A^{\text{c}}) = 0)\\
&\le \P(\bar \mu_{\bar x}^{[n-1]}(A^{\text{c}}) > 0) + \E\bigl( \bar \pi^n(A^{\text{c}}) \,|\,  \bar \mu_{\bar x}^{[n-1]}(A^{\text{c}}) = 0\bigr)\\
&\le \P(\bar \mu_{\bar x}^{[n-1]}(A^{\text{c}}) > 0) + \eps D + \P \bigl(\bar \mu^{n-1}_{\bar x}(A) > D \,|\,  \bar \mu_{\bar x}^{[n-1]}(A^{\text{c}}) = 0\bigr)\\
&\le \P(\bar \mu_{\bar x}^{[n-1]}(A^{\text{c}}) > 0) + \eps D + {1\over D}\E \bigl(\bar \mu^{n-1}_{\bar x}(A) \,|\,  \bar \mu_{\bar x}^{[n-1]}(A^{\text{c}}) = 0\bigr)\\
&\le \P(\bar \mu_{\bar x}^{[n-1]}(A^{\text{c}}) > 0) + \eps D + {(2K)^{n-1}\over D}\;,
\end{align*}
which is valid uniformly over all $D> 0$. Choosing $D \sim 1/\sqrt\eps $, we thus obtain the recursion relation 
\begin{equation*}
\P(\bar \mu_{\bar x}^{[n]}(A^{\text{c}}) > 0) \le  \P(\bar \mu_{\bar x}^{[n-1]}(A^{\text{c}}) > 0) + C\sqrt \eps\;,
\end{equation*}
from which it follows that
\begin{equation}\label{e:bound2}
\P(\bar \mu_{\bar x}^{[n]}(A^{\text{c}}) > 0) \le C\sqrt \eps\;,
\end{equation}
where in both cases the constant $C$ depends on $K$ and $n$, but not on $\eps$.
%Note first that we have the recursive bound
%\begin{equ}
%\E \bar\mu_\ell(\CY) = \E \bar\pi_\ell(\CY) = \E \int_\CY \bar\CQ(x,\CY)\bar\mu_{\ell-1}(dx) \le K \E \bar\mu_{\ell-1}(\CY)\;,
%\end{equ}
%so that $\E \bar\mu_\ell(\CY) < K^\ell$, and similarly for $\mu_\ell$. Similarly, we have the bound
%\begin{equ}
%\E \bar\mu_\ell(A^{\text{c}}) = \E \bar\pi_\ell(A^{\text{c}}) = \E \int_\CY \bar\CQ(x,A^{\text{c}})\bar\mu_{\ell-1}(dx) \le \eps \E \bar\mu_{\ell-1}(A) + K \E \bar\mu_{\ell-1}(A^{\text{c}})\;.
%\end{equ}
%Since $\bar\mu_0(A^{\text{c}}) = 0$ by the assumption that $\bar x \in A$, we can combine these two bounds to obtain the bound
%\begin{equ}
% \E \bar\mu_\ell(A^{\text{c}}) < \eps \ell K^{\ell-1}\;.
%\end{equ}

Note now that we have the bound
\begin{align*}
\|\pi^{n+1} - \bar \pi^{n+1}\|_1 &\le \Bigl\|\int_\CX \CQ(y)\, \bigl(\mu^n_x - \bar \mu^n_{\bar x}\bigr)(dy)\Bigr\|_1
+ \int_\CX \|\CQ(y) - \bar \CQ(y)\|_1 \bar \mu^n_{\bar x}(dy) \\
&\le \|\mu^n_x - \bar \mu^n_{\bar x}\|_1 \bigl(\Lip( \CQ) + \| \CQ\|_\infty\bigr) + 
\eps \bar \mu^n_{\bar x}(A)  \\
&\quad + \int_{A^{\text{c}}} \|\CQ(y) - \bar \CQ(y)\|_1 \bar \mu^n_{\bar x}(dx) \;,
\end{align*}
so that the bound 
\begin{equation*}
1 \wedge \|\pi^{n+1} - \bar \pi^{n+1}\|_1 \le 2K \bigl(\|\mu^n_x - \bar \mu^n_{\bar x}\|_1 \wedge 1\bigr) + 
\eps D + \one_{\bar\mu^n_{\bar x}(A) > D}  + \one_{\bar\mu^n_{\bar x}(A^{\text{c}}) > 0}\;,
\end{equation*}
is valid for every $D>0$. Furthermore, by Chebyshev's inequality and \eref{e:bound1}--\eref{e:bound2}, one has
\begin{align*}
\P(\bar \mu^n_{\bar x}(A) > D) &\le \P(\bar \mu^n_{\bar x}(A) > D \,|\,  \bar\mu_{\bar x}^{n-1}(A^{\text{c}}) = 0\bigr) + \P \bigl(\bar\mu_{\bar x}^{n-1}(A^{\text{c}}) > 0\bigr)\\
&\le {K^n\over D} + C\sqrt \eps\;.
\end{align*}
Choosing again $D \sim 1/\sqrt \eps$, we thus obtain the bound
\begin{equ}
\E \bigl(1 \wedge \|\pi^{n+1} - \bar \pi^{n+1}\|_1\bigr) \le 2K \E\bigl(\|\mu^n_x - \bar \mu^n_{\bar x}\|_1 \wedge 1\bigr) + C \sqrt \eps\;.
\end{equ}

Applying Proposition~\ref{prop:distPPP}, we conclude that, given $\mu^n_x$ and $\bar \mu^n_{\bar x}$, it is 
possible to construct a coupling between $\mu^{n+1}_x$ and $\bar \mu^{n+1}_{\bar x}$ such that 
\begin{equ}
\E \bigl(1 \wedge \|\mu^{n+1}_x - \bar\mu^{n+1}_{\bar x}\|_1\bigr) \lesssim \E \bigl(1 \wedge \|\mu^{n}_x - \bar\mu^{n}_{\bar x}\|_1\bigr) +  \sqrt \eps\;,
\end{equ}
from whence the claim now follows at once.
\end{proof}

\subsection{Convergence of the truncated distributions}
\label{sec:convFD}

We are now in a position to provide the proof of \eref{e:convFD}. Again, throughout this section,
we make the standing assumption that the one-step distribution $\nu$ for the random walk \eref{e:RW}
has some exponential moment. We also use the notations $\tilde \mu^{\eps,[n]}_\gamma$ and 
$\mu^{[n]}_\gamma$ from \eref{e:convFD}.
Let  $\CE$ be the space of real-valued excursions as before.
We then introduce the map $\J^\eps \colon \CE \to \MM_+(\R^2)$ given by
\begin{equ}[e:defJ]
\J^\eps(w)(dz,dt) = a \eps\sum_{\eps k \in [0,1]} e^{a \sqrt \eps z} \, \delta_{\eps k}(dt)\, \one_{[0,\Dpe_k w]}(z)\,dz\;,
\end{equ}
where $\delta_z$ denotes the Dirac measure located at $z$, $\Dpe w_k$ is defined by
\begin{equ}
\Dpe_k w = {w_{(k+1)\eps}-w_{k\eps}\over \sqrt \eps}\;,
\end{equ}
and we used the convention that $\one_{[0,z]} = 0$ if $z < 0$.

As before, denote by $\Q_{z,\gamma}^{\eps}$ the law of the random walk \eref{e:RW}, starting at $z\sqrt \eps$, 
and conditioned to hit level $\gamma$ before becoming negative. We stop it as soon as it hits $\R_-$, so that we interpret
$\Q_{z,\gamma}^{\eps}$ as a measure on $\CE_0$. 
Recall also that $\bar P_{z,\gamma/\sqrt \eps}$, with $\bar P_{z,\gamma}$ defined as
in Section~\ref{sec:BMformal} denotes the probability that this event actually happens. 

With this notation, the measure $\CQ^\eps_\gamma$ describing $\tilde \mu^{\eps,[n]}_\gamma$ 
(i.e.\ $\CQ^\eps_\gamma(w,\cdot)$ is the intensity measure of $\QQtd(w)$) is given by
\begin{equs}
\CQ^\eps_\gamma(w, \cdot) &= \int_{\R^2} \Theta_{w,t}^\star \Q^\eps_{z,\gamma}\, \bar P_{z,{\gamma \over \sqrt \eps}} \, {\J^\eps(w)(dz,dt) \over \sqrt \eps}\\
&= \int_{\R^2} \Bigl({1\over \gamma} \Theta_{w,t}^\star \Q^\eps_{z,\gamma}\Bigr)\,\Bigl({\gamma \over \sqrt \eps} \bar P_{z,{\gamma \over \sqrt \eps}}\Bigr) \, \J^\eps(w)(dz,dt)\;.
\end{equs}
Note now that if $w$ is a typical realisation of  $\Q^\eps_{z,\gamma}$,  then $\J^\eps(w)$ is expected to be close to the measure
\begin{equ}
\J(w) = \one_{\l(w)\cap [0,1]}(t)\,dt \otimes \hat \nu(dz) \;,
\end{equ}
where $\hat \nu$ is the measure on $\R_+$ given by
\begin{equ}
\int G(z)\hat \nu(dz) = \int_0^\infty \int_0^z G(y)\,dy\,\nu(dz)\;,
\end{equ}
for any test function  $G$.
This is because $e^{a \sqrt \eps \Dpe_k w} \sim 1$ and the law of $\Dpe_k w$ would be given
by $\nu$, were it not for the conditioning. 

On the other hand, the kernel $\CQ_\gamma$ describing the truncated Brownian fan 
$\mu^{[n]}_\gamma$ with parameter $a$ is given by
\begin{equ}[e:constrQ]
\CQ_\gamma(w, \cdot) = {a\over 2\delta}\int_{\l(w)\cap [0,1]} \Theta_{w,t}^\star \Q_\gamma\, dt
= {1\over \delta}\int_{\R^2} \Theta_{w,t}^\star \Q_\gamma\,G(z)\, \J(w)(dz,dt)\;,
\end{equ}
where $\Q_\gamma$ is the law of a Brownian excursion conditioned to reach level $\gamma$ and $G$ was defined in \eqref{e:defG}. This is the case because of the well-known 
fact that ${1\over \gamma}\Q_\gamma$ is the law of the unnormalised Brownian excursion \textit{restricted} to the set of excursions 
reaching level $\gamma$. The second identity in \eref{e:constrQ} is a consequence of the definition of $\hat \nu$, combined with Proposition~\ref{prop:expG}.
With these notations at hand, we have:

\begin{proposition}\label{prop:convFD}
Consider the setting and assumptions of Theorem~\ref{theo:finalConv}.
For every $\delta < {1\over 32}$, there exists a constant $C$ depending on $\gamma$, $n$ and $\delta$ such that
\begin{equ}
\E \bigl(1 \wedge \|\tilde \mu^{\eps,[n]}_\gamma-\mu^{[n]}_\gamma\|_1\bigr) \le  C \eps^\delta\;, 
\end{equ}
uniformly over $\eps \le \eps_0$ for some $\eps_0$ small enough.
\end{proposition}

\begin{proof}
We apply Theorem~\ref{theo:convPPP} with $\CQ = \CQ_\gamma$, $\bar \CQ = \CQ^\eps_\gamma$, 
and $A$ to be determined. Using the results obtained earlier in this section, it turns out that
the assumptions are then relatively straightforward to check. 

For convenience, we introduce the notation
\begin{equ}
\tJe(w)(dz,dt) =  G(x) \J^\eps(w)(dz,dt)\;,
\end{equ}
and similarly for $\tJ$. We also denote by $\Pi_2 \colon \R^2 \to \R$ the projection onto the second component,
so that $\Pi_2^\star\tJe(w)$ (and similarly for $\tJ$) is the projection of $\tJe(w)$ onto the $t$-component.
We also denote by $\Omega_\eps \subset \CE$ the subset of excursions given by
\begin{equ}
\Omega_\eps = \Omega_\eps^{(1)} \cap \Omega_\eps^{(2)} \;,
\end{equ}
where we set
\begin{equs}
\Omega_\eps^{(1)} &= \bigl\{w\in \CE \,:\, \tJe(w)\bigl(\{x > \eps^{-1/3}\}\bigr) = 0\bigr\}\;,\\
%\Omega_\eps^{(2)} &= \bigl\{w\in \CE \,:\, \tJe(w)(\R^2) \le \eps^{-1/10}\bigr\}\;,\\
\Omega_\eps^{(2)} &= \bigl\{w\in \CE \,:\, \|\Pi_2^\star\tJ(w) - \Pi_2^\star\tJe(w)\|_1 \le \eps^{1/10} \bigr\}\;.
\end{equs}
Also, in the definition of $\Omega_\eps^{(2)}$, the Wasserstein-$1$ distance is taken with respect to the somewhat unusual distance on $\R$ given by 
\begin{equ}[e:funnyD]
d(t, t') = 1 \wedge |t-t'|^{1/3}\;.
\end{equ}
The reason for this seemingly strange choice will become clear later. The set $\Omega_\eps$ defined in this way
will play the role of $A$ when applying Theorem~\ref{theo:convPPP}.

Note first that $\CQ_\gamma(w,\CX) \le {a\over 2\gamma}$, which is bounded independently of $w$.
Furthermore, one has the bound
\begin{equs}
\bigl\|\CQ_\gamma(w,\CX) - \CQ_\gamma(\bar w,\CX)\bigr\|_1 &\le  {a\over 2\gamma} \int_{\l(w)\cap \l(\bar w)\cap [0,1]} \|\Theta_{w,t}^\star \Q_\gamma - \Theta_{\bar w,t}^\star \Q_\gamma\|_1\, dt \\ &\quad + {a\over 2\gamma} \bigl(|\s(w) - \s(\bar w)| + |\e(w) - \e(\bar w)|\bigr)\;.
\end{equs}
Since furthermore $\|\Theta_{w,t}^\star \Q_\gamma - \Theta_{\bar w,t}^\star \Q_\gamma\|_1 \le |w_t - \bar w_t|$, it does indeed
follow that $\CQ$ is globally Lipschitz continuous as required, so that \eref{e:ass1} holds with some constant
$K$ depending on $\gamma$.

It remains to check that \eref{e:ass2} holds, with $\eps$ replaced by some power of $\eps$.
By Theorem~\ref{theo:convLaws}, we already know that 
there exists a constant $C$, possibly depending on $\gamma$, such that 
\begin{equ}
\|\Q^\eps_{z,\gamma}- \Q_\gamma\|_1 \le C \eps^{\delta}\;,
\end{equ}
for any exponent $\delta \in (0,{1\over 16})$, and uniformly over all $z \le \eps^{-1/3}$.
As a consequence, for every $w \in \Omega_\eps$, we have the bound
\begin{equs}
\bigl\|\CQ_\gamma(w,\CX) -  \CQ^\eps_\gamma(w,\CX)\bigr\|_1 &\le 
{1\over \gamma} \int_{\R^2} \Bigl({\gamma\over \sqrt \eps} \bar P_{z,{\gamma \over \sqrt \eps}} - G(z)\Bigr) \, \J^\eps(w)(dz,dt) \\
&\quad + {1\over \gamma}\int_{\R^2} \|\Theta_{w,t}^\star \Q^\eps_{z,\gamma}-\Theta_{w,t}^\star \Q_\gamma\|\,\tJe(w)(dz,dt)\\
&\quad + {1\over \gamma}\Bigl\|\int_{\R^2} \Theta_{w,t}^\star \Q_\gamma \bigl(\Pi_2^\star \tJe(w)(dt)-\Pi_2^\star\tJ(w)(dt)\bigr)\Bigr\|_1\\
&\le C_\gamma \eps^{1/5} \J^\eps(w)(\R^2) + C_\gamma \eps^{\delta} \tJe(w)(\R^2) \\
&\quad + C_\gamma \eps^{1/10} \bigl(1 + \Lip\, \Theta_{w,t}^\star \Q_\gamma \bigr)\;,
\end{equs}
where $\Lip\, \Theta_{w,t}^\star \Q_\gamma$ denotes the quantity
\begin{equ}
\Lip\, \Theta_{w,t}^\star \Q_\gamma = \sup_{t\neq t'} {\|\Theta_{w,t}^\star \Q_\gamma-\Theta_{w,t}^\star \Q_\gamma\|_1\over 1 \wedge |t-t'|^{1/3}}\;.
\end{equ} 
It is at this stage that the choice \eref{e:funnyD} of distance function becomes clear: with respect to the
usual Euclidean distance, the map $t \mapsto \Theta_{w,t}^\star \Q_\gamma$ would not be Lipschitz
continuous. In this way however, it follows immediately from the H\"older continuity of Brownian motion
that $\Lip\, \Theta_{w,t}^\star \Q_\gamma < \infty$.

Furthermore, it follows from the definition of $\Omega_\eps^{(2)}$ that $\tJe(w)(\R^2)$ is bounded
 by a constant independent of $w$ and $\eps$. Since $G(x)$ is bounded from below by a constant,
 this  implies that the same is true of  $\J^\eps(w)(\R^2)$. Combining these bounds, we then obtain
 \begin{equ}
 \bigl\|\CQ_\gamma(w,\CX) -  \CQ^\eps_\gamma(w,\CX)\bigr\|_1 \le C_\gamma \eps^\delta\;,
 \end{equ}
for some constant $C$ and any $\delta < {1\over 16}$.

In order to complete the proof, it thus remains to show that 
$\inf_{w \in \Omega_\eps} \CQ^\eps_\gamma(w, \Omega^{\text{c}}_\eps) \to 0$ sufficiently fast as $\eps \to 0$, where 
$\Omega^{\text{c}}_\eps$ denotes the complement of $\Omega_\eps$.
As a consequence of the definition of $\Omega_\eps$, this follows if we can show that 
$\Q^\eps_{z,\gamma}(\Omega^{\text{c}}_\eps)\to 0$ as $\eps \to 0$, uniformly over all $z \le \eps^{-1/3}$. 
Instead of considering $\Q^\eps_{z,\gamma}$, it is much easier to consider $\Q_z^{\eps}$, the law of
the rescaled random walk \eref{e:RW} over the time interval $[0,1]$.

Observe that $\Q^\eps_{z,\gamma}$ is obtained from $\Q_z^{\eps}$ by first conditioning on the event that
$\gamma$ is reached before the walk becomes negative and then stopping the walk. By Proposition~\ref{prop:defG}, 
the probability of this event is bounded from below by $C\sqrt \eps$ for some constant $C$ depending on $\gamma$ but independent
of $z \in [0,\eps^{-1/3}]$.  Therefore it is sufficient to find a set $\Xi_\eps$ with the following two properties:
\begin{claim}
\item There exists an exponent $\zeta > {1\over 2}$ and a constant $C$ such that $\Q_z^{\eps}(\CX\setminus\Xi_\eps) \le C \eps^\zeta$
for every $\alpha \le 1$ and every $z \in [0,\eps^{-1/3}]$.
\item For every $w \in \Xi_\eps$ and every $t \in [0,1]$, the path $\tilde w$ obtained from $w$ by stopping it at time $t$
belongs to $\Omega_\eps$.
\end{claim}
In order to determine whether a path $w$ belongs to $\Xi_\eps$, we ``coarse-grain'' it into pieces of length 
$\eps^{1/3}$ (the precise exponent is not very important as we do not endeavour to obtain optimal convergence
rates)
and we impose that the contribution of $\tJe(w)$ on each piece is very close to what it should be.
In other words, setting $I_k \subset [0,1]$ by $I_k = [\eps^{1/3} k,\eps^{1/3}(k+1))$, we define $\Xi_\eps$ as
\begin{equ}
\Xi_\eps = \Omega_\eps^{(1)} \cap \Bigl\{w\,:\, \Bigl|\Pi_2^\star \tJe(w)(I_k) - {a\over 2} |I_k|\,\Bigr| \le \sqrt\eps\Bigr\}\;,
\end{equ}
where $|I_k|$ denotes the length of the interval $I_k$. 

We first note that if $w \in \Xi_\eps$ and $t \in [0,1]$, then the path $\tilde w$ obtained by stopping $w$ at
time $t$ does belong to $\Omega_\eps$. Since $\tJe(\tilde w) \le \tJe(w)$, $w\in \Omega_\eps^{(1)}$ implies
that $\tilde w \in \Omega_\eps^{(1)}$. Denote now by $\eta^\eps$ the measure $\eta^\eps = \Pi_2^\star\tJe(\tilde w)$ and by $\eta$ the target measure, namely $\eta = {a\over 2}\lambda|_{[0,t]}$, where $\lambda$ denotes
the Lebesgue measure.
By the assumption on $\tJe(w)$, it follows that one has 
$|\eta^\eps(I_k) - \eta(I_k)| \le \sqrt\eps$ for each of the intervals $I_k$, 
except possibly for the interval containing $t$. As a consequence, denoting by $\eta^\eps_k$ and
$\eta_{k}$ the restrictions of $\eta^\eps$ and $\eta$ to $I_k$, one obtains for all $k$ such that $t \not \in I_k$
the bound
\begin{equ}
\|\eta^\eps_{k} - \eta_{k}\|_1 \le \eps^{{1\over 9} + {1\over3}} + \eps^{1\over 2}\;.
\end{equ}
(Recall that we use the distance function \eref{e:funnyD}, this is the reason for the exponent $1\over 9$.)
Summing over $k$ and using the fact that $\|\eta^\eps_{k} - \eta_{k}\|_1 \le C \eps^{1/3}$ for the value
$k$ such that
$t \in I_k$, it follows that one has $\|\eta^\eps - \eta\|_1 \le C \eps^{1/9}$, which indeed implies that
$\tilde w \in \Omega_\eps$, at least for $\eps$ small enough.

It remains to show that $\Q_z^{\eps}(\CE\setminus\Xi_\eps) \le C \eps^\zeta$ for sufficiently
large $\zeta$. Observe first that $\Q_z^{\eps}(\CE\setminus \Omega_\eps^{(1)}) \le C \eps^p$ for
every $p>0$ thanks to the fact that the distribution $\nu$ of the steps of our random walk has exponential tails.

To obtain the required bound on $\Q_z^{\eps}(\CE\setminus\Xi_\eps)$, we use the fact 
that $\Pi_2^\star \tJe(w)(I_k)$ consists of the sum of $\eps^{-2/3}$
i.i.d.\ copies of a random variable $Y$ with law given by
\begin{equ}
Y \eqlaw a \eps  \int_0^Z e^{a\sqrt \eps z} G(z)\,dz\;, \qquad \CD(Z) = \nu\;,
\end{equ}
where $\nu$ is the one-step distribution of the underlying random walk.
Note now that, as a consequence of Proposition~\ref{prop:expG} and the fact that  $\nu$ admits some 
exponential moment, one has
\begin{equ}
\E Y = {a\over 2}\eps + \CO(\eps^{3/2})\;,
\end{equ}
for all $\eps$ sufficiently small. Similarly, it is straightforward to check that $\E |Y|^p = \CO(\eps^p)$
for every $p > 0$. As a consequence, one has for every $p>0$ the bound
\begin{equ}
\E \Bigl|\Pi_2^\star \tJe(w)(I_k) - {a\over 2} |I_k|\,\Bigr|^p \le  C \eps^{2p/3}\;,
\end{equ}
for a constant possibly depending on $p$. It immediately follows that
\begin{equ}
\P \Bigl(\Bigl|\Pi_2^\star \tJe(w)(I_k) - {a\over 2} |I_k|\,\Bigr| > \sqrt\eps\Bigr) \le C \eps^p\;,
\end{equ}
for every $p > 0$. Summing over $k$ and combining this with the previously obtained bound
on $\Q_z^{\eps}(\CE\setminus \Omega_\eps^{(1)})$, we conclude that  
$\Q_z^{\eps}(\CE\setminus\Xi_\eps) \le C \eps^p$ for every $p \ge 0$, which concludes our proof.
\end{proof}

\bibliographystyle{./Martin}
\markboth{\sc \refname}{\sc \refname}
\bibliography{./refs}

\end{document}